\documentclass[11pt,a4paper]{book}

\usepackage{listings}
\usepackage{algorithm}
\usepackage{algorithmic}
\usepackage[margin=2.5cm]{geometry}
\usepackage{amsmath,amssymb,amsthm,mathtools}
\usepackage{bm}
\usepackage{xcolor}
\usepackage{graphicx}
\usepackage{caption}
\usepackage{subcaption}
\usepackage{hyperref}
\usepackage{epstopdf}
\usepackage{booktabs}
\usepackage{tikz}
\usepackage{pgfplots}
\pgfplotsset{compat=1.18}
\usetikzlibrary{shadows,spy,fixedpointarithmetic,patterns,intersections,arrows,shapes,decorations.text,decorations.pathmorphing,backgrounds,fit,positioning,shapes.symbols,chains,arrows.meta,calc}
\usetikzlibrary{decorations.pathreplacing}
\usepackage{enumitem}
\usepackage{microtype}
\usepackage[most]{tcolorbox}
\usepackage{mdframed}
\usepackage{ifthen}
\usepackage[T1]{fontenc}
\usepackage[scaled=0.92]{inconsolata} 

\renewcommand{\bibname}{Bibliography}

\definecolor{TUblue}{HTML}{004C97}
\definecolor{TUeBlue}{HTML}{004C97}
\definecolor{softblue}{HTML}{EAF2FB}
\definecolor{softgreen}{HTML}{EAF8F1}
\definecolor{softyellow}{HTML}{FFF9E6}
\definecolor{softred}{HTML}{FDECEC}
\definecolor{ink}{HTML}{1F2937}

\tcbset{
  frame code={},
  center upper,
  top=4pt,bottom=6pt,left=8pt,right=8pt,
  boxsep=0pt,arc=2mm,outer arc=2mm,
  colback=white,colframe=TUblue!20!black,
}

\newtcolorbox{keyidea}[1][]{
  enhanced,unbreakable,colback=softgreen,colframe=TUblue!40!black,
  title=\textbf{Key idea}, fonttitle=\bfseries, #1
}

\newtcolorbox{takeaway}[1][]{
  enhanced,unbreakable,colback=softblue,colframe=TUblue!40!black,
  title=\textbf{Takeaways}, fonttitle=\bfseries, #1
}

\newtcolorbox{warningbox}[1][]{
  enhanced,unbreakable,colback=softred,colframe=red!50!black,
  title=\textbf{Pitfall}, fonttitle=\bfseries, #1
}

\theoremstyle{plain}
\newtheorem{theorem}{Theorem}[section]
\newtheorem{proposition}[theorem]{Proposition}
\newtheorem{corollary}[theorem]{Corollary}
\theoremstyle{definition}
\newtheorem{definition}[theorem]{Definition}
\newtheorem{example}[theorem]{Example}
\theoremstyle{remark}
\newtheorem{remark}[theorem]{Remark}

\newmdenv[
  leftline=true,
  topline=false,
  bottomline=false,
  rightline=false,
  linecolor=gray,
  linewidth=3pt,
  backgroundcolor=gray!10,
  innertopmargin=5pt,
  innerbottommargin=5pt,
  innerleftmargin=10pt,
  innerrightmargin=10pt
]{anecdote}

\newcommand{\nbpath}[1]{\texttt{#1}}   
\newcommand{\notebookpointer}[1]{%
  \par\noindent{\footnotesize\emph{Companion notebook:}~#1}\par}

\newmdenv[
  backgroundcolor=gray!10,
  linecolor=red!60,
  linewidth=1pt,
  topline=false,
  bottomline=false,
  rightline=false,
  leftline=true,
  leftmargin=0pt,
  rightmargin=0pt,
  innerleftmargin=12pt,
  innerrightmargin=12pt,
  innertopmargin=10pt,
  innerbottommargin=10pt,
  roundcorner=8pt,
  font=\itshape,
  skipabove=10pt,
  skipbelow=10pt,
  nobreak=true
]{remarkablefact}

\newif\ifhandout
\handoutfalse

\graphicspath{{figures/}}

\hypersetup{
  colorlinks=true,
  linkcolor=TUblue!70!black,
  citecolor=TUblue!55!black,
  urlcolor=TUblue!70!black,
  filecolor=TUblue!70!black,
  linktoc=all,
  pdftitle={Advanced Linear Algebra with Applications --- Part I},
  pdfauthor={Victorita Dolean and Jemima Tabeart}
}

\title{Advanced Linear Algebra with Applications --- Part I}
\author{Victorita Dolean \and Jemima M. Tabeart}
\date{}

\newcommand{\coverillustration}{%
\begin{tikzpicture}[scale=1,every node/.style={font=\small}]

  \begin{scope}[xshift=-6.6cm,yshift=0cm]
    \foreach \x in {0,...,3}{
      \foreach \y in {0,...,3}{
        \fill[TUblue!70!black] (0.5*\x,0.5*\y) circle (1.2pt);
      }
    }
    \foreach \x in {0,...,3}{ \draw[TUblue!35!white,thin] (0.5*\x,0) -- (0.5*\x,1.5); }
    \foreach \y in {0,...,3}{ \draw[TUblue!35!white,thin] (0,0.5*\y) -- (1.5,0.5*\y); }
    \draw[TUblue!80!black,thick] (0.5,0.5) -- (1.0,0.5) -- (0.5,0.5) -- (0.5,1.0) -- (0.5,0.5) -- (0.0,0.5);
    \node at (0.75,-0.45) {\footnotesize\itshape PDE discretisation};
  \end{scope}

  \begin{scope}[xshift=-2.1cm,yshift=0cm]
    \draw[TUblue!80!black,thick] (0,0) rectangle (1.5,1.5);
    \foreach \i in {0,...,5}{
      \foreach \j in {0,...,5}{
        \pgfmathtruncatemacro{\d}{abs(\i-\j)}
        \ifnum\d<2
          \fill[TUblue!70!black,opacity=0.85] (0.25*\i+0.05,1.5-0.25*\j-0.2) rectangle ++(0.18,0.18);
        \fi
      }
    }
    \node at (0.75,-0.45) {\footnotesize\itshape Sparse matrices};
  \end{scope}

  \begin{scope}[xshift=2.4cm,yshift=0cm]
    \coordinate (g1) at (0,1.3);
    \coordinate (g2) at (0.9,1.5);
    \coordinate (g3) at (1.5,0.8);
    \coordinate (g4) at (0.9,0.1);
    \coordinate (g5) at (0.1,0.4);
    \draw[TUblue!45!white,thick] (g1)--(g2)--(g3)--(g4)--(g5)--(g1);
    \draw[TUblue!45!white,thick] (g1)--(g4);
    \foreach \g in {g1,g2,g3,g4,g5}{ \fill[TUblue!80!black] (\g) circle (2.2pt); }
    \node at (0.75,-0.45) {\footnotesize\itshape Graph Laplacians};
  \end{scope}

  \begin{scope}[xshift=6.9cm,yshift=0cm]
    \draw[TUblue!80!black,thick] (0,0.75) -- (1.5,0.75);
    \draw[TUblue!70!black,thick,samples=60,domain=0:1.5,variable=\t]
      plot ({\t},{0.75+0.55*sin(2*180*\t/1.5)});
    \draw[TUblue!35!white,thick,samples=60,domain=0:1.5,variable=\t]
      plot ({\t},{0.75+0.3*sin(6*180*\t/1.5)});
    \node at (0.75,-0.45) {\footnotesize\itshape Spectral analysis in ML};
  \end{scope}

\end{tikzpicture}%
}

\begin{document}

\frontmatter

\begin{titlepage}
  \centering
  \vspace*{\fill}

  {\color{TUblue!85!black}\rule{\textwidth}{1.1pt}}\\[0.9em]
  {\Huge\bfseries\color{TUblue!85!black} Advanced Linear Algebra\\[4pt] with Applications}\\[0.6em]
  {\LARGE\itshape\color{TUblue!60!black} Part I}\\[0.9em]
  {\color{TUblue!85!black}\rule{\textwidth}{1.1pt}}\\[2.6em]

  \begin{center}
    \resizebox{0.92\textwidth}{!}{\coverillustration}
  \end{center}

  \vspace{2.8em}
  {\Large Victorita Dolean \quad\&\quad Jemima Tabeart}\\[0.5em]
  {\large\itshape Numerical linear algebra for PDEs, machine learning, and data assimilation}

  \vspace*{\fill}
  {\normalsize\color{ink} Lecture notes}\\[0.3em]
  {\small\color{ink!70} \today}

\end{titlepage}

\tableofcontents

\chapter*{Preface}
\addcontentsline{toc}{chapter}{Preface}
\markboth{Preface}{Preface}

These notes form the first part of a longer course on advanced numerical linear algebra, taught at
master's level. Their purpose is not only to present the classical algorithms, but to show why the
subject has become considerably more central than it was a generation ago.

Numerical linear algebra grew up alongside the numerical solution of partial differential equations,
and for a long time that was where its large sparse systems came from. This is no longer the whole
story. Ranking the nodes of a network, assimilating observations into a weather forecast, and fitting
a model to a large noisy data set all lead to problems of exactly the same kind: too large to
factorise, structured, and accessible only through matrix--vector products. Strikingly few ideas
are needed for all of them. Sparsity, the spectrum, Krylov subspaces and preconditioning
recur throughout, and a method designed for a discretised Laplacian is usually the right method for a
graph Laplacian or a design matrix as well.

The chapters that follow are organised around that observation. Each develops a classical topic and
then puts it to work outside its original setting: sparse matrices from finite differences, graphs and
machine learning; stationary iterations and the smoothing property; conjugate gradient and Lanczos,
with spectral clustering and regularisation by early stopping; Arnoldi and GMRES, with PageRank and
large least squares; and finally preconditioning, domain decomposition and multigrid.

This first part is deliberately foundational: factorisations and conditioning, iterative methods,
Krylov subspaces and preconditioners are what the second part of the course, not reproduced in this
volume, takes for granted throughout.

The level is that of a first-year master's student in mathematics, computer science or a
computational discipline. We assume a first course in linear algebra and some familiarity with
numerical analysis, and we prove what can be proved compactly, leaving the more technical results to
the literature. Every section closes with a summary of what should be retained, and every chapter
with exercises, several taken from past examinations.

Reading about an iterative method is, however, a poor substitute for watching one converge. The
illustrations in these notes are reproducible, and we encourage readers to run and modify the
accompanying Python code, available at \url{https://github.com/vicdolean/scicomp_examples}:
changing a parameter, a matrix or a stopping criterion is by far the fastest way to develop a
feeling for when these methods work and when they do not.

We are grateful to Rafa\l{} Wrona (Independent Researcher, Krak\'ow, Poland), whose careful reading
and explicit counterexamples led us to correct a number of errors, mainly in Chapter~3.

\paragraph{A note on the use of AI tools.}
The structure of these notes, the choice and ordering of the topics, and the mathematical content are
the authors' own. Some of the material is standard and presented here in our own way; some of it is
drawn from research we are ourselves engaged in. In preparing the manuscript we used a generative AI
assistant to help identify the seminal references on a given topic, to
improve the language, and to improve the graphical content. All statements, proofs and references
have been checked by the authors, who take full responsibility for the contents.

\vspace{0.6em}
\noindent
Victorita Dolean \quad and \quad Jemima M. Tabeart \hfill Eindhoven, \today

\mainmatter

\chapter{Solving Linear Systems: A Unified Foundation}

\section*{Overview}

Linear algebra forms the backbone of many modern scientific and engineering computations. From solving systems of equations in structural analysis, optimizing models in machine learning, and forecasting weather through data assimilation, linear algebraic methods are pervasive. However, in practical computations, the mathematical representations of real numbers must confront the limitations of computer arithmetic. Floating-point numbers can only approximate real numbers, leading to small but inevitable rounding errors, and even mathematically exact operations can be extremely sensitive to those tiny errors depending on the structure of $A$.

This chapter lays the numerical foundations used throughout the rest of the course: how to measure the size of vectors, matrices, and errors (\emph{norms}); how a problem's own sensitivity to perturbations is quantified (\emph{conditioning}); how floating-point arithmetic and algorithms can amplify or control error (\emph{stability}); and the two standard families of solvers for $Ax=b$ — direct factorizations (LU, Cholesky) and stationary iterative methods (Jacobi, Gauss--Seidel, SOR, Richardson).

\begin{remarkablefact}
\textbf{The Scale of Linear Algebra in Real Life} \\
Weather forecasts use systems with $10^9$ unknowns -- solving linear systems is essential every few hours, and must be done in minutes.\\
Google's PageRank was originally computed by solving a giant eigenvalue problem.\\
Deep Learning has a lot of linear algebra: backpropagation is mostly matrix multiplies and chain-rule gradients.
\end{remarkablefact}

\begin{tcolorbox}[colback=softblue,colframe=TUblue!40!black,title=\textbf{Learning objectives},fonttitle=\bfseries]
By the end of this chapter, you should be able to:
\begin{itemize}
    \item Define and compute common vector and matrix norms, and explain why all norms are equivalent in finite dimension.
    \item Relate eigenvalues, eigendecompositions, the Schur decomposition, and the SVD to the spectral radius and to matrix norms.
    \item Explain the sources of floating-point error (rounding, truncation, cancellation) and use the condition number $\kappa(A)$ to relate backward and forward error.
    \item Compute the $LU$ and Cholesky factorizations of a matrix, and state why partial pivoting is needed for backward stability.
    \item Derive the Jacobi, Gauss--Seidel, SOR, and Richardson iterations from a matrix splitting $A=M-N$, and state the convergence criterion $\rho(G)<1$.
\end{itemize}
\end{tcolorbox}

\section{Norms and Matrix Decompositions}

\begin{anecdote}
The first electronic computer, ENIAC (1940s), already faced severe rounding problems.
Engineers had to carefully redesign algorithms to avoid divergence even after a few thousand operations.
\end{anecdote}

\subsection{Vector and Matrix Norms}

\begin{tcolorbox}[colback=green!5!white,colframe=green!50!black]
\begin{definition}[Vector norms]
Given a vector $x \in \mathbb{R}^n$, the common norms are:
\begin{itemize}
    \item $\ell^1$ norm: $\|x\|_1 = \sum_{i=1}^n |x_i|$
    \item $\ell^2$ norm (Euclidean): $\|x\|_2 = \left( \sum_{i=1}^n |x_i|^2 \right)^{1/2}$
    \item $\ell^\infty$ norm: $\|x\|_\infty = \max_i |x_i|$
\end{itemize}
\end{definition}
\end{tcolorbox}

These norms satisfy the following properties:
\begin{itemize}
    \item Non-negativity: $\|x\| \geq 0$ and $\|x\| = 0$ if and only if $x = 0$
    \item Homogeneity: $\|\alpha x\| = |\alpha| \|x\|$ for any scalar $\alpha$
    \item Triangle inequality: $\|x + y\| \leq \|x\| + \|y\|$
\end{itemize}

Norms give us a way to measure the size (or length) of a vector. They are particularly useful in determining convergence of sequences, measuring distances between vectors, and quantifying approximation errors.

\begin{figure}[h]
\centering
\includegraphics[width=0.4\linewidth]{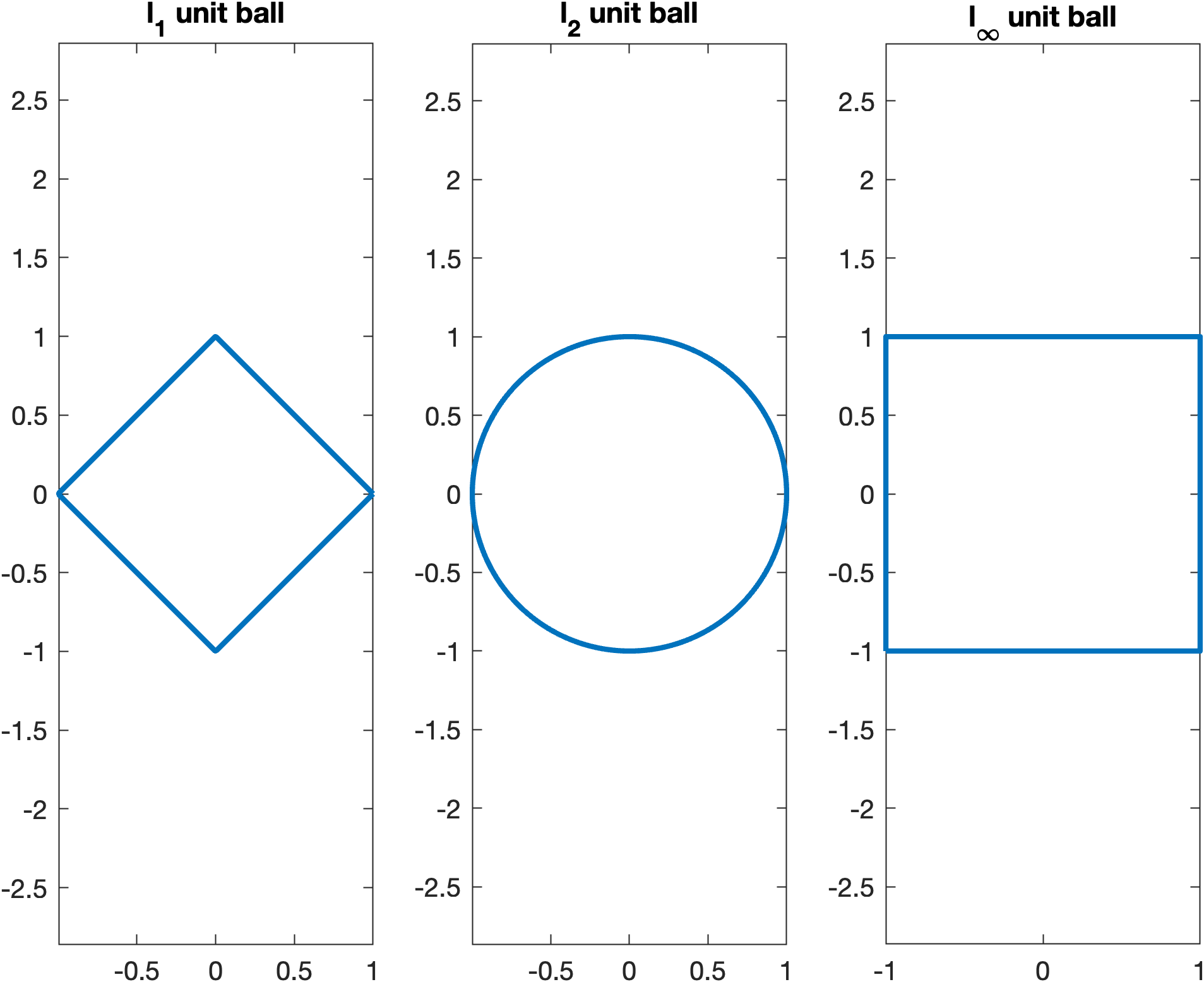}
\caption{Unit balls $\{x:\|x\|=1\}$ for the $\ell^1$, $\ell^2$, and $\ell^\infty$ norms in $\mathbb{R}^2$: the same vector has a different "size" depending on which norm is used.}
\end{figure}

\begin{tcolorbox}[colback=yellow!5!white,colframe=yellow!50!black]
\begin{proposition}[Norm equivalence]
All norms on a finite-dimensional vector space are equivalent. This means that for any two norms $\|\cdot\|_a$ and $\|\cdot\|_b$, there exist constants $c_1, c_2 > 0$ such that:
\[
    c_1\|x\|_a \leq \|x\|_b \leq c_2\|x\|_a \quad \text{for all } x \in \mathbb{R}^n.
\]
Thus, different norms can be used interchangeably when analyzing asymptotic behavior, even though their numerical values may differ. For instance, $\|x\|_\infty \le \|x\|_2 \le \sqrt{n}\|x\|_\infty$ and $\|x\|_2\le\|x\|_1\le\sqrt{n}\|x\|_2$, so here $c_1=1$ and $c_2=\sqrt n$ (see Problem~2).
\end{proposition}
\end{tcolorbox}

\subsection{Eigenvalues and Matrix Decompositions}

Before turning to matrix norms, it is useful to recall the key spectral objects that will reappear throughout the course: eigenvalues, and the two canonical matrix factorizations built from them.

\begin{tcolorbox}[colback=green!5!white,colframe=green!50!black]
\begin{definition}[Eigenvalues, eigendecomposition, spectral radius]
A scalar $\lambda\in\mathbb{C}$ and nonzero vector $x$ satisfying $Ax=\lambda x$ are an \emph{eigenvalue} and \emph{eigenvector} of $A$. If $A\in\mathbb{R}^{n\times n}$ is diagonalizable, it admits an \emph{eigendecomposition}
\[
A = V\Lambda V^{-1}, \qquad \Lambda=\operatorname{diag}(\lambda_1,\dots,\lambda_n),
\]
where the columns of $V$ are eigenvectors. The \emph{spectral radius} is $\rho(A)=\max_i|\lambda_i|$.
\end{definition}
\end{tcolorbox}

\begin{tcolorbox}[colback=green!5!white,colframe=green!50!black]
\begin{definition}[Schur decomposition and SVD]
Every $A\in\mathbb{C}^{n\times n}$ admits a \emph{Schur decomposition} $A=QTQ^*$ with $Q$ unitary ($Q^*Q=QQ^*=I$) and $T$ upper triangular, with the eigenvalues of $A$ on the diagonal of $T$. Every $A\in\mathbb{R}^{m\times n}$ (or $\mathbb{C}^{m\times n}$) admits a \emph{singular value decomposition (SVD)} $A=U\Sigma V^*$ with $U,V$ unitary and $\Sigma=\operatorname{diag}(\sigma_1,\dots)$ where $\sigma_i\ge 0$.
\end{definition}
\end{tcolorbox}

\begin{remark}
Every symmetric matrix is diagonalizable with real eigenvalues; if it is also symmetric positive definite (SPD), i.e.\ $x^\top Ax\ge 0$ with equality only for $x=0$, its eigenvalues are strictly positive. Unlike the eigendecomposition, the Schur decomposition and the SVD exist for \emph{every} matrix, which is why they underpin general-purpose numerical algorithms.
\end{remark}

\begin{takeaway}
Eigenvalues and the spectral radius $\rho(A)$ control the long-run behaviour of iterative processes (Section~\ref{sec:stationary}); the Schur decomposition and the SVD are the two factorizations that always exist and that later chapters (and the induced 2-norm below) build on.
\end{takeaway}

\subsection{Matrix Norms}

Matrix norms are critical for analyzing the stability and conditioning of algorithms and problems involving linear systems, eigenvalue computations, and optimization. Matrix norms are often induced by vector norms.

\begin{tcolorbox}[colback=green!5!white,colframe=green!50!black]
\begin{definition}[Induced norm (operator norm)]
For a matrix $A \in \mathbb{R}^{m \times n}$ the {induced norm (operator norm)} is defined by
    \[
        \|A\| = \sup_{x \neq 0} \frac{\|Ax\|}{\|x\|} = \max_{\|x\|=1} \|Ax\|
    \]
\end{definition}
\end{tcolorbox}
The three induced norms used throughout these notes are the following.
\begin{itemize}
    \item \textbf{Spectral norm (2-norm):} $\|A\|_2 = \sigma_{\max}(A)$, where $\sigma_{\max}$ is the largest singular value.
        \item \textbf{1-norm} $\|A\|_1$: maximum absolute column sum
        \item \textbf{Infinity-norm} $\|A\|_\infty$: maximum absolute row sum
\end{itemize}
Not every matrix norm is induced by a vector norm. The most important exception is the
\emph{Frobenius norm},
    \[
        \|A\|_F = \left( \sum_{i=1}^m \sum_{j=1}^n |a_{ij}|^2 \right)^{1/2}
    \]
    
    It is exactly the $\ell^2$ norm of the matrix viewed as a vector of its entries.

\begin{tcolorbox}[colback=yellow!5!white,colframe=yellow!50!black]
\begin{proposition}[Note on induced norms]
For an induced matrix norm $\|A\|$, we have:
\begin{itemize}
    \item $\|AB\| \leq \|A\| \cdot \|B\|$
    \item $\|A\| \geq \|Ax\|/\|x\|$ for all $x \neq 0$
\end{itemize}
This submultiplicative property makes induced norms useful in bounding errors and proving convergence.
\end{proposition}
\end{tcolorbox}

\begin{takeaway}
On finite-dimensional spaces all norms are equivalent: asymptotically they give the same notion of convergence, but numerically their values can differ significantly.
\end{takeaway}

\section{Floating-Point Arithmetic, Error, Conditioning, and Factorizations}

\subsection{Floating-Point Arithmetic}

\begin{anecdote}
    The Ariane 5 rocket self-destructed (1996) due to a floating-point overflow error when converting a 64-bit number to a 16-bit integer. This incident underscores the risks inherent in floating-point arithmetic and careless type conversions.
\end{anecdote}

Floating-point arithmetic is governed by the IEEE~754 standard \cite{ieee754-2019}, which fixes both
the set of representable numbers and the way operations on them are rounded.

\begin{tcolorbox}[colback=green!5!white,colframe=green!50!black]
\begin{definition}[Binary floating-point numbers]\label{def:fpnumbers}
Fix a precision $p$ and an exponent range $[e_{\min},e_{\max}]$. A \emph{normal} floating-point
number is
\[
    x = (-1)^{S}\,(1.m_1m_2\cdots m_{p-1})_2\,2^{\,e},
    \qquad
    (1.m_1\cdots m_{p-1})_2 = 1+\sum_{k=1}^{p-1} m_k 2^{-k} \in [1,2),
\]
with sign $S\in\{0,1\}$, bits $m_k\in\{0,1\}$ and $e_{\min}\le e\le e_{\max}$.
\end{definition}
\end{tcolorbox}

The leading $1$ is not stored, so $p$ bits of precision cost
$p-1$ bits of storage. The exponent is held as an unsigned $E=e+\text{bias}$,
$\text{bias}=2^{k-1}-1$ for $k$ exponent bits; the two extreme $E$ are reserved, one for zero and
the \emph{subnormals} (leading digit $0$, filling the gap around zero), the other for $\pm\infty$
and NaN. For instance $-6.5=-(1.101)_2 2^{2}$: $S=1$, $m=1010\ldots0$, $e=2$.

The first IEEE standard was introduced in 1985. For a long time single (8 exponent bits and 23-bits for the significand) and double precision (11 exponent bits and 52 significand bits) were the most commonly-used formats. However, in recent years there has been a growth in interest in reduced precisions, such as half precision (5 exponent bits and 10 significand bits), non-standard precisions such as bfloat16 (8 exponent bits and  7 significand bits) and proprietary formats like Nvidia's TensorFloat-32 (8 exponent bits and 10 significand bits). Much of this development has been driven by machine learning applications, and the desire to accelerate performance. In data assimilation, the computational cost is dominated by evaluations of the PDE model, so large gains can be made by applying the model in single precision \cite{hatfieldSinglePrecisionTangentLinearAdjoint2020}.

One option is to run an entire algorithm in a reduced precision, which will lead to increased computational efficiency, but may result in degraded performance. This has motivated `mixed-precision' approaches, particularly in the numerical linear algebra community, where different steps of an algorithm are run in different precisions. Such approaches aim to combine the computational benefits of reduced precision, while ensuring rapid convergence and computational stability (see e.g. \cite{highamMixedPrecisionAlgorithms2022,abdelfattahSurveyNumericalLinear2021} for comprehensive reviews).
 
\paragraph{The Gap Between Neighbouring Floats}
Floating-point numbers are not spread evenly along the real line. Between $2^{e}$ and $2^{e+1}$,
with $e$ the exponent of Definition~\ref{def:fpnumbers}, they are evenly spaced $2^{\,e-(p-1)}$
apart, since the significand advances one bit at a time. Because $|x|$ lies in that range, the
spacing near $x$ or its \emph{gap}, is
\[
    \mathrm{gap}(x) \;\approx\; |x| \cdot 2^{-(p-1)} ,
\]
proportional to the magnitude: doubling $x$ doubles it.  In double precision the gap just above $1$ is
$2^{-52}\approx2.2\times10^{-16}$; near $10^{16}$ it is $2$, so the representable numbers there run
$\ldots,10^{16},\,10^{16}\!+\!2,\ldots$ with nothing in between, and $10^{16}+1$ falls exactly
halfway between two neighbours.

\paragraph{Unit Roundoff}
Rounding to the nearest moves $x$ by at most \emph{half a gap}, that is by at most $|x|\cdot 2^{-p}$.
Dividing by $|x|$, the \emph{relative} error is at most $2^{-p}$ regardless of the magnitude of $x$.

\begin{tcolorbox}[colback=green!5!white,colframe=green!50!black]
\begin{definition}[Unit roundoff and machine precision]
For a precision-$p$ format, the \emph{unit roundoff} is $u:=2^{-p}$ and the \emph{machine precision}
is $\varepsilon_{\mathrm{mach}}:=2u=2^{-(p-1)}$, the gap between $1$ and the next floating-point
number. Under round-to-nearest, every real $x$ in the normal range and every basic operation
$\ast\in\{+,-,\times,/\}$ satisfy
\[
    \mathrm{fl}(x)=x(1+\delta),\qquad \mathrm{fl}(a\ast b)=(a\ast b)(1+\delta),
    \qquad |\delta|\le u .
\]
In double precision $u=2^{-53}\approx1.11\times10^{-16}$ and
$\varepsilon_{\mathrm{mach}}=2^{-52}\approx2.22\times10^{-16}$, the value reported by
\texttt{numpy.finfo(float).eps}.
\end{definition}
\end{tcolorbox}

\paragraph{Sources of Numerical Errors}
\begin{itemize}
    \item \textbf{Rounding errors:} Due to truncating or rounding numbers to fit into finite precision.
    \item \textbf{Truncation errors:} Arise when series are approximated by finite sums.
    \item \textbf{Cancellation errors:} Occur when subtracting two nearly equal numbers, leading to loss of significant digits.
\end{itemize}

\paragraph{Example: Catastrophic Cancellation}
Two beams are measured with a ruler marked in centimetres: $a=253$\,cm and $b=252$\,cm, each correct
to within half a centimetre, so each is good to about $0.2\%$. Their difference is $1$\,cm, but
the true difference could be anything between $0$ and $2$\,cm, a relative error of $100\%$.
Subtracting did not make the measurements worse; it produced an answer so small that the errors they
were already carrying are the same size as the answer itself. Floating-point arithmetic behaves the
same way, with the unit roundoff $u$ in the role of the ruler's half-centimetre: when two nearly
equal numbers are subtracted, the leading digits they agree on cancel and only the uncertain
trailing digits remain. Problem~5 shows the standard remedy, rewriting $1-\cos x$ as
$2\sin^{2}(x/2)$ so that the subtraction never happens.

\begin{takeaway}
\begin{itemize}
\item Rounding is \emph{relative}: $\mathrm{fl}(x)=x(1+\delta)$ with $|\delta|\le u$, and
$u=2^{-53}\approx10^{-16}$ in double precision.
\item The \emph{gap} between neighbouring floats is not constant: it is $\approx|x|\cdot2^{-52}$
near $x$.
\item Cancellation i.e. subtracting two nearly equal numbers, is the most dangerous source of
error: it destroys significant digits.
\item Floating-point addition is not associative, $(x+y)+z\ne x+(y+z)$, so the order of summation
matters; see Problem~6.
\item Lower precisions are faster but more exposed to overflow and rounding.
\end{itemize}
\end{takeaway}

\paragraph{Further Reading}
Higham \cite[Ch.~2]{highamAccuracyStability2002} is the standard reference; Goldberg
\cite{goldbergWhatEveryComputer1991} the readable introduction, Overton
\cite{overtonNumericalComputing2025} a short book on IEEE arithmetic, and Muller et
al.\ \cite{mullerHandbookFloatingPoint2018} the handbook.

\subsection{Error, Conditioning and Stability of algorithms}
Numerical errors in computations arise from various sources and can significantly affect the reliability of results. To understand and mitigate these effects, we analyze the nature of errors and how sensitive a problem is to perturbations.
\begin{anecdote}
    During the Gulf War (1991), the Patriot missile system failed to intercept a missile due to a small rounding error in the system's clock, leading to a timing drift of 0.34 seconds. This small error had devastating consequences, highlighting the importance of careful error analysis.
\end{anecdote}

\begin{anecdote}
    The Norwegian oil platform Sleipner A collapsed (1991) due to a small error in the finite element analysis stiffness matrix, demonstrating how critical numerical stability and accuracy are in engineering simulations.
\end{anecdote}

\begin{tcolorbox}[colback=green!5!white,colframe=green!50!black]
\begin{definition}[Forward and backward error]
If $x$ is the exact solution and $\hat{x}$ the computed one:
\[
\text{forward error}=\|\hat{x}-x\|,\qquad
\text{backward error}=\min\{\|\delta A\|:\ (A+\delta A)\hat{x}=b\}.
\]
\end{definition}
\end{tcolorbox}
\begin{tcolorbox}[colback=green!5!white,colframe=green!50!black]
\begin{definition}[Condition number]
For a nonsingular $A$,
\[
\kappa(A)=\|A\|\cdot\|A^{-1}\|.
\]
\end{definition}
\end{tcolorbox}

\begin{tcolorbox}[colback=green!5!white,colframe=green!50!black]
\begin{definition}[Condition number in the $2$-norm]
For a symmetric positive definite $A$,
\[
\kappa_2(A)=\|A\|_2\cdot\|A^{-1}\|_2 = \frac{\lambda_{max}(A)}{\lambda_{min}(A)}.
\]
\end{definition}
\end{tcolorbox}

\begin{tcolorbox}[colback=yellow!5!white,colframe=yellow!50!black]
\begin{proposition}[Error amplification]
A perturbation $\delta b$ yields
\[
\frac{\|\delta x\|}{\|x\|}\le \kappa(A)\,\frac{\|\delta b\|}{\|b\|}.
\]
\end{proposition}
\end{tcolorbox}

\begin{example}[Error amplification]
Consider
\[
    A = \begin{bmatrix} 1 & 1 \\ 1 & 1.0001 \end{bmatrix}, \quad b = \begin{bmatrix} 2 \\ 2.0001 \end{bmatrix}.
\]
This system is very close to being singular, so solving $Ax = b$ produces a result that is highly
sensitive to small perturbations in $A$ or in $b$.
\end{example}

\begin{tcolorbox}[colback=gray!10!white,colframe=gray!60!black,title=\textbf{Stability at a glance}]
\begin{itemize}
    \item \textbf{Stable algorithm:} Controls the propagation of rounding errors.
    \item \textbf{Backward stable algorithm:} Solves a nearby (perturbed) problem exactly.
    \item \textbf{Combined view:} Stability of the algorithm + conditioning of the problem together determine the final accuracy.
\end{itemize}
\end{tcolorbox}

\begin{takeaway}
\begin{itemize}
\item Small $\kappa(A)$ (close to 1): well-conditioned, stable predictions.  
\item Large $\kappa(A)$: ill-conditioned, even perfect algorithms may give unreliable results.  
 \item An algorithm can be stable (small backward error) yet produce inaccurate results if the problem is ill-conditioned (large condition number).
    \item Conversely, for ill-conditioned problems no algorithm can suppress error amplification.
\end{itemize}

\end{takeaway}

\subsection{Gaussian Elimination and LU Factorization}

\begin{anecdote}
Gaussian elimination, though often attributed to Gauss (1820), was already known in ancient China
(Chapter~8 of \emph{The Nine Chapters on the Mathematical Art}, $\sim$200~AD).
Gauss' major contribution was to interpret elimination as a matrix factorization. It took him more than twenty years to arrive at that interpretation.
\end{anecdote}

We wish to solve $Ax=b$ with $A\in\mathbb{R}^{n\times n}$, and proceed in three steps.
\begin{enumerate}
  \item Apply elimination to zero out entries below the diagonal.
  \item This produces a factorization $A=LU$ with $L$ unit lower triangular, $U$ upper triangular.
  \item Solve $Ly=b$ by forward substitution, then $Ux=y$ by backward substitution.
\end{enumerate}

\begin{tcolorbox}[colback=green!5!white,colframe=green!50!black]
\begin{definition}[LU factorization]
A matrix $A\in\mathbb{R}^{n\times n}$ admits an LU factorization if it can be written
\[
A=LU,
\]
where $L$ is unit lower triangular and $U$ is upper triangular.
\end{definition}
\end{tcolorbox}

\begin{example}[Worked $3\times 3$ LU factorization, no pivoting]
For
\[
A=\begin{bmatrix} 2 & 1 & 1 \\ 4 & -6 & 0 \\ -2 & 7 & 2 \end{bmatrix},
\]
elimination proceeds as follows:
\[
\begin{aligned}
\ell_{21}&=4/2=2, &\quad \mathrm{row}_2 &\leftarrow \mathrm{row}_2-2\,\mathrm{row}_1=[0,-8,-2];\\
\ell_{31}&=-2/2=-1, &\quad \mathrm{row}_3 &\leftarrow \mathrm{row}_3+\mathrm{row}_1=[0,8,3];\\
\ell_{32}&=8/(-8)=-1, &\quad \mathrm{row}_3 &\leftarrow \mathrm{row}_3-(-1)\,\mathrm{row}_2=[0,0,1],
\end{aligned}
\]
the second pivot being $U_{22}=-8$. This gives
\[
L=\begin{bmatrix}1&0&0\\ 2&1&0\\ -1&-1&1\end{bmatrix},\qquad
U=\begin{bmatrix}2&1&1\\ 0&-8&-2\\ 0&0&1\end{bmatrix},
\]
and one checks $A=LU$. To solve $Ax=b$, solve $Ly=b$ by forward substitution, then $Ux=y$ by backward substitution.
\end{example}

\begin{tcolorbox}[colback=yellow!5!white,colframe=yellow!50!black]
\begin{proposition}[Complexity]
The cost of Gaussian elimination (LU factorization) is $\mathcal{O}(n^3)$ flops,
while forward and backward substitution cost $\mathcal{O}(n^2)$ each.
\end{proposition}
\end{tcolorbox}

\begin{remark}
For large linear systems with $n\sim 10^6$, $\mathcal{O}(n^3)$ is infeasible.
This motivates the use of iterative solvers and preconditioners.
\end{remark}

\begin{tcolorbox}[colback=green!5!white,colframe=green!50!black]
\begin{definition}[Partial pivoting]
During elimination, rows are swapped to ensure that pivot elements are large in magnitude.
This produces a factorization
\[
PA=LU,\quad P \text{ a permutation matrix}.
\]
\end{definition}
\end{tcolorbox}

\begin{theorem}[Backward stability of GEPP]
Gaussian elimination with partial pivoting (GEPP) is backward stable: the computed $\hat{x}$ solves
\[
(A+\delta A)\hat{x}=b,\qquad \frac{\|\delta A\|}{\|A\|}=\mathcal{O}(\varepsilon_{\text{mach}}).
\]
\end{theorem}

\begin{warningbox}
Without pivoting, Gaussian elimination may fail catastrophically even for well-conditioned~$A$.
\end{warningbox}

\begin{anecdote}
\cite{bisainNewUpperBound2025} develops bounds on the growth of entries when using \textit{full} pivoting (not commonly used in practice). The appendix contains a nice overview of computational issues with partial pivoting, including in relation to new hardware such as GPUs.
\end{anecdote}

\begin{anecdote}
Gauss used elimination in 1809 to compute planetary orbits from astronomical data.
Today, $LU$ is the default “black box” in numerical linear algebra libraries (LAPACK, and hence NumPy/SciPy).
\end{anecdote}

\begin{takeaway}
\begin{itemize}
  \item LU factorization provides a direct, stable way to solve $Ax=b$.
  \item Pivoting is essential for stability.
  \item Complexity $\mathcal{O}(n^3)$ makes LU unsuitable for very large sparse linear systems.
  \item Iterative solvers (next section) are designed to overcome this scaling barrier.
\end{itemize}
\end{takeaway}

\subsection{Cholesky Factorization for SPD Matrices}

When $A$ is symmetric positive definite, LU factorization can be specialized to exploit the symmetry, halving both storage and computational cost.

\begin{tcolorbox}[colback=green!5!white,colframe=green!50!black]
\begin{definition}[Cholesky factorization]
If $A\succ 0$ is SPD, then $A=LL^\top$ for a unique lower triangular $L$ with positive diagonal entries, computed via
\[
\ell_{kk}=\sqrt{a_{kk}-\sum_{j=1}^{k-1}\ell_{kj}^2},\qquad
\ell_{ik}=\frac{a_{ik}-\sum_{j=1}^{k-1}\ell_{ij}\ell_{kj}}{\ell_{kk}},\quad i>k.
\]
\end{definition}
\end{tcolorbox}

\begin{keyidea}
Cholesky needs no pivoting and is backward stable for SPD matrices, uses half the storage of LU (only $L$ is stored), and costs $\approx n^3/3$ flops — about half the cost of a general LU factorization.
\end{keyidea}

\begin{takeaway}
\textbf{If $A$ is SPD, use Cholesky first}: it is the fastest, most memory-efficient, backward-stable direct solver available. If some $\ell_{kk}\le 0$ is encountered during the factorization, $A$ is not SPD (or is poorly scaled); switch to $LDL^\top$ with symmetric pivoting, or regularize via $A\leftarrow A+\lambda I$.
\end{takeaway}

\section{Stationary Iterative Methods}\label{sec:stationary}

Direct methods (Gaussian elimination / LU with pivoting) are robust and backward stable,
but they become impractical at scale.

\begin{itemize}
  \item \textbf{Cost.} Factorization costs $\mathcal{O}(n^3)$ flops; each triangular solve costs $\mathcal{O}(n^2)$.
        For sparse banded matrices with half-bandwidth $w$, the costs are $\mathcal{O}(nw^2)$ (factor) and $\mathcal{O}(nw)$ (solve).
  \item \textbf{Memory.} \emph{Fill-in} during elimination destroys sparsity; storage can grow from $\mathcal{O}(\mathrm{nnz}(A))$\footnote{$\mathrm{nnz}(A)$ denotes the number of non-zero elements of the matrix $A$}
        to orders of magnitude larger.
\end{itemize}

\begin{keyidea}
Iterative methods trade exactness for scalability: they approximate the solution via many
\emph{cheap} updates. One step costs $\mathcal{O}(\mathrm{nnz}(A))$, preserves sparsity, and is parallel-friendly.
\end{keyidea}

\begin{anecdote}
In the early 19th century, Gauss and Jacobi independently realized that, instead of eliminating variables,
one can successively refine a guess. Jacobi's 1845 paper on planetary orbits is often cited as the first
systematic study of iterative solvers.
\end{anecdote}

\begin{tcolorbox}[colback=green!5!white,colframe=green!50!black]
\begin{definition}[Splitting method]
Let $A=M-N$ with $M$ invertible. Define the stationary iteration
\[
x^{(k+1)} \;=\; M^{-1}N\,x^{(k)} + M^{-1}b,
\qquad
G:=M^{-1}N,\quad c:=M^{-1}b.
\]
If it converges, the limit $x^\star$ satisfies $(I-G)x^\star=c$ and hence $Ax^\star=b$.
\end{definition}
\end{tcolorbox}

\begin{tcolorbox}[colback=yellow!5!white,colframe=yellow!50!black]
\begin{theorem}[Convergence criterion]\label{th:fp}
The iteration converges for every initial guess $x^{(0)}$ if and only if $\rho(G)<1$.
Moreover, $x^{(k)}\to x^\star=(I-G)^{-1}c$ and the solution is unique.
\end{theorem}
\end{tcolorbox}

\begin{proof}[Sketch]
Let $e^{(k)}:=x^{(k)}-x^\star$. Subtract the fixed-point relation $(I-G)x^\star=c$ from the iteration to obtain
$e^{(k+1)}=G\,e^{(k)}$, hence $e^{(k)}=G^k e^{(0)}$.

\emph{($\Rightarrow$)} If $\rho(G)\ge 1$, there is an eigenpair $(\lambda,v)$ with $|\lambda|\ge 1$; choose $e^{(0)}=v$ to get
$e^{(k)}=\lambda^k v$, which does not tend to $0$. Thus convergence for \emph{all} $x^{(0)}$ implies $\rho(G)<1$.

\emph{($\Leftarrow$)} If $\rho(G)<1$, then for any $\varepsilon>0$ there exists a (subordinate) norm $\|\cdot\|_\varepsilon$ with
$\|G\|_\varepsilon \le \rho(G)+\varepsilon=:q<1$\footnote{see the theoretical problems at the end of the chapter}. Therefore
$\|e^{(k)}\|_\varepsilon \le \|G\|_\varepsilon^k \|e^{(0)}\|_\varepsilon \le q^k \|e^{(0)}\|_\varepsilon \to 0$.
Equivalently, the Neumann series $\sum_{j=0}^\infty G^j$ converges\footnote{we admit this result without proof} and $(I-G)^{-1}=\sum_{j=0}^\infty G^j$,
hence $x^\star=(I-G)^{-1}c$.
\end{proof}

\begin{corollary}[A practical sufficient condition]\label{cor:norm}
If there exists any subordinate matrix norm with $\|G\|<1$, then $x^{(k)}\to x^\star$ and
\[
\|e^{(k)}\| \;\le\; \|G\|^k \,\|e^{(0)}\|,
\qquad
\|x^\star - x^{(k)}\| \;\le\; \frac{\|G\|^k}{1-\|G\|}\,\|x^{(1)}-x^{(0)}\|.
\]
\end{corollary}

\begin{takeaway}
To design an effective stationary method:
\begin{itemize}
  \item Choose $M$ \emph{easy to invert} (triangular / block-diagonal / diagonal) so each step is cheap.
  \item Ensure $\rho(M^{-1}N)<1$ (often via diagonal dominance or SPD structure).
  \item Prefer splittings that keep $\|G\|$ comfortably below $1$ to avoid roundoff-limited plateaus.
\end{itemize}
\end{takeaway}

\subsection{Examples of Stationary Methods and their convergence}

Throughout, write
\[
A = D-L-U,
\]
with $D$ diagonal, $-L$ the strictly lower and $-U$ the strictly upper triangular part of $A$
(so $L$ and $U$ carry the minus signs of $A$ and have nonnegative entries when $A$ has nonpositive
off-diagonal entries, as in all our model problems). \textbf{This convention is used in every chapter
of these notes and in the accompanying code}; a stray sign here is the single most common source of
confusion in this subject.

\begin{itemize}
  \item \textbf{Jacobi:} $M=D$, $G_J=D^{-1}(L+U)$. Each component is updated
  using only values from the previous iterate.
  \item \textbf{Gauss--Seidel:} $M=D-L$, $G_{GS}=(D-L)^{-1}U$. New values are
  used immediately within the same sweep, leading to faster convergence.
  \item \textbf{Successive Over-Relaxation (SOR):}
  \[
  M=\tfrac{1}{\omega}D-L,\qquad
  N=\tfrac{1-\omega}{\omega}D+U,\qquad
  G_\omega=\Big(\tfrac{1}{\omega}D-L\Big)^{-1}\Big(\tfrac{1-\omega}{\omega}D+U\Big)
  =(D-\omega L)^{-1}\big((1-\omega)D+\omega U\big).
  \]
  Here $\omega>0$ is a relaxation parameter; Kahan's theorem below shows that convergence forces
  $0<\omega<2$. For $\omega=1$ this gives $M=D-L$ and $G_1=(D-L)^{-1}U$: we recover Gauss--Seidel,
  as we must.
\item \textbf{Richardson's method:} $M = \frac{1}{\alpha}I$, $G_R = I- \alpha A$ with $\alpha$ a well-chosen parameter.

\end{itemize}
\begin{remark}
While Jacobi updates all components simultaneously from the old iterate,
Gauss--Seidel reuses fresh values within the same sweep. SOR goes one step further:
each update is extrapolated by $\omega$, effectively accelerating convergence
if $\omega$ is chosen well.
\end{remark}

\begin{anecdote}
Lewis Fry Richardson also proposed the first numerical weather forecasting approach, via the solution of differential equations. His paper was published  in 1922 and is the foundation of modern numerical weather prediction. He attempted to perform computations by hand while serving with a Quaker ambulance unit on the Western Front in 1917 (his results were not good, but the method was later shown to be sound   \href{https://www.americanscientist.org/article/the-weatherman}{https://www.americanscientist.org/article/the-weatherman}).
\end{anecdote}

\begin{anecdote}
The SOR method was proposed by David M.~Young in 1950 while working on elliptic
PDEs as a PhD student at Harvard. His method was so effective that it immediately became the standard
iterative scheme of its era, especially in engineering applications.
\end{anecdote}

\begin{tcolorbox}[colback=yellow!5!white,colframe=yellow!50!black]
\begin{theorem}[Convergence under diagonal dominance]\label{th:DD-detailed}
If $A$ is strictly diagonally dominant by rows, then both Jacobi and Gauss--Seidel converge.
\end{theorem}
\end{tcolorbox}

\begin{proof}[Sketch]
With $A=D-L-U$ as above, the Jacobi iteration matrix is $G_J=D^{-1}(L+U)$. For the
$\|\cdot\|_\infty$ norm,
\[
\|G_J\|_\infty
= \max_i \sum_{j\neq i}\Big|\frac{a_{ij}}{a_{ii}}\Big|
= \max_i \frac{\sum_{j\neq i}|a_{ij}|}{|a_{ii}|}.
\]
Strict row diagonal dominance (SRDD) means $|a_{ii}|>\sum_{j\neq i}|a_{ij}|$ for all $i$, hence
$\|G_J\|_\infty<1$. Therefore $\rho(G_J)\le \|G_J\|_\infty<1$, and Jacobi converges by the fixed-point criterion.\\
For Gauss-Seidel similar arguments can be employed. An idea would be to use a \emph{weighted} $\infty$-norm $\|x\|_{\infty,w}=\max_i |x_i|/w_i$. SRDD implies the existence of $w>0$ with $\|G_{GS}\|_{\infty,w}<1$. 
\end{proof}

\begin{tcolorbox}[colback=yellow!5!white,colframe=yellow!50!black]
\begin{theorem}[Gauss--Seidel for SPD matrices,  Theorem 11.2.3 \cite{GolubvanLoan}]
If $A$ is symmetric positive definite, then Gauss--Seidel converges.
\end{theorem}
\end{tcolorbox}

\begin{remark}
Jacobi and GS are \emph{not} comparable in general: examples exist where one converges and the other does not.
\end{remark}

\begin{tcolorbox}[colback=yellow!5!white,colframe=yellow!50!black]
\begin{theorem}[Convergence SOR (Kahan)]
For any $A$ and $\omega\in\mathbb{R}$,
\[
\rho(G_\omega)\ \ge\ |\,\omega-1\,|.
\]
Hence a \emph{necessary} condition for convergence is $0<\omega<2$.
\end{theorem}
\end{tcolorbox}
\begin{proof}
    See \cite[Theorem 6.32]{axelssonIterativeSolutionMethods1994} 
\end{proof}

\begin{tcolorbox}[colback=yellow!5!white,colframe=yellow!50!black]
\begin{theorem}[Convergence SOR (Ostrowski \cite{ostrowski1954linear})]
Let $A$ be an SPD matrix, so that in our convention $A=D-L-L^\top$ with $D$ diagonal and positive.
Then SOR converges for \textbf{all} $0<\omega<2$.
\end{theorem}
\end{tcolorbox}

\begin{tcolorbox}[colback=yellow!5!white,colframe=yellow!50!black]
\begin{theorem}[Convergence rate SOR (Young \cite{youngConvergencePropertiesSymmetric1970})]
Assume that $A$ has property~A and is consistently ordered (as the 1D Poisson matrix is, in the
natural ordering), that the Jacobi eigenvalues $\mu(G_J)$ are real, and that $\rho(G_J)<1$. Then the
\emph{optimal} relaxation parameter in the SOR method is
\[
\boxed{\ \omega^\star=\frac{2}{\,1+\sqrt{\,1-\rho(G_J)^2\,}}\ }\qquad(0<\omega^\star<2),
\]
and the optimized SOR factor is
\[
\rho(G_{\omega^\star})=\left(\frac{\rho(G_J)}{1+\sqrt{\,1-\rho(G_J)^2\,}}\right)^{\!2}.
\]
\end{theorem}
\end{tcolorbox}

\begin{tcolorbox}[colback=yellow!5!white,colframe=yellow!50!black]
\begin{theorem}[Convergence of Richardson's method for SPD matrices \cite{youngRichardsonsMethodSolving1953}]
Let $A$ a SPD matrix. Then:
\begin{itemize}
\item Convergence \textbf{iff}\ $0<\alpha<\dfrac{2}{\rho(A)}=\dfrac{2}{\lambda_{\max}(A)}$.
\item The optimal step is $\displaystyle \alpha^\star=\frac{2}{\lambda_{\max}(A)+\lambda_{\min}(A)}$.
\item The optimized factor is $\displaystyle \rho(I-\alpha^\star A)=\frac{\kappa(A)-1}{\kappa(A)+1}$ with $\kappa(A)=\lambda_{\max}/\lambda_{\min}$.
\end{itemize}
\end{theorem}
\end{tcolorbox}

\begin{example}[Diagonally dominant system]
The tridiagonal matrix
\[
A=\begin{bmatrix}
4 & -1 & 0 & \cdots \\
-1 & 4 & -1 & \ddots \\
0 & -1 & 4 & \ddots \\
\vdots & \ddots & \ddots & \ddots
\end{bmatrix}
\]
is strictly diagonally dominant. Jacobi and Gauss--Seidel are guaranteed to converge.
\end{example}

\begin{example}[Ill-behaved system]
The Hilbert matrix $H_{ij}=1/(i+j-1)$ is symmetric positive definite,
but extremely ill-conditioned. Gauss--Seidel converges in theory,
but in finite precision it stagnates quickly due to roundoff.
\end{example}

\begin{anecdote}
It is now widely recognised that Jacobi and Gauss--Seidel should not be judged solely by their speed
as solvers. Their true importance lies in their role as building blocks for the preconditioned methods
widely used in modern scientific computing.
\end{anecdote}

\begin{takeaway}
\begin{itemize}
\item \textbf{Existence of a good splitting} and $\rho(G)<1$ are the universal convergence keys.
\item \textbf{Jacobi/GS}: guaranteed under SRDD; GS also for all SPD $A$.
\item \textbf{SOR}: must take $0<\omega<2$; for SPD, convergence holds for every such $\omega$; if Jacobi’s $\rho(G_J)$ is known/estimated, Young’s formula gives a near-optimal $\omega$.
\item \textbf{Richardson}: simple baseline; optimal $\alpha^\star$ gives factor $(\kappa-1)/(\kappa+1)$.
  \item Though obsolete as standalone solvers for large systems, Jacobi, GS, and SOR
        remain indispensable as \emph{smoothers} in multigrid and as \emph{preconditioners}
        in Krylov methods.
\end{itemize}
\end{takeaway}

\newpage
\section{Exercises}

\paragraph{Problem 1 (Condition number comparison).}
Compute the condition number of
\begin{itemize}
\item the identity matrix $I$,
\item a diagonal matrix with entries ranging from $1$ to $10^5$,
\item a $5 \times 5$ Hilbert matrix.
\end{itemize}
Discuss what these condition numbers tell you.

\notebookpointer{\nbpath{week1/script\_week1.ipynb}}
\paragraph{Problem 2 (Vector norm equivalence).}
Prove that for any vector $x \in \mathbb{R}^n$,
\[
\|x\|_\infty \leq \|x\|_2 \leq \sqrt{n} \|x\|_\infty,
\]
in other words that these norms are equivalent.

\paragraph{Problem 3 (Properties of the Frobenius norm, exam 2024/2025).}
Let $A\in\mathbb{R}^{n\times n}$ and let $U,V\in\mathbb{R}^{n\times n}$ be unitary (orthogonal) matrices, i.e.,
$U^\top U=UU^\top=I$ and $V^\top V=VV^\top=I$. The Frobenius norm is
\[
\|A\|_F := \Bigg(\sum_{i=1}^n\sum_{j=1}^n a_{ij}^2\Bigg)^{1/2}.
\]
The trace of $A$ is $\mathrm{trace}(A)=\sum_{i=1}^n a_{ii}$.
\begin{enumerate}
\item[(a)] Show that the Frobenius norm is a matrix norm, but it is \emph{not} induced by any vector norm.
\item[(b)] Show that $\|A\|_F^2=\mathrm{trace}(A^\top A)$.
\item[(c)] Show that unitary transformations leave the Frobenius norm unchanged:
$\|UAV\|_F=\|A\|_F$.
\end{enumerate}

\paragraph{Problem 4 (Induced 2-norm, resit exam 2024/2025).}
\begin{enumerate}
    \item[(a)] Let $A$ be an $n \times n$ matrix. Define the induced 2-norm as
    \[ \|A\|_2 = \sup_{\|x\|_2=1} \|Ax\|_2. \]
    Prove that the induced 2-norm satisfies the properties of a matrix norm.
    \item[(b)] Show that the induced 2-norm of a diagonal matrix is equal to the maximum absolute value of its diagonal entries.
    \item[(c)] Let $A \in \mathbb{R}^{n \times n}$ and $U, V$ be unitary matrices of appropriate dimensions. Prove that $\|UAV\|_2 = \|A\|_2$.
\end{enumerate}

\paragraph{Problem 5 (Catastrophic cancellation).}
Let $a = 1.0000000001$, $b = 1.0000000$ and assume we work in double precision.
\begin{itemize}
\item Compute $a - b$ and compare to the expected result.
\item Explain what happens if these numbers are rounded to fewer significant digits (single precision).
\end{itemize}

\notebookpointer{\nbpath{week1/script\_week1.ipynb}}
\paragraph{Problem 6 (Floating-point summation order).}
Design an experiment to show that the order of operations in floating-point summation matters. Sum a list of numbers in increasing and decreasing order and compare the results.

\notebookpointer{\nbpath{week1/script\_week1.ipynb}}
\paragraph{Problem 7 (Effect of conditioning).}
Solve the systems $Ax = b$ for two matrices:
\[
    A_1 = \begin{bmatrix} 1 & 2 \\ 3 & 4 \end{bmatrix}, \quad
    A_2 = \begin{bmatrix} 1 & 1 \\ 1 & 1.0001 \end{bmatrix}, \quad
    b = \begin{bmatrix} 5 \\ 11 \end{bmatrix}
\]
Compare the computed solutions and relate the error to the condition number of $A$.

\notebookpointer{\nbpath{week1/script\_week1.ipynb}}

\paragraph{Problem 8 (Approximating the spectral radius by an induced norm).}
\emph{For any $A\in\mathbb{C}^{n\times n}$ and any $\varepsilon>0$, show there exists a (vector) norm
on $\mathbb{C}^n$ such that the induced matrix norm satisfies}
\[
\rho(A)\ \le\ \|A\|\ \le\ \rho(A)+\varepsilon.
\]

\paragraph{Problem 9 (Gelfand's formula for induced norms).}
\emph{For any induced matrix norm and any $G\in\mathbb{R}^{n\times n}$, prove}
\[
\lim_{k\to\infty} \|G^k\|^{1/k} = \rho(G).
\]

\paragraph{Problem 10 (Neumann series).}
\emph{Let $B\in\mathbb{C}^{n\times n}$ with $\rho(B)<1$. Prove that $I-B$ is invertible and}
\[
(I-B)^{-1}=\sum_{j=0}^{\infty} B^j.
\]

\paragraph{Problem 11 (Basic iterative methods, resit exam 2024/2025).}
Consider the linear system $Ax = b$, where
\[ A = \begin{bmatrix} 4 & -1 & 0 \\ -1 & 4 & -1 \\ 0 & -1 & 4 \end{bmatrix}, \quad b = \begin{bmatrix} 2 \\ 6 \\ 2 \end{bmatrix}. \]
\begin{enumerate}
    \item[(a)] Derive the iteration formula for the Jacobi method for solving this system.
    \item[(b)] Perform two iterations of the Jacobi method starting from $x^{(0)} = [0, 0, 0]^T$.
    \item[(c)] Prove that the iteration matrix $G$ for the Jacobi method satisfies $\rho(G) < 1$, ensuring convergence of the method.
\end{enumerate}

\paragraph{Problem 12 (Stationary methods applied to a SOAR correlation matrix).}
A second-order autoregressive (SOAR) correlation matrix, as used for spatial correlations in data assimilation, is defined in Chapter~2. Implement Jacobi, Gauss--Seidel, and SOR for a SOAR matrix with $n=100$, $L=0.4$, $a=1$. How do the numerical results connect to the theory?
\notebookpointer{\nbpath{week1/script\_week1.ipynb}}

\chapter{Discretization of 1D and 2D PDEs (and Beyond)}

\section*{Overview}

Chapter 1 built the numerical foundations (norms, conditioning, direct and stationary iterative solvers) for $Ax=b$. In this chapter we will see where the matrix $A$ actually comes from, and why it is almost always \emph{sparse}: the discretization of a PDE by finite differences, the Laplacian of a graph, and the normal equations of a machine learning model all produce matrices whose nonzero pattern is dictated by \emph{local structure} - a stencil, an edge list, or a feature's support. This chapter covers: (1) from Taylor expansions to finite-difference formulas and the 1D tridiagonal Poisson system; (2) the 2D Poisson matrix, graph Laplacians, and sparsity in machine learning; (3) a toolbox for working with sparse matrices in practice - storage formats, bandwidth/fill-in, reordering, and the spectral properties that govern solver convergence.

\begin{tcolorbox}[colback=softblue,colframe=TUeBlue!40!black,title=\textbf{Learning objectives},fonttitle=\bfseries]
By the end of this chapter, you should be able to:
\begin{itemize}
    \item Derive finite-difference formulas from Taylor expansions and state their consistency order (truncation error).
    \item Assemble the tridiagonal (1D) and block-tridiagonal/Kronecker (2D) linear systems that arise from discretizing the Poisson equation.
    \item Define the graph Laplacian $L=D-A$ and relate its sparsity pattern and spectrum to those of PDE and machine-learning matrices.
    \item Explain why sparsity appears in ML (normal equations, Hessians) and data assimilation (localized correlation matrices).
    \item Choose an appropriate sparse storage format (CSR/CSC, banded), and explain how half-bandwidth, profile, and reordering affect fill-in and solver cost.
    \item Relate the spectral properties (eigenvalues, condition number) of a sparse matrix to the convergence of iterative solvers and the need for preconditioning.
\end{itemize}
\end{tcolorbox}

\section{Finite difference approximation}
In Chapter 1 we discussed the importance of numerical solvers for differential equations. Some differential equations can be solved analytically (e.g. separable partial differential equations, or using the method of characteristics), giving solution in closed form. However, many PDEs of interest for applications cannot be solved analytically. In this case, we solve the problem approximately by discretising the continuous problem, which can then be solved on a computer. Three of the most common approaches include the finite element method, finite volume method, and finite difference method. In this course, we focus on finite difference approximations and consider how the properties of the resulting matrices play a role in the convergence of iterative methods.

\begin{anecdote}
Long before computers, finite differences were already a powerful idea. Euler used them to tabulate functions, and a century later Richardson boldly extended the same spirit to the whole atmosphere. In his 1922 Weather Prediction by Numerical Process, he replaced derivatives by local difference operators to turn the equations of fluid dynamics into linear systems, attempting a six-hour forecast by hand. The calculation took him weeks and produced unstable and incorrect results, but his dream came back to life again with the advent of parallel computers \cite{youngRichardsonsMethodSolving1953}.
\end{anecdote}

Let $u:\mathbb{R}\to\mathbb{R}$ be smooth and $h>0$. A Taylor expansion around $x$ gives the following result.
\[
u(x\pm h)=u(x)\pm h u'(x) + \tfrac{h^2}{2}u''(x)\pm \tfrac{h^3}{6}u^{(3)}(x)+ \tfrac{h^4}{24}u^{(4)}(\xi_\pm),
\]
for some $\xi_\pm$ between $x$ and $x\pm h$. From this we obtain the following.

\begin{tcolorbox}
[colback=green!5!white,colframe=green!50!black]
\begin{definition}[Finite difference operators]
\begin{align*}
\text{Forward:}\quad &\frac{u(x+h)-u(x)}{h} \;=\; u'(x) \;+\; \tfrac{h}{2}u''(x) + \mathcal{O}(h^2),\\[2pt]
\text{Backward:}\quad &\frac{u(x)-u(x-h)}{h} \;=\; u'(x) \;-\; \tfrac{h}{2}u''(x) + \mathcal{O}(h^2),\\[2pt]
\text{Central:}\quad &\frac{u(x+h)-u(x-h)}{2h} \;=\; u'(x) \;+\; \mathcal{O}(h^2),\\[4pt]
\text{Second derivative:}\quad &\frac{u(x+h)-2u(x)+u(x-h)}{h^2} \;=\; u''(x) \;+\; \mathcal{O}(h^2).
\end{align*}
\end{definition}
\end{tcolorbox}

 \begin{figure}
\begin{center}
\begin{tikzpicture}[scale=0.95]
  \draw[->] (-0.2,0) -- (6.2,0) node[below] {$x$};
  \draw[domain=0.2:6, smooth, samples=120, TUeBlue, thick]
    plot (\x,{0.7 + 0.5*sin(0.9*\x r) + 0.08*\x}) node[right] {$u(x)$};
  \def\xa{2} \def\xm{3} \def\xb{4}
  \draw (\xa,0) -- (\xa,0.12) node[below] {$x-h$};
  \draw (\xm,0) -- (\xm,0.12) node[below] {$x\vphantom{h}$};
  \draw (\xb,0) -- (\xb,0.12) node[below] {$x+h$};
  \fill (\xa,{0.7 + 0.5*sin(0.9*\xa r) + 0.08*\xa}) circle (1.3pt);
  \fill (\xm,{0.7 + 0.5*sin(0.9*\xm r) + 0.08*\xm}) circle (1.3pt);
  \fill (\xb,{0.7 + 0.5*sin(0.9*\xb r) + 0.08*\xb}) circle (1.3pt);
  \node[gray] at (\xm, -0.55) {\small local Taylor expansion};
\end{tikzpicture}
\caption{Zoom-in of $u(x)$ around a point with neighbors $x\pm h$}
\end{center}
\end{figure}
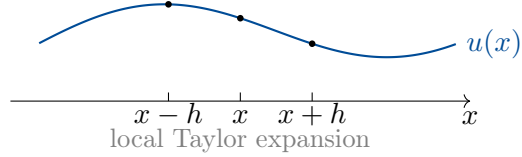
\begin{tcolorbox}
[colback=green!5!white,colframe=green!50!black]
\begin{definition}[Consistency and truncation error]
A finite difference operator $\mathcal{L}_h$ is \emph{consistent} with a differential operator $\mathcal{L}$ if
$(\mathcal{L}_h u)(x)=(\mathcal{L} u)(x)+\tau_h(x)$ with $\|\tau_h\|\to 0$ as $h\to 0$.
The term $\tau_h$ is the \emph{local truncation error}; its order ($\mathcal{O}(h)$, $\mathcal{O}(h^2)$,\dots) is the \emph{accuracy}.
\end{definition}
\end{tcolorbox}

\begin{tcolorbox}[colback=yellow!5!white,colframe=yellow!50!black]
\begin{proposition}[Consistency and order of the three-point stencil]
Let $u\in C^{4}$ in a neighborhood of $x\in\mathbb{R}$. Then
\[
D_{xx}^h u(x) :=\frac{u(x+h)-2u(x)+u(x-h)}{h^{2}}
= u''(x) + \frac{h^{2}}{12}\,u^{(4)}(\xi)
\quad\text{for some }\xi\in(x-h,x+h).
\]
In particular, the finite difference operator \(
D_{xx}^h u(x) \)
is consistent with $u''(x)$ and has truncation error $\mathcal{O}(h^{2})$ as $h\to 0$.
\end{proposition}
\end{tcolorbox}

\begin{proof}
By Taylor’s theorem with Lagrange remainder, for some $\xi_{+}\in(x,x+h)$ and
$\xi_{-}\in(x-h,x)$,
\[
u(x\pm h)
= u(x) \pm h u'(x) + \frac{h^{2}}{2}u''(x)
\pm \frac{h^{3}}{6}u^{(3)}(x)
+ \frac{h^{4}}{24}u^{(4)}(\xi_{\pm}).
\]
Adding the $+$ and $-$ expansions and subtracting $2u(x)$ yields
\[
u(x+h)-2u(x)+u(x-h)
= h^{2}u''(x) + \frac{h^{4}}{24}\big(u^{(4)}(\xi_{+})+u^{(4)}(\xi_{-})\big).
\]
Dividing by $h^{2}$ gives
\[
\frac{u(x+h)-2u(x)+u(x-h)}{h^{2}}
= u''(x) + \frac{h^{2}}{24}\big(u^{(4)}(\xi_{+})+u^{(4)}(\xi_{-})\big).
\]
By the intermediate value theorem, there exists $\xi\in(x-h,x+h)$ such that
$\tfrac12\big(u^{(4)}(\xi_{+})+u^{(4)}(\xi_{-})\big)=u^{(4)}(\xi)$, hence
\[
\frac{u(x+h)-2u(x)+u(x-h)}{h^{2}}
= u''(x) + \frac{h^{2}}{12}\,u^{(4)}(\xi).
\]
Therefore
\[
\Bigl|D_{xx}^h u(x)-u''(x)\Bigr|
\le \frac{h^{2}}{12}\max_{|y-x|\le h}|u^{(4)}(y)|
= \mathcal{O}(h^{2}),
\]
which proves consistency with second-order accuracy.
\end{proof}

\begin{remark}
Central differences are second order accurate for $u'$ and $u''$, while forward/backward are first order for $u'$.
Higher order formulas follow by eliminating higher derivatives from longer Taylor stencils.
\end{remark}

Consider now the boundary value problem (BVP)
\[
-\,u''(x)=f(x)\quad\text{for }x\in(0,1),\qquad u(0)=\alpha,\;u(1)=\beta.
\]
Let $x_i=ih$ with $h=1/(n+1)$ and unknowns $u_i\approx u(x_i)$ for $i=1,\dots,n$.
Approximating $u''(x_i)$ with the central difference gives
\[
-\,\frac{u_{i-1}-2u_i+u_{i+1}}{h^2} \;=\; f(x_i),\qquad i=1,\dots,n,
\]
where $u_0=\alpha$ and $u_{n+1}=\beta$ are known from the Dirichlet boundary conditions.

\begin{figure}
\begin{center}
\begin{subfigure}[t]{0.48\textwidth}
\begin{tikzpicture}[scale=0.95]
  \draw[->] (0,0) -- (6.2,0) node[below] {$x$};
  \foreach \x/\lbl in {0/$\alpha$,1/$u_1$,2/$u_2$,3/$\cdots$,4/$u_{n-1}$,5/$u_n$,6/$\beta$}{
    \draw (\x,0) -- (\x,0.12);
    \fill (\x,0.12) circle (1.1pt) node[above,yshift=1pt] {\scriptsize \lbl};
  }
\end{tikzpicture}
\end{subfigure}
\begin{subfigure}[t]{0.48\textwidth}
\begin{tikzpicture}[scale=1.0]
  \draw[->] (-0.5,0) -- (6.2,0) node[below] {$x$};
  \foreach \x/\lab in {1/{$i-1$},3/{$i$},5/{$i+1$}}{
    \draw (\x,0) -- (\x,0.12);
    \fill (\x,0.12) circle (1.3pt) node[above] {\scriptsize \lab};
  }
  \node[TUeBlue] at (1,0.7) {\Large $-1$};
  \node[TUeBlue] at (3,1.0) {\Large $+2$};
  \node[TUeBlue] at (5,0.7) {\Large $-1$};
  \draw[decorate, decoration={brace, amplitude=6pt}] (0.7,1.3) -- (5.3,1.3)
    node[midway, yshift=10pt] {\small $[-1,\;2,\;-1]/h^2$};
      \draw[decorate, decoration={brace, amplitude=3pt},TUeBlue] (1.0,0.1) -- (3.0,0.1)
    node[midway, yshift=10pt,TUeBlue] {\small $h$};
\end{tikzpicture}
\end{subfigure}
\caption{Left: Function values on discretised mesh. Right: Stencil for central differences in 1D}
\end{center}
\end{figure}
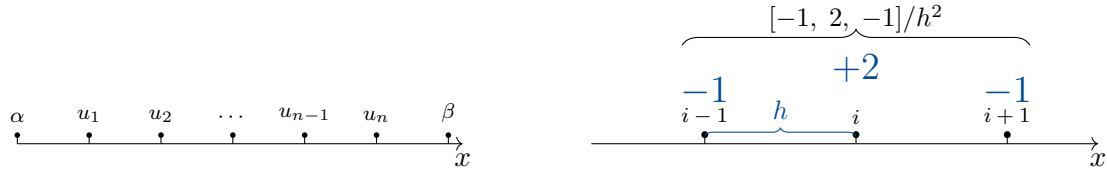

\begin{tcolorbox}[colback=yellow!5!white,colframe=yellow!50!black]
\begin{proposition}[Tridiagonal linear system]
Collecting the $n$ interior equations yields
\[
A\,\bm{u}=\bm{g},\qquad
A=\frac{1}{h^2}\,\mathrm{tridiag}(-1,\,2,\,-1)\in\mathbb{R}^{n\times n},
\]
with $\bm{g}=(f(x_i))_{i=1}^n$ modified at the endpoints:
$g_1\leftarrow g_1 + \alpha/h^2$, \ $g_n\leftarrow g_n + \beta/h^2$.
\end{proposition}
\end{tcolorbox}

\paragraph{Explicit form of the matrix.}
For $n=5$, the system has the form
\[
\frac{1}{h^2}
\begin{bmatrix}
2 & -1 & 0 & 0 & 0 \\
-1 & 2 & -1 & 0 & 0 \\
0 & -1 & 2 & -1 & 0 \\
0 & 0 & -1 & 2 & -1 \\
0 & 0 & 0 & -1 & 2
\end{bmatrix}
\begin{bmatrix} u_1\\u_2\\u_3\\u_4\\u_5 \end{bmatrix}
=
\begin{bmatrix} g_1\\g_2\\g_3\\g_4\\g_5 \end{bmatrix}.
\]
The sparsity pattern corresponds exactly to the stencil \(\boxed{[-1,\; 2,\; -1]/h^2}.\)

\begin{takeaway}
\begin{itemize}
  \item Finite difference discretisation converts differential operators into sparse matrices whose
  nonzero pattern is determined by the \emph{stencil}.
  \item The stencil is therefore the object to look at: it fixes the bandwidth, and with it the cost
  of every direct method applied to the resulting system.
\end{itemize}
\end{takeaway}

\section{Sparse Matrices: PDEs, Graphs, and Machine Learning}

\begin{anecdote}
In 1928, Courant, Friedrichs, and Lewy introduced the finite difference method for PDEs in their seminal work. At the same time, scientists like Kirchhoff (1847) were studying electrical networks using adjacency matrices and Laplacians.  Nearly a century later, sparse matrices from both PDEs and graphs lie at the core of  modern scientific computing and machine learning.
\end{anecdote}

\subsection{The 2D Poisson equation: structure, stencils and blocks}\label{sec:2DPoisson}

Consider the Poisson problem on the unit square
\[
- \Delta u(x,y) = f(x,y), \qquad (x,y)\in (0,1)^2, 
\quad u|_{\partial\Omega}=0.
\]
Let $x_i=ih$, $y_j=jh$, with $h=\tfrac{1}{n+1}$ and interior indices $i,j=1,\dots,n$.
Denote the grid unknowns by $u_{i,j}\approx u(x_i,y_j)$.
We discretize the two second derivatives separately:
\[
u_{xx}(x_i,y_j)\ \approx\ \frac{u_{i-1,j}-2u_{i,j}+u_{i+1,j}}{h^2},
\qquad
u_{yy}(x_i,y_j)\ \approx\ \frac{u_{i,j-1}-2u_{i,j}+u_{i,j+1}}{h^2}.
\]
Thus
\[
-\Delta u(x_i,y_j)\ \approx \ \frac{4u_{i,j}-u_{i-1,j}-u_{i+1,j}-u_{i,j-1}-u_{i,j+1}}{h^2}.
\]

\begin{figure}
    \centering
\begin{tikzpicture}[scale=0.9]
  \draw[thick] (-8,-1) rectangle (-3,4);
  \foreach \x in {-7.5,-7.0,...,-3.5} \draw[gray!70] (\x,-1) -- (\x,4);
  \foreach \y in {-0.5,0.0,...,3.5} \draw[gray!70] (-8,\y) -- (-3,\y);
  \fill[TUeBlue] (-5.5,2.5) circle (1.4pt) node[above right] {\scriptsize $(i,j)$};

  \def\S{1.6}   
  \def\dot{2pt} 
  \def\dy{10pt} 
\draw[dashed,TUeBlue] (-5.5,2.5) -- (0*\S,0*\S);
\draw[dashed,TUeBlue] (-5.5,2.5) -- (0*\S,2*\S);
  \path[use as bounding box] (0.2*\S,-0.8*\S) rectangle (3.2*\S,3.8*\S);

  \foreach \x in {0,1}{
    \foreach \y in {0,1}{
      \fill[gray!15] (\x*\S,\y*\S) rectangle ++(\S,\S);
    }
  }

  \foreach \k in {0,1,2}{
    \draw[gray!55] (\k*\S,0) -- (\k*\S,2*\S);
    \draw[gray!55] (0,\k*\S) -- (2*\S,\k*\S);
  }

  \coordinate (W) at (0,\S);
  \coordinate (E) at (2*\S,\S);
  \coordinate (S) at (\S,0);
  \coordinate (N) at (\S,2*\S);
  \coordinate (C) at (\S,\S);

  \foreach \P in {(0,0),(2*\S,0),(0,2*\S),(2*\S,2*\S)} \fill[gray!50] \P circle (1.4pt);

  \fill (W) circle (\dot);
  \fill (E) circle (\dot);
  \fill (S) circle (\dot);
  \fill (N) circle (\dot);
  \fill[TUeBlue] (C) circle (\dot);

  \node[anchor=south] at ($(C)+(0,0)$) {\large $(i,j)$};
  \node[anchor=north, TUeBlue] at ($(C)+(0,-0.02)$) {\Large $\mathbf{+4}$};

  \node[anchor=south] at ($(W)+(0,0)$) {\large $(i\!-\!1,j)$};
  \node[anchor=north, TUeBlue] at ($(W)+(0,0)$) {\Large $-1$};

  \node[anchor=south] at ($(E)+(0,0)$) {\large $(i\!+\!1,j)$};
  \node[anchor=north, TUeBlue] at ($(E)+(0,0)$) {\Large $-1$};

  \node[anchor=south] at ($(S)+(0,0)$) {\large $(i,j\!-\!1)$};
  \node[anchor=north, TUeBlue] at ($(S)+(0,0)$) {\Large $-1$};

  \node[anchor=south] at ($(N)+(0,0)$) {\large $(i,j\!+\!1)$};
  \node[anchor=north, TUeBlue] at ($(N)+(0,0)$) {\Large $-1$};
\end{tikzpicture}
\caption{Left: Indexing of grid. Right: stencil to update $u(i,j)$}

\end{figure}
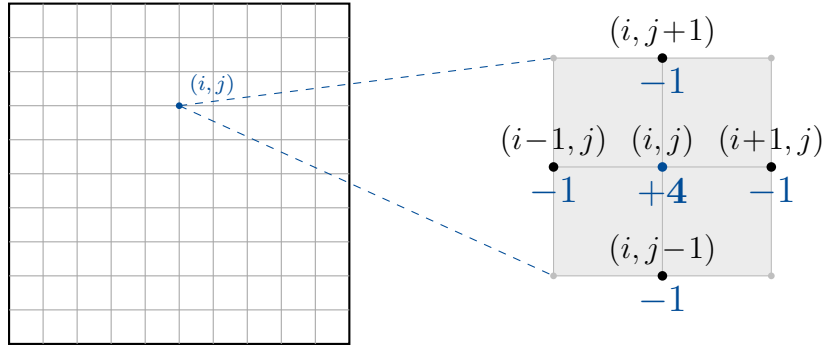

\begin{tcolorbox}
[colback=green!5!white,colframe=green!50!black]
\begin{definition}[Kronecker product]
For matrices $A \in \mathbb{R}^{m \times n}$ and $B \in \mathbb{R}^{p \times q}$, 
their \emph{Kronecker product} is the block matrix
\[
A \otimes B \;=\;
\begin{bmatrix}
a_{11}B & a_{12}B & \cdots & a_{1n}B \\
a_{21}B & a_{22}B & \cdots & a_{2n}B \\
\vdots  & \vdots  & \ddots & \vdots \\
a_{m1}B & a_{m2}B & \cdots & a_{mn}B
\end{bmatrix}
\in \mathbb{R}^{mp \times nq}.
\]
\end{definition}
\end{tcolorbox}

\paragraph{From grid to linear system.}
Collect the $n^2$ interior equations into $A_{2D}\, \bm{u} = \bm{g}$.
Using \emph{row-wise} (lexicographic) ordering $p=(j-1)n+i$, we obtain the \emph{block tridiagonal} matrix
\[
A_{2D}=\frac{1}{h^2}\,T_{2D},\qquad
T_{2D}=I_n \otimes T \;+\; T \otimes I_n,\qquad
T=\mathrm{tridiag}(-1,\,2,\,-1)\in\mathbb{R}^{n\times n}.
\]
Throughout, $T$ and $T_{2D}$ denote the \emph{unscaled} 1D and 2D matrices, and $A=\frac{1}{h^2}T$,
$A_{2D}=\frac{1}{h^2}T_{2D}$ the matrices that actually discretise the operator. Eigenvalues of the
one are $h^{-2}$ times those of the other; keeping the two apart avoids a factor $h^{-2}$ later.

Written blockwise, the same matrix has the structure
\[
A_{2D}=\frac{1}{h^2}\begin{bmatrix}
B & -I & 0  & \cdots & 0\\
-I&  B & -I & \ddots & \vdots\\
0 & -I & B  & \ddots & 0\\
\vdots & \ddots & \ddots & \ddots & -I\\
0 & \cdots & 0 & -I & B
\end{bmatrix}.
\]
Each block $B$ is itself tridiagonal with stencil $[-1,4,-1]$, and each $-I$ couples neighboring \emph{rows} of the grid.

\begin{remark}
$A_{2D}$ is sparse (each row has at most 5 nonzeros), symmetric, and positive definite (SPD) under Dirichlet data.
\end{remark}

\paragraph{Boundary conditions.}
For homogeneous Dirichlet ($u=0$ on $\partial\Omega$) the boundary values are known and \emph{do not} appear as unknowns.
For inhomogeneous Dirichlet ($u=g$), add boundary contributions to the right hand side near the boundary;
sparsity is preserved. One sided differences handle Neumann or Robin conditions, modifying only the first/last block rows.

\subsection{Graph Laplacians}

Graphs give rise to another fundamental family of sparse matrices.
Let $G=(V,E)$ be an undirected graph with $|V|=n$ vertices and edges $E$. The \emph{adjacency matrix} $A$ is defined by
\[
A_{ij}=\begin{cases}
1 & \text{if }(i,j)\in E,\\
0 & \text{otherwise}.
\end{cases}
\]
The \emph{degree matrix} $D$ is diagonal with entries $D_{ii}=\deg(i)=\sum_j A_{ij}$.

\begin{tcolorbox}[colback=green!5!white,colframe=green!50!black,title=Graph Laplacian]
The \textbf{graph Laplacian} of $G$ is
\(
L = D - A.
\)\\
It is symmetric and sparse, with $\mathrm{nnz}(L)=2|E|+\#\{i:\deg(i)>0\}$, which is $2|E|+n$ when
no vertex is isolated. It is positive semidefinite because
\[
x^\top L x=\sum_{\{i,j\}\in E}(x_i-x_j)^2\ \ge\ 0 ,
\]
so $L$ measures how much a signal varies along the edges of $G$.
\end{tcolorbox}

\paragraph{Example.}
For the path graph on 5 vertices:

\begin{center}
\begin{tikzpicture}[scale=1, every node/.style={circle,draw,fill=blue!15,inner sep=1pt,minimum size=0.6cm}]
\node (1) at (0,0) {1};
\node (2) at (1.5,0) {2};
\node (3) at (3,0) {3};
\node (4) at (4.5,0) {4};
\node (5) at (6,0) {5};
\foreach \i/\j in {1/2,2/3,3/4,4/5}{
  \draw (\i) -- (\j);
}
\end{tikzpicture}
\end{center}

the adjacency matrix is
\[
A=\begin{bmatrix}
0&1&0&0&0\\
1&0&1&0&0\\
0&1&0&1&0\\
0&0&1&0&1\\
0&0&0&1&0
\end{bmatrix},
\qquad
D=\mathrm{diag}(1,2,2,2,1),
\]
so
\[
L=D-A=\begin{bmatrix}
1&-1&0&0&0\\
-1&2&-1&0&0\\
0&-1&2&-1&0\\
0&0&-1&2&-1\\
0&0&0&-1&1
\end{bmatrix}.
\]

\begin{remark}
This $L$ is a rank-$2$ perturbation of $T=\mathrm{tridiag}(-1,2,-1)$, the unscaled matrix of the 1D
finite difference discretisation of $-u''(x)$: only the two corner entries differ.
\end{remark}

\paragraph{Topology changes the pattern.} The path graph is only one extreme; the \emph{cycle} $C_n$ (periodic path) and the \emph{star} $K_{1,n-1}$ (one hub, $n-1$ leaves) illustrate how differently $L$ can look for the same number of vertices.

\begin{center}
\begin{tikzpicture}[scale=0.85, every node/.style={circle,draw,fill=blue!15,inner sep=1pt,minimum size=0.5cm}]
\foreach \k/\lab in {0/1,1/2,2/3,3/4,4/5,5/6}{
  \node (c\k) at ({1.3*cos(60*\k)},{1.3*sin(60*\k)}) {\lab};
}
\foreach \k in {0,...,5}{
  \pgfmathtruncatemacro{\nxt}{mod(\k+1,6)}
  \draw (c\k) -- (c\nxt);
}
\node[draw=none,fill=none] at (0,-2.0) {Cycle graph $C_6$};
\end{tikzpicture}
\qquad
\begin{tikzpicture}[scale=0.85, every node/.style={circle,draw,fill=blue!15,inner sep=1pt,minimum size=0.5cm}]
\node (hub) at (0,0) {1};
\foreach \k/\ang in {2/90,3/18,4/-54,5/-126,6/162}{
  \node (v\k) at ({1.6*cos(\ang)},{1.6*sin(\ang)}) {\k};
  \draw (hub) -- (v\k);
}
\node[draw=none,fill=none] at (0,-2.2) {Star graph $K_{1,5}$};
\end{tikzpicture}
\end{center}

For the cycle $C_n$, $L$ is a circulant tridiagonal matrix with wrap-around corner entries; its eigenvalues are $\lambda_k=2-2\cos(2\pi k/n)$ (see Section~\ref{sec:sparsetoolbox}). For the star $K_{1,n-1}$, the hub has degree $n-1$ and every leaf has degree $1$, giving the sparse but highly non-tridiagonal spectrum $\sigma(L)=\{0,\,\underbrace{1,\dots,1}_{n-2},\,n\}$: a single dominant eigenvalue reflects the hub's central role, in sharp contrast to the path and cycle.

\begin{remarkablefact}
On a path graph, the Laplacian matrix ``almost'' coincides with the finite difference stencil of the PDE Laplacian.\footnote{More rigorously, this would correspond to the discrete version of the Laplacian with homogeneous Neumann boundary conditions}
\end{remarkablefact}

Graph Laplacians are put to work in several directions.
\begin{itemize}
  \item \textbf{Spectral clustering:} eigenvectors of $L$ reveal community structure in networks.
  \item \textbf{Graph neural networks:} message passing relies on $L$ or its normalized variants.
  \item \textbf{Electrical networks:} resistor networks are governed by Kirchhoff’s laws, which are encoded in $L$.
  \item \textbf{Data assimilation:} $L$ is used to construct diffusion operators, which are commonly used as spatial covariance matrices in meteorological and oceanographic data assimilation.

\end{itemize}

\begin{takeaway}
Graph Laplacians are sparse and symmetric matrices that extend the notion of the Laplace operator
from continuous domains to arbitrary discrete networks. Their role as ``discrete PDEs on graphs''
explains their central place in modern data science and network analysis.
\end{takeaway}

\subsection{Sparse Matrices in Machine Learning}

Machine learning problems generate some of the largest sparse matrices encountered today.  
The source of sparsity is \emph{locality}: each data point, feature, or parameter interacts only with a small subset of the whole system. In \emph{least squares regression}, for instance,
\[
\min_x \|Ax-b\|_2^2,
\]
the optimality condition yields the \emph{normal equations}
\[
A^\top A x = A^\top b,
\]
where $A\in\mathbb{R}^{m\times n}$ is the feature matrix and $b\in\mathbb{R}^m$ the labels.

\begin{tcolorbox}[colback=green!5!white,colframe=green!50!black,title=Sparsity of $A$ and $A^\top A$]
If each row of $A$ contains only $k\ll n$ nonzeros (each sample depends on $k$ features), 
then $\mathrm{nnz}(A)=O(mk)$ and $A$ is extremely sparse.  
The matrix $A^\top A$ is SPSD and encodes feature correlations.  
It is usually denser than $A$, but it still preserves significant sparsity compared to a dense $n\times n$ matrix.
\end{tcolorbox}

\begin{remark}
In text classification, $A$ is the \emph{document-term matrix}: rows correspond to documents, columns to words, and entries record word frequencies. Each document contains only a few words compared to the dictionary size $\Rightarrow$ $A$ is sparse. Then $A^\top A$ counts word co-occurrences across documents, producing a sparse graph of word relations.
\end{remark}

\begin{center}
\begin{tikzpicture}[scale=0.8]
  \draw (0,0) rectangle (2,3);
  \node at (1,-0.4) {$A$};
  \foreach \i/\j in {0.3/0.5,0.8/1.2,1.5/2.4,0.5/2.7,1.6/0.8}{
    \fill[blue] (\i,\j) rectangle +(0.2,0.2);
  }
  \draw (3,0) rectangle (5,2);
  \node at (4,-0.4) {$A^\top A$};
  \foreach \i/\j in {3.5/1.5,3.1/1.7,3.3/1.5,3.3/1.3,4.5/0.8,4.1/0.5,4.1/1.1,3.9/0.9,4.7/1.7,3.1/0.1}{
    \fill[red] (\i,\j) rectangle +(0.2,0.2);
  }
\end{tikzpicture}

\vspace{-0.2cm}
\emph{Sparse design matrix $A$ (left) vs.\ denser normal matrix $A^\top A$ (right).}
\end{center}
Second-order optimisation methods require the Hessian
\[
H_{ij}=\frac{\partial^2}{\partial x_i\partial x_j}\,F(x).
\]
When the loss $F$ couples only nearby features, $H$ inherits sparsity.  
This occurs frequently in \emph{PDE-constrained optimization}, and \emph{deep learning layers with local connectivity}.
\vspace{-0.3cm}
\begin{tcolorbox}[colback=yellow!5!white,colframe=yellow!50!black,title=Example: Logistic regression]
For logistic regression with data matrix $A$ (columns $a_k$) and labels $b_k\in \{-1,1\}$,
\[
F(x) = \sum_{k=1}^m \log\big(1+\exp(-b_k\,a_k^\top x)\big),
\]
the Hessian takes the form
\(H = A^\top W A\) where $W$ is a diagonal matrix of weights depending on $x$ ($w_k = \frac{\exp{b_k a_k^\top x}}{(1+\exp{b_k a_k^\top x})^2}$).  If $A$ is sparse, so is $H$, and computations can be done efficiently.
\end{tcolorbox}

\begin{takeaway}
In machine learning, sparsity comes from locality in data or models.  
Normal equations, Hessians, and Jacobians inherit this structure.  
Harnessing it is essential: without sparse techniques, modern large-scale ML training would be infeasible.
\end{takeaway}

\subsection{Sparsity in data assimilation}
In the final part of the course we will see that covariance matrices are important for weighting the contribution of different components to a data assimilation system. The sparsity structure of these covariance matrices has important computational and physical consequences. 

\begin{remarkablefact}
    In ensemble data assimilation methods, underlying `physical' covariances are obtained using sample estimates. This can lead to spurious long-range correlations which destroy sparsity. In a later part of the course we will consider \textit{localisation} methods, which are used to prescribe a given local sparsity structure on estimated spatial covariance matrices.
\end{remarkablefact}

\begin{remarkablefact}
    Sometimes we wish to estimate the value of the covariance matrix from measurement data. One method is to use the measurements to fit appropriate parameter values for a standard correlation function (e.g. Gaussian, auto-regressive functions) see e.g. \cite{martinDataAssimilationFOAM2007}. These distributions all have long tails and therefore yield dense matrices. In practice, correlation distributions are often truncated when they fall below  a threshold value (typically chosen as $0.2$) to ensure  sparsity/fixed bandwidth of the resulting matrices.
\end{remarkablefact}

\section{Sparse Matrix Toolbox}\label{sec:sparsetoolbox}

We have now seen where sparse matrices come from: discretised PDEs, graphs, and machine learning.
Knowing where they arise is not enough to compute with them. We also need to know how to store them,
how to operate on them efficiently, and how their spectral properties affect the algorithms of the
following chapters. That is the subject of this section.

\subsection{Sparsity Patterns and storage formats}

\begin{tcolorbox}[colback=green!5!white,colframe=green!50!black,title=Definition]
The \emph{sparsity pattern} of a matrix $A$ is the set of positions 
$(i,j)$ where $a_{ij}\neq 0$.
\end{tcolorbox}

\paragraph{Example.}
For the tridiagonal finite difference Poisson matrix:
\[
\mathrm{nnz}(A) \approx 3n,\quad 
\mathrm{density}(A)=\tfrac{\mathrm{nnz}(A)}{n^2}\approx 1/n.
\]

\paragraph{1D Poisson (tridiagonal).}
For $n$ interior points, $A=\tfrac{1}{h^2}\,\mathrm{tridiag}(-1,2,-1)$ has nonzeros only on the
main diagonal and the first sub-/super-diagonals.

\begin{center}
\begin{tikzpicture}[scale=0.45]
  \def\N{8} 
  \draw[step=1,gray!20,thin] (0,-\N-0.2) grid (\N+0.2,0);
  \foreach \i in {1,...,\N}{
    \fill[blue] (\i,-\i) circle (0.18);                 
    \ifnum\i<\N
      \fill[blue] (\i+1,-\i) circle (0.18);             
    \fi
    \ifnum\i>1
      \fill[blue] (\i-1,-\i) circle (0.18);             
    \fi
  }
\end{tikzpicture}
\end{center}
\vspace{-3mm}
\emph{Spy pattern for 1D Poisson, $n=8$. Half-bandwidth $b=1$, nnz/row $\approx 3$ (except at the ends).}

\paragraph{2D Poisson with lexicographic ordering.}
Let the grid be $n\times n$ ($N=n^2$ unknowns) and use row-wise (lexicographic) ordering. Then $A$ has nonzeros at indices $j=i$ (main), $j=i\pm 1$ \emph{within the same row block}, and $j=i\pm n$ (coupling to the row above/below). That is exactly the 5-point stencil.
\begin{center}
\begin{tikzpicture}[scale=0.4]
  \def\n{4}
  \pgfmathtruncatemacro{\N}{\n*\n}
  \draw[step=1,gray!20,thin] (0,-\N-0.2) grid (\N+0.2,0);

  \foreach \i in {1,...,\N}{
    \fill[blue] (\i,-\i) circle (0.15);

    \pgfmathtruncatemacro{\jp}{\i+\n}
    \pgfmathtruncatemacro{\jm}{\i-\n}
    \ifnum\jp>\N\else \fill[blue] (\jp,-\i) circle (0.15);\fi
    \ifnum\jm>0     \fill[blue] (\jm,-\i) circle (0.15);\fi

    \pgfmathtruncatemacro{\col}{mod(\i-1,\n)}
    \ifnum\col<\numexpr\n-1\relax
      \pgfmathtruncatemacro{\ip}{\i+1}
      \fill[blue] (\ip,-\i) circle (0.15);
    \fi
    \ifnum\col>0
      \pgfmathtruncatemacro{\im}{\i-1}
      \fill[blue] (\im,-\i) circle (0.15);
    \fi
  }
\end{tikzpicture}
\end{center}
\vspace{-2mm}
\emph{Spy pattern for 2D Poisson, $n=4$. Half-bandwidth $b=4$, nnz/row $\approx 5$ (except at end blocks).}

\paragraph{Graph Laplacians (path vs.\ star).}
Two small graphs with $n=8$ nodes to highlight how topology changes the pattern.

\begin{center}
\begin{tikzpicture}[scale=0.45]
  \def\n{8}

  \begin{scope}
    \draw[step=1,gray!20,thin] (0,-\n-0.2) grid (\n+0.2,0);
    \foreach \i in {1,...,\n}{
      \fill[red] (\i,-\i) circle (0.16);
      \ifnum\i<\n
        \fill[red] (\i+1,-\i) circle (0.16);
      \fi
      \ifnum\i>1
        \fill[red] (\i-1,-\i) circle (0.16);
      \fi
    }
    \node at (4.5,1) {\small Path graph $L$};
  \end{scope}

  \begin{scope}[shift={(12,0)}]
    \draw[step=1,gray!20,thin] (0,-\n-0.2) grid (\n+0.2,0);
    \foreach \j in {1,...,\n}{
      \fill[red] (1,-\j) circle (0.16);
      \fill[red] (\j,-1) circle (0.16);
    }
    \foreach \i in {1,...,\n}{
      \fill[red] (\i,-\i) circle (0.16);
    }
    \node at (4.5,1) {\small Star graph $L$};
  \end{scope}
\end{tikzpicture}
\end{center}
\emph{Left: path $L$ $\Rightarrow$ tridiagonal. Right: star $L$ $\Rightarrow$ dense first row/column + diagonal.}

\paragraph{Matrix structure and storage.}  
Sparse matrices are not stored entry by entry; only the nonzero pattern is kept.  
Common formats include:
\begin{itemize}
  \item \textbf{CSC/CSR (Compressed Sparse Column/Row):} general-purpose, store indices and values compactly.  
  \item \textbf{Diagonal/banded storage:} efficient when nonzeros cluster near the diagonal. 

\end{itemize}

\begin{example}[CSR and CSC formats]
Consider the $4\times 4$ sparse matrix
\[
A=\begin{bmatrix}
1 & 0 & 0 & 2 \\
0 & 3 & 0 & 4 \\
5 & 6 & 7 & 0 \\
0 & 0 & 8 & 9
\end{bmatrix}.
\]

\paragraph{CSR (Compressed Sparse Row).}
\begin{align*}
\texttt{val} &= [1,\,2,\,3,\,4,\,5,\,6,\,7,\,8,\,9], \\
\texttt{col\_ind} &= [1,\,4,\,2,\,4,\,1,\,2,\,3,\,3,\,4], \\
\texttt{row\_ptr} &= [1,\,3,\,5,\,8,\,10].
\end{align*}

\paragraph{CSC (Compressed Sparse Column).}
\begin{align*}
\texttt{val} &= [1,\,5,\,3,\,6,\,7,\,8,\,2,\,4,\,9], \\
\texttt{row\_ind} &= [1,\,3,\,2,\,3,\,3,\,4,\,1,\,2,\,4], \\
\texttt{col\_ptr} &= [1,\,3,\,5,\,7,\,10].
\end{align*}

\medskip
\noindent
Here, indices are 1-based for readability (note that Python/NumPy uses 0-based indexing). \\
- In CSR, \texttt{row\_ptr}[i] points to the start of row $i$ in \texttt{val}.  \\
- In CSC, \texttt{col\_ptr}[j] points to the start of column $j$ in \texttt{val}.
\end{example}

\paragraph{Bandwidth vs.\ profile.}  
Two measures quantify “how far” nonzeros lie from the diagonal:
\begin{itemize}
  \item \textbf{Half-bandwidth:} 
  \[
  b = \max\{|i-j| : a_{ij}\neq 0\},
  \]
  the worst-case distance from the diagonal; the band is then $2b+1$ diagonals wide. A tridiagonal
  matrix has $b=1$. Narrow bands allow cheap banded LU; wide bands inflate fill and storage.
  \item \textbf{Profile (envelope):} for each row $i$, the span $\ell_i=i-\min\{j:a_{ij}\neq0\}$ to
  the left of the diagonal; the profile is $\sum_i\ell_i$. This reflects actual memory usage more
  sensitively than bandwidth alone, since one wide row does not inflate it the way it inflates $b$.
\end{itemize}

\begin{remark}
In the 2D Poisson problem with $N=n^2$, each row has only 5 nonzeros, yet the half-bandwidth grows like $n$.  
Naive banded LU then costs $\mathcal{O}(N\,b^2) = \mathcal{O}(N^2)$, whereas good reorderings (e.g.\ nested dissection) reduce the cost to $\mathcal{O}(N^{3/2})$.
\end{remark}

\paragraph{Fill-in.}  
During LU factorization, zeros in $A$ can become nonzeros in $L$ or $U$.  
\begin{itemize}
  \item For banded matrices, fill grows with bandwidth.  
  \item In general, the amount of fill depends strongly on the sparsity pattern.  
\end{itemize}

\begin{center}
\begin{tikzpicture}[scale=0.55]
  \draw (0,0) rectangle (2,2);
  \node at (1,-0.4) {$A$};
  \fill[blue] (0.5,1.5) circle (0.12);
  \fill[blue] (1.5,0.5) circle (0.12);
  \fill[blue] (1.5,1.5) circle (0.12);
  \fill[blue] (0.5,0.5) circle (0.12);
  \draw (3,0) rectangle (5,2);
  \node at (4,-0.4) {$LU$};
  \foreach \i/\j in {3.5/1.5,4.5/0.5,4.5/1.5,3.5/0.5,4/1}{
    \fill[red] (\i,\j) circle (0.12);
  }
\end{tikzpicture}

\vspace{-0.2cm}
\emph{Illustration: fill-in introduces new nonzeros (red).}
\end{center}

\begin{remark}
By permuting rows/columns before factorization, one can reduce bandwidth and hence fill-in. Effective reordering can save orders of magnitude in memory and runtime---critical for large-scale finite element systems.
\end{remark}

\begin{anecdote}
The invention of compressed storage and reordering heuristics in the 1970s was pivotal: without them, sparse PDE systems of size $10^5$ could not be solved on mainframes.
\end{anecdote}

The inverse of a sparse matrix is not guaranteed to be sparse, and can suffer from more fill-in than an LU or Cholesky decomposition. One example is the given in Figure \ref{fig:sparsityinverse}, where $A$ is a lower triangular matrix with ones on the diagonal and first subdiagonal, plus a single entry above the diagonal. The LU factorization of $A$ is sparse, with some fill-in seen in $U$ (lower right panel). This aligns with the element above the diagonal. On the other hand, $A^{-1}$ is a dense matrix, with almost all entries being non-zero. In general we aim to avoid using inverse matrices, but in the final section of the course we will see that inverse covariance matrices play a very important role in data assimilation and cannot be completely avoided. 

\begin{figure}
    \centering
   \includegraphics[width=0.5\linewidth]{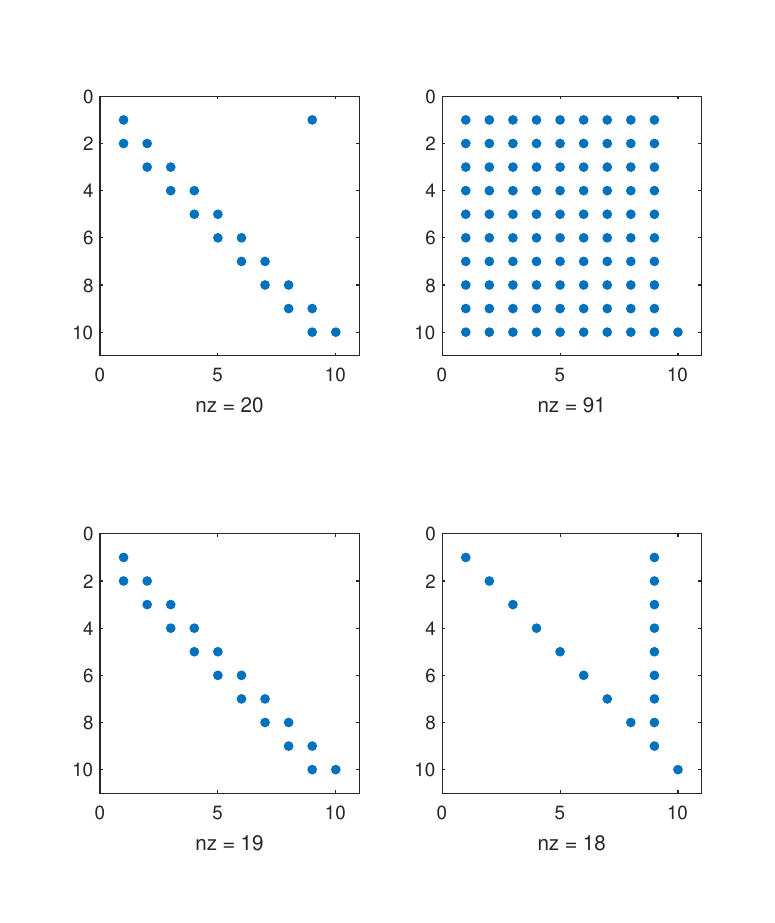}
    \caption{Visualisation of fill-in with LU vs inverse. Top left: sparsity pattern original matrix, $A$. Top right: sparsity pattern of $A^{-1}$. Lower left: sparsity pattern of $L$. Lower right: sparsity pattern of $U$.}
    \label{fig:sparsityinverse}
\end{figure}

\subsection{Spectral Properties}\label{sec:spectral}

For some particular sparsity patterns, we can compute the eigenvalues explicitly:
\begin{tcolorbox}[colback=yellow!5!white,colframe=yellow!50!black,title=Tridiagonal matrices]
If a tridiagonal matrix is \textbf{Toeplitz}, i.e.\ with stencil $[p,d,q]$ ($p$ on the subdiagonal,
$d$ on the diagonal, $q$ on the superdiagonal) and $pq>0$, then the eigenvalues can be computed
explicitly as:
\[
\lambda_j =d-2\sqrt{pq}\,\cos\Big(\frac{j\pi}{n+1}\Big)  \quad j=1,\dots,n.
\]
Replacing $-$ by $+$ gives the same set of eigenvalues, since $j\mapsto n+1-j$ flips the sign of the
cosine; both forms occur in the literature. In the symmetric case $q=p$ this reads
$\lambda_j=d+2p\cos\big(\frac{j\pi}{n+1}\big)$, which is the form used in Problem~4.
If $pq<0$ the eigenvalues are complex and the formula does not apply.
\end{tcolorbox}

\begin{tcolorbox}[colback=green!5!white,colframe=green!50!black,title=Definition]
A circulant matrix is a particular form of Toeplitz matrix which has rows composed of all the same elements, with each row being shifted one element to the right relative to the preceding row, i.e.
\[C = \begin{bmatrix} c_0& c_1 & \dots & c_{n-2}& c_{n-1} \\
c_{n-1} & c_0 &c_1 & & c_{n-2} \\
\vdots & c_{n-1} & c_0 &\ddots& \vdots \\
c_{2} && \ddots &\ddots &  c_1 \\
c_1 & c_2 &\dots &c_{n-1} & c_0\end{bmatrix}.\]
 Circulant matrices are diagonalized by a discrete Fourier transform, i.e. the eigenvalues are given by 
\[
\lambda_j =c_0 + c_{n-1} \omega^{-j}+c_{n-2}\omega^{-2j}+\dots +c_{1}\omega^{-(n-1)j}  \quad j=0,\dots,n-1,
\]
where $\omega = \exp(\frac{2\pi i}{n})$ is an $n$-th root of unity. The eigenvectors of $C$ are given by
\[v_j = \frac{1}{\sqrt{n}}(1,\omega^j,\omega^{2j},\dots,\omega^{(n-1)j})^T, \quad j=0,1,2,\dots,n-1.\]
\end{tcolorbox}

\paragraph{Eigenvalues of model problems.}
For the 1D Poisson matrix,
\[
\lambda_j = \frac{4}{h^2}\sin^2\!\left(\frac{j\pi}{2(n+1)}\right),\quad j=1,\dots,n.
\]
Hence
\[
\kappa(A)=\frac{\lambda_{\max}}{\lambda_{\min}}=\mathcal{O}(n^2).
\]

Similarly, consider the matrix of an implicit diffusion step on the unit square with Dirichlet
boundary conditions,
\[
A_\nu \;=\; I + \nu\,A_{2D}\;\in\;\mathbb{R}^{n^2\times n^2},
\qquad
A_{2D}=\frac{1}{h^2}\big(I_n\otimes T+T\otimes I_n\big),\quad T=\mathrm{tridiag}(-1,2,-1),
\]
where $\nu>0$ collects the diffusion coefficient and the time step. (The letter $A_\nu$ keeps this
matrix apart from the 1D Poisson matrix $A$ of the previous paragraph.) Since $A_{2D}$ has the
eigenvalues of Section~\ref{sec:2DPoisson}, the $n^2$ eigenvalues of $A_\nu$ are 
$$
\mu_{i,j} = 1 + \frac{4 \nu}{h^2} \left( \sin^2 \left ( \frac{i}{2(n+1)}\pi \right) + \sin^2 \left ( \frac{j}{2(n+1)}\pi  \right ) \right), \quad i, j = 1,2, \ldots, n.
$$

\begin{tcolorbox}[colback=yellow!5!white,colframe=yellow!50!black,title=Implication]
These condition numbers grow as $h^{-2}$: iterative solvers (CG, GMRES, etc.) slow down as the mesh refines.  
Spectral clustering of eigenvalues governs convergence rates: if most eigenvalues lie in a tight interval, Krylov methods converge quickly; wide spreads or outliers cause stagnation.  
\textbf{Preconditioners} are therefore essential to compress the spectrum.
\end{tcolorbox}

\paragraph{Spectral viewpoint on sparsity.}
Even though $A$ is sparse, its eigenvalues may cluster near zero, leading to ill-conditioning.  
Sparsity alone does not guarantee fast convergence: what matters is the \emph{distribution of eigenvalues}.  
For example, reordering does not change eigenvalues (it is a similarity transformation), but preconditioning actively reshapes the spectrum to improve solver efficiency.

\paragraph{Graph Laplacians.}
Eigenvalues encode combinatorial structure:
\begin{itemize}
  \item Smallest eigenvalue $=0$ (multiplicity = number of connected components).  
  \item The second smallest eigenvalue (Fiedler value) measures algebraic connectivity and governs diffusion/mixing rates.
\end{itemize}

\paragraph{Machine learning.}
Normal equations $A^\top A$ have eigenvalues $\sigma_i^2$, the squares of the singular values of $A$.  
Ill-conditioning corresponds to nearly collinear features, where small singular values slow down gradient descent and amplify noise.  
Regularization (ridge regression, Tikhonov) shifts the spectrum, lifting small eigenvalues away from zero.

\begin{takeaway}
\begin{itemize}
  \item Sparsity patterns (bandwidth, profile) determine storage and fill-in.  
  \item Compressed storage formats (CSR/CSC) make large problems feasible.  
  \item Fill-in can destroy sparsity; reordering strategies mitigate this.  
  \item \textbf{Spectral properties are decisive:} eigenvalues dictate conditioning, convergence, and the design of effective preconditioners.  
\end{itemize}
Efficient sparse linear algebra requires both algebraic insight (spectrum, conditioning) and practical storage/reordering techniques.
\end{takeaway}

\newpage
\section{Exercises}

\paragraph{Problem 1 (FD accuracy at a point).}
For $u(x)=\sin(\pi x)$, approximate $u''(0.5)$ with the 3‑point stencil
\[
  D_{xx}^h u(0.5)=\frac{u(0.5+h)-2u(0.5)+u(0.5-h)}{h^2}
\]
using $h=0.25$, $h=0.125$ and $h=0.0625$. Compare with the exact $u''(0.5)$, and check that the
error is divided by roughly $4$ each time $h$ is halved.

\notebookpointer{\nbpath{week2/script\_week2\_1.ipynb}}

\paragraph{Problem 2 (1D Poisson: assemble \& solve by hand).}
Discretize $-u''=1$ on $(0,1)$ with $u(0)=u(1)=0$ and $n=4$ interior points ($h=1/5$).
\begin{enumerate}
  \item Write the linear system $Au=g$.
  \item Solve it explicitly for $u\in\mathbb{R}^4$.
\end{enumerate}

\notebookpointer{\nbpath{week2/script\_week2\_1.ipynb}}

\paragraph{Problem 3 (Discretisation of a BVP and application of a basic iterative method, exam 2024/2025).}
Consider the one-dimensional boundary value problem
\[
\begin{cases}
-\,u''(x) + c\,u(x) = 0, & x\in(0,1),\\
u(0)=1,\quad u(1)=0,
\end{cases}
\]
where $c>0$ is a constant. Use a uniform mesh with grid points $x_j=jh$, $h=\frac{1}{n+1}$, $j=0,\ldots,n+1$, and denote $u_j\approx u(x_j)$. 
\begin{enumerate}
\item[(a)] Give the finite difference scheme for the discretisation of the BVP.
\item[(b)] Let $u=(u_1,\ldots,u_n)^\top$ be the unknown vector solving $A u=b$. Determine $A$ and $b$.
\item[(c)] We solve the linear system by the Jacobi method based on the splitting $A=D-E-F$,
where $D$ is the diagonal part and $E$ (resp.\ $F$) the strict lower (resp.\ upper) triangular part.
(This is the splitting $A=D-L-U$ of Chapter~1, with $E=L$ and $F=U$; the exam statement uses the
other pair of letters.) Give the exact Jacobi iteration.
\item[(d)] If $e^k=u-u^k$ denotes the error at iteration $k$, then $e^k=G^k e^0$ with iteration matrix $G$.
Give $G$ explicitly.
\end{enumerate}

\paragraph{Problem 4 (Finite Differences for Boundary Value Problems, resit exam 2024/2025).}
Consider the boundary value problem:
\[ -u''(x) + cu(x) = f(x), \quad x \in (0,1), \quad u(0) = \alpha, \quad u(1) = \beta, \]
where $c$ is a positive constant, and $f(x)$ is a given source term.
\begin{enumerate}
    \item[(a)] Discretize the equation using the finite difference method with a uniform grid $x_j = jh$, $j = 0, 1, \dots, n+1$, where $h = 1/(n+1)$. Derive the corresponding difference scheme.
    \item[(b)] Write the scheme in matrix form. Give the expression of the matrix and right hand side.
    \item[(c)] Compute the spectral condition number of the matrix $A$. \\
    \textbf{Hint:}  For a symmetric positive definite matrix we have: $\kappa(A) = \frac{\lambda_{\max}(A)}{\lambda_{\min}(A)}$, 
    where $\lambda_{\max}(A)$ and $\lambda_{\min}(A)$ are the largest and smallest eigenvalues of $A$, respectively.
    By the Toeplitz box of Section~\ref{sec:spectral}, a symmetric tridiagonal Toeplitz matrix
    $\mathrm{tridiag}(p,d,p)$ has eigenvalues
    \[
    \lambda_k = d + 2p \cos\left(\frac{k\pi}{n+1}\right), \quad k = 1, 2, \dots, n.
    \]
    This formula can be used to compute $\lambda_{\max}$ and $\lambda_{\min}$.

    \item[(d)] Analyze how the spectral condition number $\kappa(A)$ behaves when $c$ decreases and $n$ increases and interpret the result.
\end{enumerate}

\paragraph{Problem 5 (Bandwidth and naive LU cost).}
For a 2D $n\times n$ grid with natural ordering, take $n=20$ ($N=n^2$). Estimate the half-bandwidth $b$ and the work of banded LU via $O(N\,b^2)$.

\notebookpointer{\nbpath{week2/script\_week2\_2.ipynb}}

\paragraph{Problem 6 (Graph spectra, small and exact).}
For the path graph on $4$ vertices, compute the eigenvalues of its Laplacian $L$. Then do the same for the complete graph $K_4$. Compare the spectral gaps.

\paragraph{Problem 7 (2D Laplacian on an $n\times n$ grid: spectrum and conditioning).}
  Let $T=\mathrm{tridiag}(-1,2,-1)\in\mathbb{R}^{n\times n}$ (Dirichlet 1D Laplacian without the
  $1/h^2$ scaling) and let
  \[
    T_{2D}=I_n\otimes T + T\otimes I_n\in\mathbb{R}^{n^2\times n^2}
  \]
  be the unscaled 2D matrix of Section~\ref{sec:2DPoisson}, so that $A_{2D}=\frac{1}{h^2}T_{2D}$.
  \begin{enumerate}
    \item[(i)] Show that the eigenvalues of $T_{2D}$ are
      \[
        \lambda_{p,q}=\lambda_p(T)+\lambda_q(T)
        = \Big(2-2\cos\frac{p\pi}{n+1}\Big)+\Big(2-2\cos\frac{q\pi}{n+1}\Big),
        \quad p,q=1,\dots,n.
      \]
    \item[(ii)] With the standard spacing $h=1/(n+1)$, the discretisation matrix is
      $A_{2D}=\frac{1}{h^2}T_{2D}$. Prove that
      \[
        \lambda_{\max}(A_{2D})\sim \frac{8}{h^2},\qquad
        \lambda_{\min}(A_{2D})\to 2\pi^2\quad(n\to\infty),
      \]
      and deduce the asymptotic condition number
      \(
        \kappa_2(A_{2D})={\cal O}(n^2).
      \)
    \item[(iii)] (Optional, $d$D generalization.) For $d$ spatial dimensions with $n^d$ unknowns,
      show that $\kappa_2(A_{dD})={\cal O}(n^2)$ still holds, where $A_{dD}$ is the analogous scaled matrix.
  \end{enumerate}

\notebookpointer{\nbpath{week2/script\_week2\_2.ipynb}}

\paragraph{Problem 8 (Graph Laplacians: spectra and expansion, general $n$).}
  Let $G=(V,E)$ be a simple undirected graph on $n$ vertices with (combinatorial) Laplacian $L=D-A$.
  \begin{enumerate}
    \item[(i)] Show that $0$ is an eigenvalue of $L$ with multiplicity equal to the number of connected components.
    \item[(ii)] For the cycle graph $C_n$, give the eigenvalues and eigenvectors of its Laplacian in
terms of $\omega=\exp\big(\tfrac{2\pi i}{n}\big)$. (You may use the properties of circulant
matrices.) Give an interval in which these eigenvalues lie.
      \item[(iii)] For the star $K_{1,n-1}$ on $n$ vertices (one hub, $n-1$ leaves), show that
        $\sigma(L)=\{0,\,1\ \text{(mult $n{-}2$)},\,n\}$.
   
  \end{enumerate}

\paragraph{Problem 9 (1D Poisson eigenvalues and $\kappa(A)$, small $n$).}
For $n=5$ interior points, use
\[
  \lambda_j=\frac{4}{h^2}\sin^2\!\Big(\frac{j\pi}{2(n+1)}\Big),\qquad h=\frac{1}{n+1},
\]
to list $\lambda_1,\dots,\lambda_5$ numerically and compute $\kappa(A)$.

Repeat this, for the same $n=5$, with the implicit diffusion matrix $A_\nu=I+\nu A_{2D}$ of
Section~\ref{sec:spectral}, whose $n^2=25$ eigenvalues are
$$
\mu_{i,j} = 1 + \frac{4 \nu}{h^2} \left( \sin^2 \left ( \frac{i}{2(n+1)}\pi \right) + \sin^2 \left ( \frac{j}{2(n+1)}\pi  \right ) \right), \quad i, j = 1,2, \ldots, n,
$$
taking $\nu =0.0025$. What differences do you notice between the two sets of eigenvalues? What will happen to the largest/smallest eigenvalues as $n$ increases?

\paragraph{Problem 10 (Data assimilation spectrum example).}
We can construct a second order autoregressive (SOAR) correlation function on the unit circle for $N$ equally distributed gridpoints, via
\[C(i,j) = \left(1+ \frac{|2a \sin(\frac{\theta_{i,j}}{2})|}{L}\right)\exp\left(\frac{-|2a \sin(\frac{\theta_{i,j}}{2})|}{L}\right),\quad 1\le i,j\le N \]
where $L > 0$ is the correlation lengthscale, $\theta_{i,j}$ denotes the angle between grid points $i$
and $j$, and $a$ is the radius of the domain \cite{habenConditioningPreconditioningMinimisation2011}; take $a=1$ throughout. This yields a dense matrix, which can be ill-conditioned. What happens to the spectrum if we replace all values less than some threshold (e.g. $0.2$) by zero? Investigate how this changes as the lengthscale, threshold value and problem dimension increase.
\notebookpointer{\nbpath{week2/script\_week2\_3.ipynb}}

\paragraph{Problem 11 (Concrete ML sparsity example).}
Four documents over a $6$-word vocabulary with bags:
\[
d_1=\{1,2\},\; d_2=\{2,3\},\; d_3=\{3,4\},\; d_4=\{1,4\}.
\]
Build the sparse document-term matrix $A\in\mathbb{R}^{4\times6}$ (entries $0/1$). Compute $A^\top A$ explicitly and report its sparsity (which off‑diagonal entries are nonzero?). Interpret the pattern.

\chapter{Applications of Basic Iterative Methods}

\section*{Overview}

Chapter 1 introduced Jacobi, Gauss--Seidel, SOR, and Richardson as instances of a single fixed-point iteration $x^{(k+1)}=Gx^{(k)}+c$, convergent whenever $\rho(G)<1$. This chapter shows how this works on three model problems where the spectrum of $G$ can be computed \emph{explicitly}, which turns the abstract convergence criterion into concrete, comparable convergence rates. This chapter covers: (1) the 1D Poisson matrix, where explicit eigenpairs give exact convergence factors for Jacobi/GS/SOR; (2) graph Laplacians, where the same iterations become diffusion/random-walk processes and connect to spectral clustering; (3) machine learning, where Richardson iteration is exactly gradient descent and the spectrum of $X^\top X$ governs both convergence speed and \emph{spectral bias}.

\begin{tcolorbox}[colback=softblue,colframe=TUblue!40!black,title=\textbf{Learning objectives},fonttitle=\bfseries]
By the end of this chapter, you should be able to:
\begin{itemize}
    \item Derive the explicit eigenvalues of the 1D Poisson matrix and use them to compute the exact spectral radius (and hence convergence rate) of Jacobi, Gauss--Seidel, and SOR.
    \item Explain why Gauss--Seidel converges roughly twice as fast as Jacobi, and why SOR with an optimal $\omega$ improves the rate further.
    \item Interpret Jacobi iteration on a graph Laplacian as a random walk/diffusion process, and relate its slow modes to the Fiedler vector and spectral clustering.
    \item Recognize Richardson iteration as gradient descent on a quadratic loss, and relate the spectrum of $A$ (or $X^\top X$) to the choice of step size and the resulting convergence rate.
    \item Explain the phenomenon of \emph{spectral bias} in machine learning as the same spectral-filtering mechanism that governs smoothing in PDE solvers, with the roles of low- and high-frequency modes swapped.
\end{itemize}
\end{tcolorbox}

\section{Basic Iterative Methods for PDE discretisations}\label{sec:BIMPDE}

We now go beyond the general fixed point framework of Chapter~1 and analyze in detail
the convergence of classical iterative methods for the model problem of the one-dimensional Poisson equation with Dirichlet boundary conditions introduced in Chapter 2.
This case is simple enough to allow explicit eigenvalue analysis, yet rich enough
to illustrate the key phenomena that will reappear in higher--dimensional PDEs,
graphs, and machine learning applications.

Recall the discretization of $-u''=f$ on $(0,1)$ with homogeneous Dirichlet boundary
conditions using $n$ interior points, mesh width $h=\tfrac{1}{n+1}$. The stiffness
matrix is
\begin{equation}\label{eq:Poisson disc matrix}
A = \frac{1}{h^2}\,\mathrm{tridiag}(-1,\,2,\,-1)
= \frac{1}{h^2}
\begin{bmatrix}
2 & -1 \\
-1 & 2 & -1 \\
& \ddots & \ddots & \ddots \\
&& -1 & 2 & -1 \\
&&& -1 & 2
\end{bmatrix}\in\mathbb{R}^{n\times n}.
\end{equation}
This matrix is symmetric, positive definite, and tridiagonal.

\begin{remark}
The eigenvalues of $A$ are explicitly known:
\begin{equation}
\label{eq:lambda}
\lambda_k(A) = \tfrac{1}{h^2}\bigl(2 - 2\cos\theta_k\bigr), 
\qquad \theta_k = \frac{k\pi}{n+1},\; k=1,\dots,n.
\end{equation}
or alternatively
\[
\lambda_k(A) = \frac{4}{h^2}\,\sin^2\!\left(\frac{k\pi}{2(n+1)}\right) = \frac{4}{h^2}\,\sin^2\!\left(\frac{k\pi h}{2}\right), 
\qquad k=1,\dots,n.
\]
Hence $\lambda_{\min}\sim \pi^2$, $\lambda_{\max}\sim 4/h^2$, and
$\kappa(A)=\mathcal{O}(n^2)$.
\end{remark}

\begin{center}
\begin{tikzpicture}[scale=0.8]
  \draw[->] (-0.2,0) -- (6.8,0) node[right]{\small $j$};
  \draw[->] (0,-1.3) -- (0,1.3);
  \draw[thick,domain=0:6,smooth,samples=200] plot (\x,{sin(\x*180/6)});
  \draw[thick,gray,domain=0:6,smooth,samples=400] plot (\x,{0.5*sin(4*\x*180/6)});
  \node[anchor=west] at (5.2,0.85) {\small low $k$};
  \node[anchor=west,gray] at (5.2,-0.85) {\small high $k$};
\end{tikzpicture}
\par\small\emph{Low-frequency ($k$ small) vs.\ high-frequency ($k$ large) eigenvectors
$v^{(k)}_j=\sin(jk\pi/(n+1))$ of $A$. Oscillatory modes are the ones a \emph{damped} iteration
removes first, while smooth modes decay slowly; this is the smoothing property multigrid rests on
\cite{briggsMultigridTutorial2000}. Note that undamped Jacobi does \emph{not} have it:
$|\mu_{n+1-k}|=|\mu_k|$, so the most oscillatory mode decays exactly as slowly as the smoothest
(Section~\ref{sec:BIMPDE} and Problem~2(c)); it is weighted Jacobi, at $\omega=\tfrac23$, that
smooths.}
\end{center}

\subsection{Convergence of the Jacobi iteration}

We split $A = D - L - U$ with $D = \tfrac{2}{h^2}I$.  
The Jacobi iteration matrix is
\[
G_J \;=\; I - D^{-1}A 
= I - \tfrac{h^2}{2}A 
= I - \tfrac{1}{2}\,\operatorname{tridiag}(-1,\,2,\,-1) 
= \operatorname{tridiag}\!\Bigl(\tfrac{1}{2},\,0,\,\tfrac{1}{2}\Bigr).
\]
Since $G_J = I - \tfrac{h^2}{2}A$ and given \eqref{eq:lambda} the Jacobi eigenvalues are
\[
\mu_k = 1 - \tfrac{h^2}{2}\,\lambda_k(A)
= 1 - \tfrac{h^2}{2}\cdot\tfrac{1}{h^2}\bigl(2-2\cos\theta_k\bigr)
= \cos\theta_k
= \cos\!\Bigl(\tfrac{k\pi}{n+1}\Bigr).
\]

Hence the spectral radius is
\[
\rho(G_J) = \max_k|\mu_k| 
= \cos\!\Bigl(\tfrac{\pi}{n+1}\Bigr)
= 1 - \frac{\pi^2}{2(n+1)^2} + \mathcal{O}\!\bigl(n^{-4}\bigr),
\]
so Jacobi converges very slowly as $n$ grows.

\subsection{Convergence of the Gauss--Seidel iteration}

For Gauss--Seidel we split $A=(D-L)-U$, with
\[
D=\tfrac{2}{h^2}I,\quad L=\tfrac{1}{h^2}\,\mathrm{tridiag}(1,0,0),\quad
U=\tfrac{1}{h^2}\,\mathrm{tridiag}(0,0,1).
\]
Then the iteration matrix is (see also the first exercise on the list for the exact formula)
\[
G_{GS}=(D-L)^{-1}U.
\] 

\begin{tcolorbox}[colback=yellow!5!white,colframe=yellow!50!black]
\begin{theorem}[Relation between eigenvalues of Gauss--Seidel and Jacobi]\label{thm:GSJ}
Let $A = D - L - U$ with $D$ diagonal, $-L$ strictly lower, $-U$ strictly upper, and all diagonal
entries of $D$ nonzero. Assume moreover that $A$ has \emph{property~A}.\footnote{$A$ has property~A if some permutation $P$ brings it to the block form
$\left[\begin{smallmatrix}D_1&F\\ E&D_2\end{smallmatrix}\right]$ with $D_1, D_2$ diagonal or equivalently,
if the graph of its off-diagonal pattern is bipartite. The 1D Poisson matrix in the natural ordering, the
case of interest here, satisfies this property.} Define
\[
G_J = D^{-1}(L+U), 
\qquad 
G_{GS} = (D-L)^{-1}U.
\]
If $\mu$ is a \textbf{nonzero} eigenvalue of $G_J$, then $\lambda = \mu^2$ is an eigenvalue of $G_{GS}$.
\end{theorem}
\end{tcolorbox}

\begin{proof}
By definition,
\[
G_{GS} = (I - D^{-1}L)^{-1} D^{-1}U.
\]
Hence $\lambda \in \sigma(G_{GS})$ if and only if
\[
\det\!\big((I - D^{-1}L)^{-1}D^{-1}U - \lambda I\big) = 0,
\]
which is equivalent to
\[
\det\!\big(D^{-1}U - \lambda(I - D^{-1}L)\big) = 0
\;\;\Longleftrightarrow\;\;
\det\!\left(\sqrt{\lambda} I - D^{-1}\left(\sqrt{\lambda}L+\frac{1}{\sqrt{\lambda}}U\right )\right) = 0. \tag{1}
\]
Thus $\sqrt{\lambda}$ is an eigenvalue of
\[
M(\lambda) := D^{-1}\!\Big(\tfrac{1}{\sqrt{\lambda}}U + \sqrt{\lambda}\,L\Big).
\]
Now observe that $M(\lambda)$ is similar to $D^{-1}(L+U)$. \footnote{Here we use a bit a shortcut and we skip some steps. In fact this is true for matrices which have property $A$ (i.e. can be reduced to a block structure via a permutation of rows and columns. In this proof we assume that this permutation has already been done).} 
Therefore
\[
\sigma(M(\lambda)) = \sigma(D^{-1}(L+U)) = \sigma(G_J).
\]
Hence $\sqrt{\lambda} \in \sigma(G_J)$. Same holds for $-\sqrt{\lambda}$.
\end{proof}

\noindent\textbf{Conclusion.}
Every nonzero Jacobi eigenvalue contributes $\mu_k^2=\cos^2\!\bigl(\tfrac{k\pi}{n+1}\bigr)$ to the
Gauss--Seidel spectrum, and in particular $\rho(G_{GS})=\rho(G_J)^2$, so Gauss--Seidel converges twice
as fast as Jacobi. Note that in the proof we divide by $\sqrt\lambda$, so we don't consider the case $\lambda=0$, which \emph{is} an eigenvalue
of $G_{GS}$ here, of high multiplicity (Remark~\ref{rem:GSspectrum}). And it gives an \emph{inclusion} of spectra, not an equality.

\begin{remark}[Why the property A is important]
Take $D=I$ and
\[
G_J=\begin{bmatrix}0&1&1\\ 1&0&1\\ 1&1&0\end{bmatrix},
\qquad\text{that is}\qquad
A=I-G_J=\begin{bmatrix}1&-1&-1\\ -1&1&-1\\ -1&-1&1\end{bmatrix}.
\]
Here $\sigma(G_J)=\{2,-1,-1\}$, so the squares are $\{4,1,1\}$, whereas
$\sigma(G_{GS})=\{0,\,2-\sqrt5,\,2+\sqrt5\}$: \emph{no} Jacobi eigenvalue's square is a Gauss--Seidel
eigenvalue. The off-diagonal pattern of $A$ is the triangle $K_3$, an odd cycle, hence not bipartite,
so $A$ does not have property~A.
\end{remark}

\begin{remark}[Spectrum and non-diagonalisability of Gauss--Seidel]\label{rem:GSspectrum}
For the one-dimensional Poisson problem with $n$ interior points and natural ordering, $G_J$ is
symmetric and hence diagonalisable. By contrast,
\[
\det(\lambda I-G_{GS})=\lambda^{\lceil n/2\rceil}
\prod_{j=1}^{\lfloor n/2\rfloor}\Bigl(\lambda-\cos^2\!\Bigl(\tfrac{j\pi}{n+1}\Bigr)\Bigr),
\]
so the nonzero eigenvalues are the $\lfloor n/2\rfloor$ values $\cos^2(j\pi/(n+1))$, each \emph{simple},
while $0$ has algebraic multiplicity $\lceil n/2\rceil$. Since $\ker G_{GS}=\operatorname{span}\{e_1\}$,
its geometric multiplicity is one, so a Jordan block appears as soon as $\lceil n/2\rceil\ge2$.
Hence $G_{GS}$ is non-normal for $n\ge2$ and non-diagonalisable for $n\ge3$; the two small cases are
exceptions, since $G_{GS}=0$ for $n=1$ and its eigenvalues $0,\tfrac14$ are distinct for $n=2$.
\end{remark}

\subsection{Convergence of the SOR iteration}

For SOR with relaxation parameter $\omega\in(0,2)$ we use the splitting $
A = \tfrac{1}{\omega}(D-\omega L) - \big((1-\omega)D+\omega U\big)$, so that the iteration matrix is
\[
G_\omega = (D-\omega L)^{-1}\big((1-\omega)D+\omega U\big).
\]
Young's theorem \cite{youngConvergencePropertiesSymmetric1970,vargaMatrixIterativeAnalysis2000}\footnote{we accept this result without proof} gives for certain matrices a relation between the SOR convergence
factor and that of Jacobi and the optimal parameter $\omega$:
\[
\omega_{\mathrm{opt}} = \frac{2}{1+\sqrt{1-\rho(G_J)^2}}, 
\qquad
\rho(G_{\omega_{\mathrm{opt}}}) = 
\frac{1-\sqrt{1-\rho(G_J)^2}}{1+\sqrt{1-\rho(G_J)^2}}.
\]
For our 1D Poisson problem, $\rho(G_J)=\cos\!\big(\tfrac{\pi}{n+1}\big)$, 
which yields
\[
\omega_{\mathrm{opt}} = \frac{2}{1+\sin\!\big(\tfrac{\pi}{n+1}\big)},\, \rho(G_{\omega_{\mathrm{opt}}}) = \frac{1-\sin\!\big(\tfrac{\pi}{n+1}\big)}{1+\sin\!\big(\tfrac{\pi}{n+1}\big)}.
\]
\begin{tcolorbox}[colback=yellow!5!white,colframe=yellow!50!black]
\begin{proposition}[Asymptotic Convergence and Limitations]
Based on the previous results we can derive the asymptotic convergence behaviour of the three classical iterative methods. For large $n$, the spectral radii expand as
\begin{align*}
\rho(G_J) &= 1 - \tfrac{\pi^2}{2(n+1)^2} + \mathcal{O}(n^{-4}),\\[0.3em]
\rho(G_{GS}) &= 1 - \tfrac{\pi^2}{(n+1)^2} + \mathcal{O}(n^{-4}),\\[0.3em]
\rho(G_{\omega_{\text{opt}}}) &= 1 - \tfrac{2\pi}{n+1} + \mathcal{O}(n^{-2}),
\end{align*}
\end{proposition}
\end{tcolorbox}

\begin{example}[Numerical comparison, $n=49$]
Take $n=49$ interior points (so $h=1/50$). The three spectral radii evaluate to
\[
\rho(G_J)=\cos\!\Bigl(\tfrac{\pi}{50}\Bigr)\approx 0.99803,\qquad
\rho(G_{GS})=\rho(G_J)^2\approx 0.99606,\qquad
\rho(G_{\omega_{\rm opt}})\approx 0.88184.
\]
The asymptotic rate predicts $k\approx -6/\log_{10}\rho$ iterations to gain six digits, that is
\[
k_J \approx 7000,\qquad k_{GS}\approx 3500,\qquad k_{\omega_{\rm opt}}\approx 110.
\]
So for this modest grid, optimally relaxed SOR needs roughly $64\times$ fewer iterations than Jacobi to reach the same accuracy, which is a direct, numerical illustration of the asymptotic rates above.
These are \emph{asymptotic estimates}, not guarantees: $G_{\omega_{\rm opt}}$ is non-normal
($\|G_{\omega_{\rm opt}}\|_2\approx2.78$), so the error can grow before it decays. Indeed
$\rho^{110}\approx9.8\cdot10^{-7}$ while $\|G_{\omega_{\rm opt}}^{110}\|_2\approx7.2\cdot10^{-5}$, and
in practice $k=148$ is needed to guarantee $10^{-6}$ for every initial error. 
\end{example}

\begin{takeaway}
 \begin{itemize}
\item Jacobi converges very slowly: $\rho(G_J)\to 1$ as $n\to\infty$, with rate $\mathcal{O}(1/n^2)$.
\item Gauss--Seidel roughly doubles the effective rate, but is still $\mathcal{O}(1/n^2)$.
\item SOR with optimal relaxation improves to $\mathcal{O}(1/n)$ per iteration, but this is still poor compared to modern solvers (e.g.\ multigrid or Krylov methods).
\end{itemize}
Thus none of these methods are suitable as standalone solvers for large PDE systems.
\end{takeaway}

\section{Basic Iterative Methods for Graph Laplacians}

Graph Laplacians defined in Chapter 2 provide another fundamental class of sparse matrices.
We recall that for any graph, the smallest eigenvalue of $L$ is given by $0$, with corresponding eigenvector $\mathbf{e} = [1,1,1,\dots,1]^T$ i.e. the vector with all unit entries. This can be seen as by definition, each row of the graph Laplacian matrix sums to $0$. 

Below a few examples with corresponding spectral properties

\begin{tikzpicture}[scale=1,
  every node/.style={circle,draw,fill=white,inner sep=1pt,minimum size=6pt},
  edge/.style={-}]
  \def\n{6}              
  \pgfmathtruncatemacro{\m}{\n-1}

\coordinate (Pstart) at (0,0);
\foreach \i in {1,...,\n} {
  \node (p\i) at ($(Pstart)+(\i*0.9,0)$) {};
}
\foreach \i in {1,...,\m} {
  \pgfmathtruncatemacro{\next}{\i+1}
  \draw (p\i) -- (p\next);
}
  \node[draw=none,fill=none] at ({0.9*(\n+1)/2},-0.9){$P_{\n}$};

  \def\r{1.2}
  \coordinate (Ccenter) at (7,0);
  \foreach \k in {1,...,\n} {
    \node (c\k) at ({7+\r*cos(360*\k/\n)}, {\r*sin(360*\k/\n)}) {};
  }
  \foreach \k in {1,...,\n} {
    \pgfmathtruncatemacro{\next}{mod(\k,\n)+1}
    \draw (c\k) -- (c\next);
  }
  \node[draw=none,fill=none] at (7,-2) {$C_{\n}$};

  \coordinate (Scenter) at (12,0);
  \node (hub) at (Scenter) {};
  \foreach \k in {1,...,\m} {
    \node (s\k) at ({12 + \r*1.4*cos(360*\k/\m)}, {\r*1.4*sin(360*\k/\m)}) {};
    \draw (hub) -- (s\k);
  }
  \node[draw=none,fill=none] at (12,-2.2) {$K_{1,\m}$};
\end{tikzpicture}

\paragraph{Path graph $P_n$.} 
Here $L$ is a rank$-2$ perturbation of $T=\operatorname{tridiag}(-1,2,-1)$, the \emph{unscaled}
matrix of \eqref{eq:Poisson disc matrix} (Chapter~2). We can bound its eigenvalues with the following result

\begin{tcolorbox}[colback=yellow!5!white,colframe=yellow!50!black]
\begin{theorem}[\cite{wilkinsonAlgebraicEigenvalueProblem1965}]\label{thm:Wilkinson}
Consider two symmetric matrices $S_1,S_2 \in \mathbb{R}^{n\times n}$. Let $\lambda_1\ge\lambda_2\ge \dots \ge \lambda_n$.  Then the $k$th eigenvalue of the matrix $S_1+S_2$ satisfies
\[\lambda_k(S_1) + \lambda_n(S_2) \le \lambda_k(S_1+S_2) \le \lambda_k(S_1)+\lambda_1(S_2)\]
\end{theorem}
\end{tcolorbox}

\begin{tcolorbox}[colback=yellow!5!white,colframe=yellow!50!black]
\begin{corollary}[Bound on eigenvalues of $L$ for $P_n$]
Let $T=h^2A=\operatorname{tridiag}(-1,2,-1)$ be as above. Then $L = T + D$, where $D$ is the diagonal matrix with
$D(1,1)=D(n,n)=-1$ and all other entries zero: the two ends of the path are free, where the boundary
value problem has Dirichlet conditions. Order eigenvalues decreasingly,
$\lambda_1\ge\cdots\ge\lambda_n$, so that $\lambda_1(D)=0$ and $\lambda_n(D)=-1$. Applying
Theorem~\ref{thm:Wilkinson} to $T+D$ and then to $D+T$ gives
\[
\max\big(\lambda_k(T)-1,\;\lambda_n(T)+\lambda_k(D)\big)
\;\le\;\lambda_k(L)\;\le\;
\min\big(\lambda_k(T),\;\lambda_1(T)+\lambda_k(D)\big).
\]

In particular the extreme eigenvalues are bounded by:
\[\lambda_{min}(T)-1\le \lambda_{min}(L) \le \min(\lambda_{min}(T),\lambda_{max}(T)-1),\]

\[ \max(\lambda_{max}(T)-1,\lambda_{min}(T))\le \lambda_{max}(L)\le \lambda_{max}(T) \]
\end{corollary}
\end{tcolorbox}

\paragraph{Cycle graph $C_n$.} 
Eigenvectors are Fourier modes; eigenvalues are 
$\lambda_k = 2-2\cos(2\pi k/n)$. Jacobi eigenvalues are $\cos(2\pi k/n)$. 

\paragraph{Star graph $K_{1,n-1}$.} The spectrum of the graph Laplacian is  $\{0,1(\times n-2),n\}$.\\

For real graph algorithms, one typically turns to more advanced methods. Still, the classical iterations remain a 
useful ``microscope'' for understanding graph structure. They also have clear interpretations from the perspective of classical graph theory, as we shall see later.

\begin{example}[Jacobi = random walk on a path graph]
Consider the path $P_4$ with vertices $\{1,2,3,4\}$. 
The adjacency matrix and degree matrix are
\[
A=\begin{bmatrix}
0 & 1 & 0 & 0\\
1 & 0 & 1 & 0\\
0 & 1 & 0 & 1\\
0 & 0 & 1 & 0
\end{bmatrix},\qquad 
D=\mathrm{diag}(1,2,2,1).
\]
The Jacobi iteration matrix is
\[
G_J = D^{-1}A =
\begin{bmatrix}
0 & 1 & 0 & 0\\
\tfrac{1}{2} & 0 & \tfrac{1}{2} & 0\\
0 & \tfrac{1}{2} & 0 & \tfrac{1}{2}\\
0 & 0 & 1 & 0
\end{bmatrix}.
\]

\smallskip
\noindent
\emph{Interpretation.} Each row of $G_J$ gives the transition 
probabilities of a random walk:
\begin{itemize}
\item From vertex $1$, you move to vertex $2$ with probability $1$.
\item From vertex $2$, you move to $1$ or $3$ with equal probability $1/2$.
\item From vertex $3$, you move to $2$ or $4$ with probability $1/2$.
\item From vertex $4$, you move to vertex $3$ with probability $1$.
\end{itemize}

A \emph{function} on the vertices is averaged over neighbours by
$P:=G_J$ - this is the Jacobi iteration: from $x^{(0)}=[1,0,0,0]^{\top}$,
\[
x^{(1)} = P x^{(0)} = [0,\tfrac{1}{2},0,0]^{\top}, \quad
x^{(2)} = P x^{(1)} = [\tfrac{1}{2},0,\tfrac{1}{4},0]^{\top},
\]
and so on. 
\end{example}
\begin{takeaway}
Jacobi iteration on $Lx=b$ is local averaging along the edges of the underlying graph, with the
eigenvalues describing the rates. Damping turns it into a diffusion: since
$G_\omega=I-\omega D^{-1}L=(1-\omega)I+\omega P$, damped Jacobi with $\omega=\tfrac12$ is exactly the
\emph{lazy random walk} $P_{1/2}=\tfrac12(I+P)$.
\end{takeaway}

\paragraph{Relation to spectral clustering.}
The slowest-decaying Jacobi modes are precisely the eigenvectors of the Laplacian 
associated with the smallest nonzero eigenvalues. 
These eigenvectors are smooth signals on the graph: they vary little across densely connected regions and change more abruptly across weak connections.  In spectral graph theory, this property is exploited to detect communities \cite{vonluxburgTutorialSpectralClustering2007}.
For instance, the \emph{Fiedler vector} \cite{fiedlerAlgebraicConnectivity1973} (the eigenvector associated with the second-smallest Laplacian eigenvalue) typically has nearly constant values on different clusters and a sharp transition across the cut between them.  Thresholding this vector provides a natural bipartition of the graph.

\begin{example}[Jacobi iteration and community detection]\label{ex:twotriangles}
Consider the graph consisting of two triangles connected by a single edge:  
vertices $\{1,2,3\}$ fully connected to each other, 
vertices $\{4,5,6\}$ fully connected to each other, 
and one edge $(3,4)$ connecting the two groups.

\begin{center}
\begin{tikzpicture}[scale=1.2,
  every node/.style={circle,draw,fill=white,inner sep=1pt,minimum size=6pt,font=\small},
  edge/.style={-}]
  
  \begin{scope}[on background layer]
    \fill[blue!10,rounded corners] (-0.5,-1.5) rectangle (1.5,1.5); 
    \fill[red!10,rounded corners]  (2.5,-1.5) rectangle (4.5,1.5);  
  \end{scope}
  
  \node (1) at (0,0) {1};
  \node (2) at (1,1) {2};
  \node (3) at (1,-1) {3};
  \draw[edge] (1)--(2)--(3)--(1);
  
  \node (4) at (3,0) {4};
  \node (5) at (4,1) {5};
  \node (6) at (4,-1) {6};
  \draw[edge] (4)--(5)--(6)--(4);
  
  \draw[edge,thick] (3)--(4);
\end{tikzpicture}
\end{center}

\paragraph{Graph Laplacian construction.}
The adjacency matrix is
\[
A = \begin{bmatrix}
0 & 1 & 1 & 0 & 0 & 0 \\
1 & 0 & 1 & 0 & 0 & 0 \\
1 & 1 & 0 & 1 & 0 & 0 \\
0 & 0 & 1 & 0 & 1 & 1 \\
0 & 0 & 0 & 1 & 0 & 1 \\
0 & 0 & 0 & 1 & 1 & 0
\end{bmatrix},
\]
and the degree matrix is
\[
D = \mathrm{diag}(2,\,2,\,3,\,3,\,2,\,2).
\]
Thus the Laplacian is
\[
L = D - A = 
\begin{bmatrix}
2 & -1 & -1 & 0 & 0 & 0 \\
-1 & 2 & -1 & 0 & 0 & 0 \\
-1 & -1 & 3 & -1 & 0 & 0 \\
0 & 0 & -1 & 3 & -1 & -1 \\
0 & 0 & 0 & -1 & 2 & -1 \\
0 & 0 & 0 & -1 & -1 & 2
\end{bmatrix}.
\]
\emph{Spectrum and Fiedler vector.}
For the $L$ above, $\lambda(L)\approx \{\,0,\;0.438447,\;3,\;3,\;3,\;4.561553\,\},$
so $\lambda_2\approx 0.438447$ is smallest nonzero. An associated Fiedler vector is
\[
v_{\mathrm{Fiedler}} \approx [-0.465,\,-0.465,\,-0.261,\;0.261,\;0.465,\;0.465]^{\top}.
\]
It has a sign change across the bridge $(3,4)$. Thresholding at $0$ (or performing a sweep cut) yields
the two clusters.

\emph{Two Fiedler vectors, and they are not the same.} The vector above belongs to the
\emph{combinatorial} problem $Lv=\lambda v$. The iteration of the next section instead solves the
\emph{generalized} problem $Lv=\lambda Dv$ whose
second eigenvector is
\[
v_{\mathrm{Fiedler}}^{\mathrm{norm}}=D^{-1/2}v_2(\mathcal L_{\rm sym})
\approx[-0.461,\,-0.461,\,-0.273,\;0.273,\;0.461,\;0.461]^{\top}.
\]
On this graph, the two differ and induce the same bipartition, but they are
different vectors, and on graphs with a wide spread of degrees they can differ more substantially. We will
be careful below to say which of the two an iteration actually produces.
\end{example}

\medskip
\noindent
\paragraph{Damped Jacobi as both power iteration and low-pass filter.}
Let $G=(V,E)$ be an undirected graph with adjacency $A$, degree matrix $D$, and
random walk matrix $P:=D^{-1}A$. Define the symmetric normalized adjacency
$S:=D^{-1/2}AD^{-1/2}$ and the normalized Laplacians
$\mathcal L_{\mathrm{rw}}:=I-P$ and $\mathcal L_{\mathrm{sym}}:=I-S$.
We have the similarity $P=D^{-1/2} S D^{1/2}$, hence $P$ and $S$ share eigenvalues
$1=\mu_1\ge \mu_2\ge\cdots\ge \mu_n\ge -1$. Let
$\{v_i\}_{i=1}^n$ be an orthonormal basis of eigenvectors of $S$ ($Sv_i=\mu_i v_i$),
so $v_1=\frac{D^{1/2}\mathbf 1}{\|D^{1/2}\mathbf 1\|_2}$ and
$\lambda_i^{\mathrm{sym}}=1-\mu_i$ are the eigenvalues of $\mathcal L_{\mathrm{sym}}$.

\medskip
\noindent\emph{Low–pass view.}
Damped Jacobi with $b=0$ is $x^{(k+1)}=P_{1/2}\,x^{(k)}$, so with $y^{(k)}:=D^{1/2}x^{(k)}$,
\[
y^{(k)} \;=\; \Bigl(\tfrac{I+S}{2}\Bigr)^{k} y^{(0)}
\quad\Longrightarrow\quad
x^{(k)} \;=\; D^{-1/2}\!\sum_{i=1}^n \Bigl(\tfrac{1+\mu_i}{2}\Bigr)^{k}\,( D^{1/2}x^{(0)},v_i)\,v_i.
\]
Each spectral component is therefore multiplied by
$\bigl(\tfrac{1+\mu_i}{2}\bigr)^{k}=\bigl(1-\tfrac{\lambda_i^{\mathrm{sym}}}{2}\bigr)^{k}$, a factor
lying in $[0,1]$ and \emph{decreasing} in $\lambda_i^{\mathrm{sym}}$: high graph frequencies are
attenuated fastest and smooth modes persist. Damped Jacobi is thus a \emph{low-pass} graph filter (a
filter which \textit{removes} high-frequency components and permits low-frequency components to pass
through). In the undamped Jacobi the modes are multiplied by $\mu_i^{\,k}$ instead, and
$|\mu_i|\to1$ as $\lambda_i^{\mathrm{sym}}\to2$, so the most oscillatory ones are not attenuated at
all.

\medskip
\paragraph{Power method view}

The low-pass picture has a classical counterpart. Let $M\in\mathbb{R}^{n\times n}$ have a
\emph{dominant} eigenvalue, that is, let its eigenvalues be ordered as
$|\lambda_1|>|\lambda_2|\ge\cdots\ge|\lambda_n|$, so that $\rho(M)=|\lambda_1|$. Multiplying a vector
repeatedly by $M$ amplifies its component along the dominant eigenvector $v_1$ more than any other, so
that after rescaling at each step the iterate lines up with $v_1$. This is the \emph{power method}, the
oldest of all eigenvalue iterations, and the simplest predecessor of the Lanczos method of
Chapter~4.

\begin{algorithm}[h!]
\caption{Power method}\label{alg:power}
\begin{algorithmic}[1]
\REQUIRE Matrix $M\in\mathbb{R}^{n\times n}$ with a dominant eigenvalue, starting vector $y^{(0)}$ with
$\|y^{(0)}\|_2=1$ and a nonzero component along $v_1$, tolerance $\varepsilon>0$.
\FOR{$k = 0,1,2,\dots$}
  \STATE $z \gets M\,y^{(k)}$ \hfill \{one matrix--vector product per step\}
  \STATE $\mu_k \gets (y^{(k)})^\top z$ \hfill \{Rayleigh quotient, since $\|y^{(k)}\|_2=1$\}
  \IF{$\|z-\mu_k\,y^{(k)}\|_2 \le \varepsilon$}
    \STATE \textbf{stop}
  \ENDIF
  \STATE $y^{(k+1)} \gets z/\|z\|_2$
\ENDFOR
\ENSURE Approximate dominant eigenpair $(\mu_k,\,y^{(k)})$.
\end{algorithmic}
\end{algorithm}

Provided $y^{(0)}$ has a nonzero component along $v_1$, the iterates converge (up to sign) to the
dominant eigenvector, and the Rayleigh quotients $\mu_k$ converge to $\lambda_1$. The rate is governed
by the ratio of the two largest eigenvalue moduli: the eigenvector error decays like
$|\lambda_2/\lambda_1|^k$ and the eigenvalue error, for symmetric $M$, twice as fast. Slow convergence
is therefore expected when the two largest eigenvalues are not well separated. See Problem~8 for details. 

\medskip
\noindent\textbf{Deflating the trivial mode to reveal $v_2$.}
The lazy walk $P_{1/2}$ has dominant eigenvalue $1$ with right eigenvector $\mathbf 1$.
To find the first eigenpair different from $(1,\mathbf 1)$ we \emph{project it out} after each step,
in the $D$-inner product $\langle u,v\rangle_D:=u^\top D v$, and only then normalize:
\[
\hat x^{(k+1)}=P_{1/2}\,x^{(k)}
-\frac{\langle P_{1/2}\,x^{(k)},\mathbf 1\rangle_D}{\langle \mathbf 1,\mathbf 1\rangle_D}\,\mathbf 1,
\qquad
x^{(k+1)}=\frac{\hat x^{(k+1)}}{\|\hat x^{(k+1)}\|_D}.
\]
(The projection alone does not normalize anything; the two steps are separate.)
Setting $y^{(k)}:=D^{1/2}x^{(k)}$ and expanding $y^{(0)}=\sum_{i\ge2}\alpha_i v_i$, the iterates satisfy
$$
y^{(k)}\;\propto\;\sum_{i\ge2}\alpha_i\bigl(\tfrac{1+\mu_i}{2}\bigr)^{k}v_i ,
$$
the omitted factor being the product of the normalizations. All these factors are \emph{nonnegative},
so the largest of them is simply the one corresponding to $\mu_2$: if $G$ is connected, $\mu_2$ is
simple, and the starting vector has a nonzero component along $v_2$ (that is, $\alpha_2\neq0$), then
\[
x^{(k)}\longrightarrow \pm D^{-1/2}v_2
\qquad\text{at rate}\qquad
\Bigl(\frac{1+\mu_3}{1+\mu_2}\Bigr)^{k}.
\]
The limit is the \emph{normalized} Fiedler direction of Example~\ref{ex:twotriangles}, the second
eigenvector of $Lv=\lambda Dv$. The hypothesis is the same one the power method always needs, a dominant eigenvalue, here on the
subspace $\mathbf 1^{\perp_D}$. Note that this is again a place where the damping is doing real work:
for the undamped walk the surviving mode is the one of largest \emph{modulus} among
$\mu_2,\dots,\mu_n$, which need not be $\mu_2$ at all (Problem~8(e)). Thus \emph{damped Jacobi = power iteration}, and with deflation it converges to the normalized Fiedler
direction.

\begin{remark}[How much damping? $\omega=\tfrac12$ is safe, not optimal]
Writing $P_\omega=(1-\omega)I+\omega P$, the eigenvalues are
$1-\omega+\omega\mu_i$ and the deflated iteration converges at the rate
\(
r(\omega)=\frac{\max_{i\ge3}\bigl|1-\omega+\omega\mu_i\bigr|}{1-\omega+\omega\mu_2}.
\)
For $0<\omega\le\tfrac12$ every $1-\omega+\omega\mu_i$ is nonnegative \emph{whatever the graph}, since
$\mu_i\ge-1$; so $\omega=\tfrac12$ is the best choice available without knowing the
spectrum. If $\mu_3$ and $\mu_n$ are known, balancing the two extreme unwanted modes gives
\(
\omega_\star=\min\Bigl\{1,\ \frac{2}{2-\mu_3-\mu_n}\Bigr\},
\)
which is faster still. The lazy walk is the robust default, not the optimum.
\end{remark}

\begin{takeaway}
\begin{itemize}  
\item Classical iterations such as Jacobi act as diffusion or averaging processes on graphs.  
\item Their eigenmode analysis shows that, once damped, high--frequency oscillations on the graph are quickly removed while smooth global modes persist.  
\item The slow modes correspond to eigenvectors of the Laplacian with small eigenvalues; these capture global connectivity patterns and reveal community structure.  
\item This perspective connects basic iterative methods to spectral clustering: Jacobi iterations naturally enhance low--frequency information that underlies graph partitioning.  
\item In practice their convergence is too slow to serve as efficient graph solvers, but they remain conceptually valuable as ``microscopes'' for understanding spectral structure.  
\end{itemize}
\end{takeaway}

\section{Basic Iterative Methods in Machine Learning}

The linear algebra perspective on Jacobi, Gauss--Seidel, and SOR 
extends naturally to optimization and machine learning.  
At their core, many ML algorithms amount to solving large linear or 
least-squares problems via iterative schemes.  
This section shows how the classical iterations appear as special 
cases of gradient descent, and how their spectral behaviour connects 
to convergence properties of training.

\subsection{Richardson iteration and gradient descent}

Recall the Richardson scheme for $Ax=b$ with step size $\tau>0$:
\[
x^{(k+1)} = x^{(k)} + \tau(b - Ax^{(k)}).
\]
This is exactly gradient descent for the quadratic loss
\[
\min_x f(x) = \tfrac{1}{2}x^{\top}Ax - b^{\top}x,
\]
since $\nabla f(x) = Ax-b$.  
Thus gradient descent on quadratic objectives is nothing but
Richardson iteration.

\paragraph{Spectral behaviour.}
If $A$ is symmetric positive definite with eigenpairs 
$(\lambda_j,v_j)$, then the error evolves as
\[
e^{(k)} = \sum_j (1-\tau\lambda_j)^k \,\alpha_j v_j.
\]
Hence convergence requires $0<\tau<2/\lambda_{\max}$, and the rate
is controlled by the condition number $\kappa(A)=\lambda_{\max}/\lambda_{\min}$ \cite{nesterovIntroductoryLecturesConvex2004}:
the closer eigenvalues are to zero, the slower their modes decay.
This is the same ``bottleneck'' phenomenon seen for smooth modes
in PDEs and graph Laplacians.

\begin{center}
\begin{tikzpicture}[xscale=1.1,yscale=1.6]
\draw[->] (-0.1,0) -- (6.6,0) node[right]{\small iteration $k$};
\draw[->] (0,-0.02) -- (0,1.05) node[above]{\small mode amplitude};
\draw[thick,domain=0:6,samples=140] plot (\x,{exp(-0.9*\x)});
\draw[thick,gray,domain=0:6,samples=140] plot (\x,{exp(-0.18*\x)});
\node[anchor=south west] at (0.35,0.35) {\small large $\lambda_j$ (fast)};
\node[anchor=south west,gray] at (3.4,0.55) {\small small $\lambda_j$ (slow)};
\end{tikzpicture}
\end{center}

\paragraph{Choosing the step size $\tau$.} Each mode $j$ is damped at rate $g_j(\tau)=|1-\tau\lambda_j|$, so the overall (worst-case) rate is $q(\tau)=\max_j g_j(\tau)$, which for $\lambda_j\in[\lambda_{\min},\lambda_{\max}]$ reduces to $q(\tau)=\max\{|1-\tau\lambda_{\min}|,|1-\tau\lambda_{\max}|\}$. This is exactly the Richardson-for-SPD-matrices result of Chapter~1: the minimax-optimal step
\[
\tau^\star=\frac{2}{\lambda_{\min}+\lambda_{\max}}
\quad\text{equalizes the two extreme mode gains and gives}\quad
q(\tau^\star)=\frac{\kappa-1}{\kappa+1}, \kappa=\frac{\lambda_{\max}}{\lambda_{\min}}.
\]

\begin{example}[Optimal step size for a two-eigenvalue toy problem]
Suppose $A$ (or $X^\top X$) has only two distinct eigenvalues, $\lambda_{\min}=1$ and $\lambda_{\max}=9$, so $\kappa=9$. The optimal step is $\tau^\star=2/(1+9)=0.2$, giving mode gains $|1-\tau^\star\lambda_{\min}|=|1-0.2|=0.8$ and $|1-\tau^\star\lambda_{\max}|=|1-1.8|=0.8$: both extreme modes decay at exactly the same rate $q^\star=(\kappa-1)/(\kappa+1)=8/10=0.8$, as predicted. A poorly chosen step, e.g.\ $\tau=2/\lambda_{\max}=2/9\approx0.222$, leaves the $\lambda_{\max}$ mode exactly on the boundary of divergence ($g=|1-2|=1$) while the $\lambda_{\min}$ mode still decays comfortably ($g=|1-0.222|\approx0.778$) -- a reminder that the largest eigenvalue alone does not determine the best step size; both extremes of the spectrum must be balanced.
\end{example}

\paragraph{Connection to Jacobi.}
Jacobi can be viewed as Richardson with a preconditioner $D^{-1}$:
\[
x^{(k+1)} = x^{(k)} + D^{-1}(b-Ax^{(k)}).
\]
This interpretation carries over directly to ML: preconditioning
accelerates gradient descent by scaling the update direction.

\subsection{Applications in machine learning}

\paragraph{(1) Linear regression.}
Given data $X\in\mathbb{R}^{m\times n}$ and labels $y\in\mathbb{R}^m$,
the least squares problem
\[
\min_w \tfrac{1}{2}\|Xw-y\|^2
\]
leads to the normal equations $X^{\top}Xw=X^{\top}y$.  
Richardson/gradient descent updates are
\[
w^{(k+1)} = w^{(k)} - \tau\big(X^{\top}Xw^{(k)} - X^{\top}y\big).
\]
The spectrum of $X^{\top}X$ (squared singular values of $X$) dictates
the convergence speed.  
If $X$ is ill-conditioned, convergence is slow, mirroring the 
PDE case where low frequencies persist.  
This motivates preconditioning and methods like stochastic gradient descent (SGD).

\paragraph{(2) Logistic regression and nonlinear losses.}
For convex but nonlinear losses, e.g.
\[
f(w) = \sum_{i=1}^m \log\!\big(1+\exp(-y_i x_i^{\top}w)\big),
\]
gradient descent is no longer exactly Richardson,
but locally (near the solution) the Hessian $H=\nabla^2 f(w^\ast)$
plays the role of $A$.  
Convergence is again governed by the spectrum of $H$:
flat directions (small eigenvalues) cause slow learning,
steep directions (large eigenvalues) require small step sizes 
for stability.

\subsection{Smoothing and spectral bias}

Just as Jacobi smooths high-frequency error modes in PDEs,
gradient descent smooths high-frequency parameter components. In training neural networks this is visible as \emph{spectral bias} \cite{rahamanSpectralBiasNeuralNetworks2019}: early iterations capture low-frequency/global patterns in the data, while fine details emerge later.

\begin{example}[Smoothing vs.\ spectral bias] Consider two settings: \\
\noindent \emph{(a) Jacobi for the 1D Poisson problem.}  
Let $A$ be the Laplacian matrix on $n$ grid points.  
The Jacobi error iteration reads
\[
e^{(k+1)} = G_J e^{(k)}, \qquad 
\alpha_j^{(k)} = \mu_j^k \alpha_j^{(0)}, \quad 
\mu_j = \cos\!\Big(\tfrac{\pi j}{n+1}\Big).
\]
The eigenvectors are sine waves with frequency $j$. Here $|\mu_{n+1-j}|=|\mu_j|$, so the most
oscillatory mode decays exactly as slowly as the smoothest one. Damping repairs
this. For $G_\omega=I-\omega D^{-1}A$ the eigenvalues are
$$
\mu_j(\omega)=1-2\omega\sin^2\bigl(\tfrac{j\pi}{2(n+1)}\bigr),
$$
and $\omega=\tfrac23$ gives
$|\mu_j(\omega)|\le\tfrac13$ for every $j\ge(n+1)/2$, uniformly in $n$.
Thus \emph{weighted} Jacobi acts as a \emph{low--pass smoother}: it removes oscillations first and
leaves global structure last. This is the property multigrid rests on (Chapter~6).

\medskip
\noindent
\emph{(b) Gradient descent for least squares regression.}  
Suppose we fit $Xw \approx y$ with normal equations $X^\top X w = X^\top y$.  
Let $X^\top X = V\Lambda V^\top$ with eigenpairs $(\lambda_j,v_j)$.  
Gradient descent with step size $\tau$ yields
\[
\alpha_j^{(k)} = (1-\tau\lambda_j)^k \alpha_j^{(0)}.
\]
For many data distributions (e.g. Fourier features, random features, kernel regression), it turns out that {\bf low-frequency patterns} correspond to larger eigenvalues of the kernel or data covariance, while high-frequency patterns correspond to smaller eigenvalues. 
This phenomenon is known as \emph{spectral bias}.  

\medskip
\noindent
\emph{Duality.}  
Both Jacobi and gradient descent share the same filtering mechanism: \(
\alpha_j^{(k)} = g(\lambda_j)^k \,\alpha_j^{(0)}\),
where $\lambda_j$ indexes spectral modes.  
In PDE solvers, large $\lambda_j$ corresponds to high--frequency error and is damped quickly (smoothing).  
In learning, large $\lambda_j$ corresponds to low--frequency patterns in the data, which are learned quickly.
Thus the \emph{roles of low vs.\ high frequency are swapped}, but the underlying principle is the same: iterative methods act as frequency filters, with the spectrum of $A$ or $X^\top X$ dictating which modes persist.
\end{example}

\clearpage

\section{Exercises}

\paragraph{Problem 1 (Closed form of the 1D Gauss--Seidel iteration matrix).}
Let
\[
A=\frac{1}{h^2}\,\operatorname{tridiag}(-1,\,2,\,-1)\in\mathbb{R}^{n\times n},
\qquad
A=D-L-U,
\]
where the Gauss-Seidel splitting uses nonnegative strict parts
\[
D=\frac{2}{h^2}I,\quad
L=\frac{1}{h^2}\,\operatorname{tridiag}(1,0,0),\quad
U=\frac{1}{h^2}\,\operatorname{tridiag}(0,0,1).
\]
The Gauss--Seidel iteration matrix is \(G_{GS}=(D-L)^{-1}U\).
\begin{enumerate}
\item Show that \(G_{GS}\) is independent of \(h\) and can be written as
\[
G_{GS}=B^{-1}S_{+},\qquad
B:=2I-S_{-},
\]
where \(S_{-}\) has ones on the subdiagonal and \(S_{+}\) has ones on the superdiagonal.
\item Compute \(B^{-1}\) explicitly and deduce a closed form for the entries of \(G_{GS}\).
\item Write the first few rows of \(G_{GS}\) to display its “plain matrix” pattern.
\end{enumerate}

\paragraph{Problem 2 (Error components for Jacobi).}
Let $A$ be the 1D Poisson matrix with $n$ interior points and
$G_J$ the Jacobi iteration matrix.
\begin{itemize}
\item[(a)] Show that the error $e^{(k)} = x^{(k)}-x^\ast$ can be expanded in the eigenbasis
\[
e^{(k)} = \sum_{j=1}^n \alpha_j^{(k)} v^{(j)}, \qquad
\alpha_j^{(k+1)} = \mu_j \alpha_j^{(k)},
\]
with eigenvalues $\mu_j = \cos\!\left(\tfrac{j\pi}{n+1}\right)$.
\item[(b)] Which modes (low $j$ or high $j$) have eigenvalues closest to~1?
\item[(c)] Deduce that \emph{undamped} Jacobi is not a smoother, and that $\omega=\tfrac23$ in
$G_\omega=I-\omega D^{-1}A$ gives $|\mu_j(\omega)|\le\tfrac13$ for all $j\ge (n+1)/2$.
\end{itemize}

\paragraph{Preamble (for Problems 3 and 4, exam 2021/2022).}

For $a_0,a_{-1},a_1 \in \mathbb{R}$ with $a_{-1}\cdot a_1>0$ 
the Toeplitz tridiagonal matrix
\[
A =
\begin{bmatrix}
a_0 & a_1 \\
a_{-1} & a_0 & a_1 \\
& \ddots & \ddots & \ddots \\
& & a_{-1} & a_0 & a_1 \\
& & & a_{-1} & a_0
\end{bmatrix}\in\mathbb{R}^{n\times n}
\]
has eigenvalues
\[
\lambda_p = a_0 + 2 \text{sgn}(a_{-1})\sqrt{a_{-1}a_1}\cos\!\left(\frac{p\pi}{n+1}\right), \qquad p=1,\dots,n,
\]
and associated eigenvectors
\[
s_p = \left(\sqrt{\tfrac{a_{-1}}{a_1}}^{\,i}\, \sin\!\bigg(\tfrac{i p \pi}{n+1}\bigg)\right)_{i=1}^n.
\]
If $A$ is symmetric ($a_{-1}=a_1$), then the eigenvectors form an orthogonal basis.
Moreover, for any matrix $B$, 
\[
\lim_{k\to\infty} B^k = 0 \quad \Longleftrightarrow \quad \rho(B)<1,
\]
where $\rho(B)$ denotes the spectral radius.

\paragraph{Problem 3 (Tridiagonal Laplace matrix, exam 2021/2022).}
Consider the symmetric tridiagonal matrix
\[
A = 
\begin{bmatrix}
2 & -1 \\
-1 & 2 & -1 \\
   & \ddots & \ddots & \ddots \\
   &        & -1 & 2 & -1 \\
   &        &    & -1 & 2
\end{bmatrix}\in \mathbb{R}^{n\times n}.
\]
Assume $n$ is odd.
\begin{enumerate}[label=(\alph*)]
\item Show that the vector $b\in\mathbb{R}^n$ with entries
\[
b_{2k-1}=(-1)^{k-1},\qquad b_{2k}=0,
\qquad\text{that is}\qquad b = [1,0,-1,0,1,0,-1,\dots]^\top ,
\]
is an eigenvector of $A$. (Its last entry is $+1$ only when $n\equiv1\bmod 4$.)
\item Show that all eigenvalues of $A$ lie in the interval $(0,4)$.
\end{enumerate}

\paragraph{Problem 4 (Jacobi's method and convergence, exam 2021/2022).}
Let $A\in\mathbb{R}^{n\times n}$ be nonsingular and let $D=\mathrm{diag}(A)$.
Jacobi’s iteration for solving $Ax=b$ is
\begin{equation}\label{eq:jacobi}
x^{(k+1)}=(I-D^{-1}A)x^{(k)}+D^{-1}b,\qquad k=0,1,2,\dots
\end{equation}
Let $B=I-D^{-1}A$.

\begin{enumerate}[label=(\alph*)]
\item Show that if $\hat x=\lim_{k\to\infty}x^{(k)}$ exists, then Jacobi’s method converges if and only if $A$ has no zeros on its diagonal and $\rho(B)<1$.
\item Suppose $A$ is the Toeplitz tridiagonal matrix above. Give necessary and sufficient conditions on $(a_{-1},a_0,a_1)$ independent of $n$ for Jacobi’s method to converge.
\item Consider the special case
\[
A=
\begin{bmatrix}
a_0 & a_1 \\
a_0 & a_0 & a_1 \\
    & \ddots & \ddots & \ddots \\
    &        & a_0 & a_0 & a_1 \\
    &        &     & a_0 & a_0
\end{bmatrix}\in\mathbb{R}^{n\times n}.
\]
What are the necessary and sufficient conditions on $a_0,a_1$ (independent of $n$) for $\rho(B)<1$?
\item Let $A$ be as in (c) with $a_0=1$ and $a_1>0$ (so the preamble's formula applies to $B$), $b=e_n$, and initial guess $x^{(0)}=0$. Does there exist a Jacobi iterate $x^{(k)}$ that solves $Ax=b$ exactly?
\end{enumerate}

\paragraph{Problem 5 (Numerical verification of convergence rates).}
For the 1D Poisson matrix with $n=19$ interior points:
\begin{enumerate}
\item[(a)] Compute $\rho(G_J)$, $\rho(G_{GS})$, and $\omega_{\rm opt}$, $\rho(G_{\omega_{\rm opt}})$ using the closed-form expressions derived in this chapter.
\item[(b)] Estimate the number of iterations each method needs to reduce the initial error by a factor $10^{-8}$.
\item[(c)] Implement Jacobi, Gauss--Seidel, and SOR (with $\omega=\omega_{\rm opt}$) for this system and confirm numerically that the observed asymptotic convergence factors match your predictions from (a). (See \nbpath{week3/script\_week3\_1.ipynb} in the course code repository.)
\end{enumerate}

\paragraph{Problem 6 (Spectrum of a small path graph).}
Let $P_5$ be the path graph on $5$ vertices with graph Laplacian $L$ (as constructed in Chapter~2).
\begin{enumerate}
\item[(a)] Write down $L$ explicitly and compute its eigenvalues and eigenvectors directly (by hand or numerically).
\item[(b)] Verify that they match the general formula $\lambda_k(L)=2-2\cos(\pi k/n)$ for the path graph.
\item[(c)] Identify the Fiedler vector (the eigenvector for the smallest nonzero eigenvalue) and describe the bipartition it induces on the $5$ vertices.
\end{enumerate}

\paragraph{Problem 7 (Gradient descent convergence rate from the data spectrum).}
Consider a least-squares problem $\min_w \tfrac12\|Xw-y\|^2$ where $X\in\mathbb{R}^{m\times n}$ has singular values $\sigma_1=4$, $\sigma_2=2$, $\sigma_3=1$ (so $X^\top X$ has eigenvalues $16,4,1$).
\begin{enumerate}
\item[(a)] Compute the condition number $\kappa(X^\top X)$ and the largest admissible step size $\tau_{\max}=2/\lambda_{\max}$ for gradient descent on $X^\top X w = X^\top y$.
\item[(b)] Compute the minimax-optimal step size $\tau^\star$ and the resulting worst-case convergence factor $q^\star=(\kappa-1)/(\kappa+1)$.
\item[(c)] Using $\tau^\star$, estimate how many iterations are needed to reduce the worst-case mode's error by a factor of $100$.
\item[(d)] Which mode (associated with $\sigma_1$, $\sigma_2$, or $\sigma_3$) is learned fastest, and which slowest? Relate your answer to the notion of \emph{spectral bias}.
\end{enumerate}

\paragraph{Problem 8 (Convergence of the power method).}
Let $M\in\mathbb{R}^{n\times n}$ be symmetric, with orthonormal eigenvectors $v_1,\dots,v_n$ and
eigenvalues ordered by modulus, $|\lambda_1|>|\lambda_2|\ge\cdots\ge|\lambda_n|$. Run
Algorithm~\ref{alg:power} from $y^{(0)}=\sum_{i=1}^n\alpha_i v_i$ with $\|y^{(0)}\|_2=1$ and
$\alpha_1\neq0$, and let $\theta_k\in[0,\pi/2]$ be the angle between $y^{(k)}$ and the line spanned by
$v_1$.
\begin{enumerate}
\item[(a)] Show that rescaling at every step is the same as rescaling once at the end, i.e.
\[
y^{(k)}=\frac{M^k y^{(0)}}{\|M^k y^{(0)}\|_2},
\qquad\text{so}\qquad
M^k y^{(0)}=\sum_{i=1}^{n}\alpha_i\lambda_i^{\,k}v_i .
\]
\item[(b)] Splitting this sum into its component along $v_1$ and the rest, deduce that
\[
\tan\theta_k \;\le\; \frac{\bigl(\sum_{i\ge2}\alpha_i^2\bigr)^{1/2}}{|\alpha_1|}\,
\left|\frac{\lambda_2}{\lambda_1}\right|^{k},
\]
so that $y^{(k)}\to\pm v_1$ geometrically with ratio $|\lambda_2/\lambda_1|$.
\item[(c)] Writing $c_i=(y^{(k)})^\top v_i$, show that $\mu_k=\sum_i\lambda_ic_i^2$ and hence that
\[
|\mu_k-\lambda_1|\;\le\;(\lambda_{\max}-\lambda_{\min})\,\sin^2\theta_k .
\]
Conclude that the eigenvalue converges at rate $|\lambda_2/\lambda_1|^{2k}$, twice as fast as the
eigenvector.
\item[(d)] Take the 1D Poisson matrix $A=\operatorname{tridiag}(-1,2,-1)$ with $n=19$, whose
eigenvalues are $\lambda_j=2-2\cos\!\left(\tfrac{j\pi}{20}\right)$. Identify the two largest in
modulus, compute the ratio $|\lambda_2/\lambda_1|$, and estimate how many power iterations are needed
to reduce the eigenvector error by a factor $10^{-2}$. Relate the answer to the discussion of
smoothing in Section~\ref{sec:BIMPDE}.
\item[(e)] For a connected graph $P=D^{-1}A$ has dominant eigenvalue $1$ with eigenvector
$\mathbf 1$. Show that deflating $\mathbf 1$ is not enough: on $P_4$, $q=(1,-1,1,-1)^\top$ has $Pq=-q$
and $q^\top D\mathbf 1=0$. Show that the lazy walk $\tfrac12(I+P)$ repairs this, and state the
assumption on $\mu_2$ needed. (See \nbpath{week3/script\_week3\_3.ipynb}.)
\end{enumerate}

\chapter{Krylov Methods for Symmetric Systems}

\section*{Overview}

Chapters~1 and~3 built iterative solvers from a fixed matrix splitting $A=M-N$: every step applies the \emph{same} operator $G=M^{-1}N$, so the error can be reduced by at best a fixed factor $\rho(G)$ per iteration. For the PDE model matrices of Chapter~3 that factor behaves like $1-\mathcal{O}(h^2)$, which makes stationary methods unusually slow as the mesh is refined. 
This chapter replaces that fixed recipe by an adaptive one. At step $k$, a \emph{Krylov subspace method} relies on the subspace
\(
\mathcal{K}_k(A,r^{(0)})=\operatorname{span}\{r^{(0)},Ar^{(0)},\dots,A^{k-1}r^{(0)}\}
\)
and then selects the \emph{best} approximation in the affine space $x^{(0)}+\mathcal{K}_k(A,r^{(0)})$. For symmetric positive definite (SPD) $A$, this optimal choice can be computed with short recurrences and constant storage: this is the \emph{Conjugate Gradient} (CG) method \cite{hestenesMethodsConjugateGradients1952}, and its convergence is governed by $\sqrt{\kappa(A)}$ instead of $\kappa(A)$. 

This chapter covers: (1) CG, its optimality properties and its Chebyshev convergence bound; (2) the Lanczos algorithm, including the fact that CG and Lanczos are one recurrence seen from two sides, and its use to compute the Fiedler vector of a graph; (3) machine learning, where we show that ridge regression, truncated SVD and CG are all \emph{spectral filters} of the same data, so that stopping CG early is itself a regularisation.

\begin{tcolorbox}[colback=softblue,colframe=TUblue!40!black,title=\textbf{Learning objectives},fonttitle=\bfseries]
By the end of this chapter, you should be able to:
\begin{itemize}
    \item Derive the Conjugate Gradient method from the minimisation of the quadratic functional $J(x)=\tfrac12 x^\top Ax-b^\top x$, and explain why the steepest descent is not enough.
    \item State and use the optimality property of CG: the iterate $x^{(k)}$ minimises the $A$-norm of the error over $x^{(0)}+\mathcal{K}_k(A,r^{(0)})$.
   
    \item Use the Chebyshev bound to relate the CG convergence rate to $\sqrt{\kappa(A)}$, and contrast this with the $\kappa(A)$ dependence of steepest descent and stationary methods.
    \item Describe Lanczos iteration and explain why Ritz values converge at the ends of the spectrum first, and use deflation to compute the Fiedler vector of a graph Laplacian.
    \item Explain the CG-Lanczos equivalence, and interpret ridge regression and truncated CG as spectral filters - so that early stopping is itself a regularisation.
    \item Explain what semi-convergence is, and why some stopping rule is still needed when the data are noisy.
\end{itemize}
\end{tcolorbox}

\clearpage
\section{The Conjugate Gradient method}\label{sec:cg}

We now move from stationary iterations such as Jacobi, Gauss--Seidel, and SOR to Krylov subspace methods. These methods generate successively richer approximation spaces and the most famous example is the \emph{Conjugate Gradient} (CG) method. For symmetric positive definite (SPD) matrices, such as those arising from elliptic PDEs, CG achieves convergence far superior to classical schemes. We wish
to solve $Ax=b$ with $A \in \mathbb{R}^{n\times n}$ a SPD matrix. In this case, it makes sense to consider the quadratic functional
\[
J(x) = \tfrac{1}{2} x^\top A x - b^\top x.
\]
Minimising $J(x)$ is equivalent to solving $\nabla J = Ax - b =0$ hence the linear system $Ax=b$. 
\begin{tcolorbox}[colback=green!5!white,colframe=green!50!black]
\begin{definition}[Steepest descent method]
Starting from an initial guess $x^{0}$, the \emph{steepest descent method} generates iterates
\[
x^{(k+1)} = x^{(k)} + \tau_k r^{(k)}, 
\qquad r^{(k)} = b - Ax^{(k)},
\]
where the step size $\tau_k$ is chosen to minimize $J$ along the search direction opposite to the gradient of $J$ i.e. $r^{(k)}$.
\end{definition}
\end{tcolorbox}
Let's compute the formula of the step length in the steepest descent method.
With the notations above, let $\phi(\tau)=J\!\left(x^{(k)}+\tau r^{(k)}\right)$.
Then
\[
\phi'(\tau)=\left(r^{(k)}\right)^\top\!\big(Ax^{(k)}-b\big)+\tau\left(r^{(k)}\right)^\top A r^{(k)}
= -\,\|r^{(k)}\|^2+\tau\left(r^{(k)}\right)^\top A r^{(k)}.
\]
Setting $\phi'(\tau)=0$ gives
\[
{\;\tau_k=\dfrac{r^{(k)\top}r^{(k)}}{r^{(k)\top} A r^{(k)}} = \dfrac{\|r^{(k)}\|^2}{\|r^{(k)}\|^2_A}\;},
\]
which is equivalent to minimising $J$ along the direction $r^{(k)}$.

\begin{tcolorbox}[colback=yellow!5!white,colframe=yellow!50!black]
\begin{proposition}[Convergence of steepest descent]
Let $x^\ast$ be the exact solution of $Ax=b$ and let $\kappa(A) =
\lambda_{\max}(A)/\lambda_{\min}(A)$ denote the condition number of $A$. Then the steepest descent iterates satisfy
\[
\|x^{(k)} - x^\ast\|_A \le
\left(\frac{\kappa(A)-1}{\kappa(A)+1}\right)^k \,
\|x^{(0)}-x^\ast\|_A,
\]
where $\|v\|_A^2 = v^\top A v$ (norm induced by $A$). 
\end{proposition}
\end{tcolorbox}

\begin{proof}[Sketch of proof]
Let $e^{(k)} = x^{(k)} - x^\ast$ and $r^{(k)} = b - A x^{(k)} = -A e^{(k)}$.  
A steepest descent step gives
\[
e^{(k+1)} = e^{(k)} + \tau_k r^{(k)}.
\]
Expanding the $A$-norm of the error yields
\[
\|e^{(k+1)}\|_A^2 
= \|e^{(k)}\|_A^2 - \tau_k \|r^{(k)}\|^2,
\]
where we used the optimal choice of $\tau_k$. Dividing by $\|e^{(k)}\|_A^2$ and
substituting the formula of $\tau_k$ gives
\[
\frac{\|e^{(k+1)}\|_A^2}{\|e^{(k)}\|_A^2}
= 1 - \frac{\|r^{(k)}\|^4}{(r^{(k)\top} A r^{(k)})(r^{(k)\top} A^{-1} r^{(k)})}.
\]
The denominator is bounded below via the \emph{Kantorovich inequality}\footnote{See Problem~5: $1\le \frac{(v^\top A v)(v^\top A^{-1} v)}{\|v\|^4} \le \frac{1}{4}(\sqrt{\kappa(A)}+\frac{1}{\sqrt{\kappa(A)}})^2$}, which
relates $(v^\top A v)(v^\top A^{-1} v)$ to the condition number $\kappa(A)$. 
Inserting this bound yields
\[
\frac{\|e^{(k+1)}\|_A}{\|e^{(k)}\|_A}
\;\leq\; \frac{\kappa(A)-1}{\kappa(A)+1}.
\]
By induction over $k$, the stated result follows.
\end{proof}

\begin{remark}
In particular, for PDE discretizations of the Poisson equation introduced in Chapter~2
\[
-u'' = f, \qquad u(0)=u(1)=0,
\]
which yields the SPD stiffness matrix $A \in \mathbb{R}^{n\times n}$. Here,  the
condition number scales as $\kappa(A) = \mathcal{O}(n^2)$. Hence steepest descent requires $\mathcal{O}(n^2)$ iterations, which becomes prohibitive for large systems.
\end{remark}

\begin{figure}[htbp]
\centering
\begin{tikzpicture}[>=stealth,scale=1.28]
\begin{scope}
  \node[font=\small] at (0,1.6) {Steepest descent};
  \foreach \a/\b in {2.579/0.860, 1.730/0.577, 1.031/0.344}{\draw[gray!55] (0,0) ellipse ({\a} and {\b});}
  \draw[->,thick,TUblue] (1.950,0.563) -- (1.706,-0.073);
  \draw[->,thick,TUblue] (1.706,-0.073) -- (0.867,0.250);
  \draw[->,thick,TUblue] (0.867,0.250) -- (0.758,-0.032);
  \draw[->,thick,TUblue] (0.758,-0.032) -- (0.386,0.111);
  \draw[->,thick,TUblue] (0.386,0.111) -- (0.337,-0.014);
  \draw[->,thick,TUblue] (0.337,-0.014) -- (0.171,0.050);
  \fill (0,0) circle (1.3pt);
  \node[font=\scriptsize,below left=-1pt] at (0,0) {$x^\ast$};
  \fill[TUblue] (1.950,0.563) circle (1.3pt);
  \node[font=\scriptsize,above right=-2pt] at (1.950,0.563) {$x^{(0)}$};
  \node[font=\scriptsize,align=center] at (0,-1.5) {successive residuals are orthogonal:\\the iterates zig-zag down the valley};
\end{scope}
\begin{scope}[xshift=6.6cm]
  \node[font=\small] at (0,1.6) {Conjugate Gradient};
  \foreach \a/\b in {2.579/0.860, 1.730/0.577, 1.031/0.344}{\draw[gray!55] (0,0) ellipse ({\a} and {\b});}
  \draw[->,thick,TUblue] (1.950,0.563) -- (1.706,-0.073);
  \draw[->,thick,TUblue] (1.706,-0.073) -- (0,0);
  \node[font=\scriptsize,right=1pt] at (1.75,-0.28) {$p^{(0)}$};
  \node[font=\scriptsize,below] at (0.85,-0.10) {$p^{(1)}$};
  \fill (0,0) circle (1.3pt);
  \node[font=\scriptsize,above left=-2pt] at (0,0) {$x^\ast$};
  \fill[TUblue] (1.950,0.563) circle (1.3pt);
  \node[font=\scriptsize,above right=-2pt] at (1.950,0.563) {$x^{(0)}$};
  \node[font=\scriptsize,align=center] at (0,-1.5) {two $A$-conjugate directions:\\exact solution after $n=2$ steps};
\end{scope}
\end{tikzpicture}
\caption{Steepest descent versus Conjugate Gradient on the level sets of $J(x)=\tfrac12x^\top Ax-b^\top x$
with $A=\mathrm{diag}(1,9)$ and the same starting point (iterates computed exactly). Because consecutive
steepest-descent residuals are orthogonal, the iterates zig-zag across the valley and creep towards
$x^\ast$; the number of steps grows like $\kappa(A)$, so the picture gets dramatically worse as the
contours become more elongated. CG replaces orthogonality of the \emph{residuals} by $A$-conjugacy of the
\emph{search directions}: no direction is ever revisited, and in exact arithmetic $n$ steps suffice.}
\label{fig:sd-vs-cg}
\end{figure}
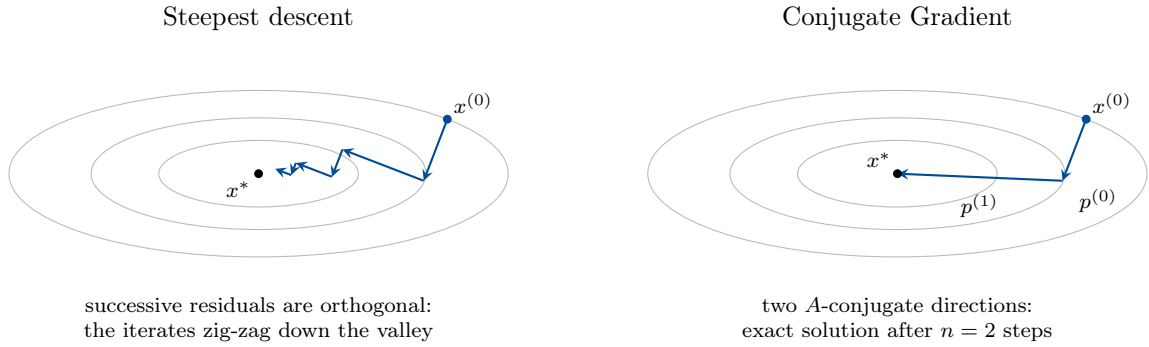

\begin{anecdote}
    When Hestenes and Stiefel introduced Conjugate Gradient in 1952, computing was still in its infancy: machines like the SWAC in Los Angeles (then fastest computer in the world) had memory measured in kilobytes and could barely store a few thousand numbers.  Steepest descent was well known, but its zig-zagging inefficiency was frustrating. The breakthrough of CG was to design successive directions that “remember” past progress and never waste effort repeating the same zig-zag. This made it one of the first practical Krylov methods and a cornerstone of numerical linear algebra.
\end{anecdote}

The idea of Conjugate Gradient (CG), developed by Hestenes and Stiefel \cite{hestenesMethodsConjugateGradients1952},
is to construct \emph{search directions} $p^{(k)}$ that are not only descent
directions but are $A$-orthogonal. Using such conjugate directions ensures that each step eliminates error
components independently, guaranteeing convergence in at most $n$ steps in exact arithmetic.

\begin{tcolorbox}[colback=green!5!white,colframe=green!50!black]
\begin{definition}[Conjugate directions]
Let $A \in \mathbb{R}^{n\times n}$ be SPD.  
Two vectors $p,q \in \mathbb{R}^n$ are said to be \emph{$A$-conjugate} if
\[
p^\top A q = 0.
\]
A family of vectors $\{p^{(0)},\ldots,p^{(k)}\}$ is called \emph{mutually $A$-conjugate} if
every pair is $A$-conjugate. Such directions play the role of “orthogonal axes”
with respect to the $A$-inner product $\langle u,v\rangle_A = u^\top A v$.
\end{definition}
\end{tcolorbox}

\begin{tcolorbox}[colback=green!5!white,colframe=green!50!black]
\begin{definition}[Krylov subspaces]
Given an initial residual $r^{(0)} = b - A x^{(0)}$, the \emph{$k$-th Krylov subspace} is
\[
\mathcal{K}_k(A,r^{(0)}) = \mathrm{span}\{r^{(0)}, Ar^{(0)}, \ldots, A^{k-1}r^{(0)}\}.
\]
Equivalently, 
\[
\mathcal{K}_k(A,r^{(0)}) = \{\,q(A)r^{(0)}:\ q \in \Pi_{k-1}\,\},
\]
where $\Pi_{k-1}$ denotes the set of polynomials of degree at most $k-1$.
\end{definition}
\end{tcolorbox}

\begin{tcolorbox}[colback=green!5!white,colframe=green!50!black]
\begin{definition}[Conjugate Gradient method]\label{def:cg}
Let $Ax=b$ with $A \in \mathbb{R}^{n\times n}$ be SPD and the right-hand side $b \in \mathbb{R}^n$.  
Starting from an initial guess $x^{(0)}$ and the initial descent direction $p^{(0)}= r^{(0)}$, the Conjugate Gradient (CG) generates a sequence of approximations $\{x^{(k)}\}$ by choosing search directions
$p^{(k)} \in \mathcal{K}_{k+1}(A,r^{(0)})$ iteratively as follows:
\[
x^{(k+1)} = x^{(k)} + \alpha_k p^{(k)}, 
\qquad \alpha_k = \frac{r^{(k)\top}r^{(k)}}{p^{(k)\top} A p^{(k)}},
\]
\[
r^{(k+1)} = r^{(k)} - \alpha_k A p^{(k)}, \qquad \beta_k = \frac{r^{(k+1)\top}r^{(k+1)}}{r^{(k)\top}r^{(k)}},
\]
\[ 
p^{(k+1)} = r^{(k+1)} + \beta_k p^{(k)}.
\]
\end{definition}
\end{tcolorbox}

\begin{tcolorbox}[colback=yellow!5!white,colframe=yellow!50!black]
\begin{proposition}[Properties of CG]\label{prop:cg_properties}
Let $A$ be SPD. Then the iterates of CG satisfy:
\begin{itemize}
  \item Residuals are mutually orthogonal: $r^{(i)\top} r^{(j)} = 0$ for $i\neq j$.
  \item Search directions are $A$-orthogonal: $p^{(i)\top} A p^{(j)}= 0$ for $i\neq j$.
  \item $x^{(k)}$ minimizes $\|x-x^\ast\|_A$ over the affine space $x^{(0)}+\mathcal{K}_k(A,r^{(0)})$.
\end{itemize}
\end{proposition}
\end{tcolorbox}

\begin{figure}[htbp]
\centering
\begin{tikzpicture}[>=stealth,scale=1.1]
  \foreach \f in {1,0.72,0.44}{\draw[rotate around={20:(0,0)},gray!55] (0,0) ellipse ({2.5*\f} and {1.15*\f});}
  \draw[dashed,thick] (-1.55,1.75) -- (2.95,-0.15);
  \node[font=\scriptsize,anchor=east] at (-1.62,1.80) {$x^{(0)}+\mathcal{K}_k(A,r^{(0)})$};
  \coordinate (x0) at (2.60,0.0);
  \coordinate (xk) at (0.42,0.92);
  \fill (0,0) circle (1.4pt);
  \node[font=\scriptsize,below=1pt] at (0,0) {$x^\ast$};
  \fill (x0) circle (1.4pt);
  \node[font=\scriptsize,below right=-2pt] at (x0) {$x^{(0)}$};
  \fill[TUblue] (xk) circle (1.4pt);
  \node[font=\scriptsize,above=1pt] at (xk) {$x^{(k)}$};
  \draw[->,thick,TUblue] (x0) -- (xk);
  \draw[densely dotted] (0,0) -- (xk);
  \node[font=\scriptsize,anchor=north] at (0.6,-1.55)
    {$x^{(k)}=\arg\min\bigl\{\|x-x^\ast\|_A\ :\ x\in x^{(0)}+\mathcal{K}_k(A,r^{(0)})\bigr\}$};
\end{tikzpicture}
\caption{The optimality property of CG. The ellipses are level sets of the energy norm
$\|x-x^\ast\|_A$. At step $k$ the iterate is the point of the affine space
$x^{(0)}+\mathcal{K}_k(A,r^{(0)})$ (dashed line) closest to $x^\ast$ in the
$A$-norm, i.e.\ the $A$-orthogonal projection of $x^\ast$ onto it. }
\label{fig:cg-optimality}
\end{figure}
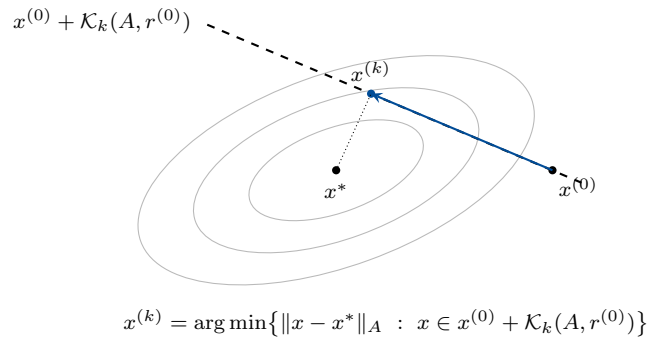

\begin{tcolorbox}[colback=green!5!white,colframe=green!50!black]
\begin{definition}[Minimal polynomial, and minimal polynomial relative to a vector]\label{def:minpoly-rel}
The \emph{minimal polynomial} of $A\in\mathbb{R}^{n\times n}$ is the monic polynomial $m$ of smallest degree with $m(A)=0$. Given also $v\neq0$, the \emph{minimal polynomial of $A$ relative to $v$} is the monic polynomial $m_v$ of smallest degree with $m_v(A)v=0$. It divides $m$, so $d_v:=\deg m_v\le\deg m\le n$, often strictly. Equivalently $d_v$ is the largest $k$ for which $v,Av,\dots,A^{k-1}v$ are linearly independent, so $\dim\mathcal{K}_k(A,v)=\min(k,d_v)$: the Krylov spaces stop growing at $k=d_v$. For symmetric $A$, take an orthonormal basis of eigenvectors $v_1,\dots,v_n$, $Av_j=\lambda_jv_j$, and expand $v=\sum_j\alpha_jv_j$ with $\alpha_j=v_j^\top v$: then $m_v(\lambda)=\prod_{\mu\in\Lambda_v}(\lambda-\mu)$, where $\Lambda_v=\{\lambda_j:\alpha_j\neq0\}$ is the set of eigenvalues present in $v$, a repeated eigenvalue counting once however many of its $\alpha_j$ are nonzero. So $d_v$ is simply the number of eigenvalues present in $v$, one fewer for every eigenvalue of $A$ that $v$ misses.
\end{definition}
\end{tcolorbox}

\begin{tcolorbox}[colback=yellow!5!white,colframe=yellow!50!black]
\begin{proposition}[Polynomial interpretation and finite termination of CG]\label{prop:cg_poly}
Let $A\in\mathbb{R}^{n\times n}$ be SPD, $x^\ast$ denote the exact solution of $Ax=b$, and $e^{(k)}=x^{(k)}-x^\ast$ the CG error at step $k$ (with $x^{(0)}$ arbitrary).

\begin{enumerate}
\item (\emph{Polynomial representation}) For each $k\ge 0$ there exists a polynomial $q_k$ with $\deg q_k\le k$ and $q_k(0)=1$ such that
\[
e^{(k)} = q_k(A)\,e^{(0)}.
\]
The polynomial $q_k$ minimizes the $A$ norm of the error at step $k$ over the set
\[
\bigl\{\,q(A)e^{(0)}:\ q\in\Pi_k,\ q(0)=1\,\bigr\}.
\]

\item (\emph{Finite termination}) Let $d_0=\deg m_{r^{(0)}}$, where $m_{r^{(0)}}$ is the minimal polynomial of $A$ relative to the initial residual $r^{(0)}$ (Definition~\ref{def:minpoly-rel}). Then CG attains the exact solution in at most $d_0$ steps, \(e^{(d_0)}=0\). Always $d_0\le d\le n$ with $d=\deg m$ (for symmetric $A$, $d$ is the number of distinct eigenvalues), and $d_0<d$ as soon as $r^{(0)}$ misses part of the spectrum.
\end{enumerate}
\end{proposition}
\end{tcolorbox}

\begin{proof}
(1) CG produces $x^{(k)}\in x^{(0)}+\mathcal{K}_k(A,r^{(0)})$, hence
\[
e^{(k)}=x^{(k)}-x^\ast = e^{(0)} - v_k,\quad v_k\in \mathcal{K}_k(A,r^{(0)}).
\]
Write $v_k=A\,s_{k-1}(A)e^{(0)}$ with $\deg s_{k-1}\le k-1$ (since $\mathcal{K}_k(A,r^{(0)})=A\,\mathcal{K}_k(A,e^{(0)})$). Then
\[
e^{(k)} = \bigl(I - A\,s_{k-1}(A)\bigr)e^{(0)} = q_k(A)e^{(0)},
\]
where $q_k(\lambda)=1-\lambda s_{k-1}(\lambda)$ has degree $\le k$ and satisfies $q_k(0)=1$. This also shows $e^{(k)}\in\mathcal{K}_{k+1}(A,e^{(0)})=\mathrm{span}\{e^{(0)},Ae^{(0)},\ldots,A^k e^{(0)}\}$. {The minimization follows from Prop. \ref{prop:cg_properties}.}

(2) $A$ is SPD, hence nonsingular, and $r^{(0)}=-Ae^{(0)}$, so $q(A)r^{(0)}=0$ if and only if $q(A)e^{(0)}=0$: the relative minimal polynomial is the same for $r^{(0)}$ and for $e^{(0)}$. Its roots are eigenvalues of $A$, which are positive, so $m_{r^{(0)}}(0)\neq0$ and $q_{d_0}=m_{r^{(0)}}/m_{r^{(0)}}(0)$ has $q_{d_0}(0)=1$ and $q_{d_0}(A)e^{(0)}=0$. By part (1), the attainable errors after $k\ge d_0$ steps include $0$, so $\|e^{(d_0)}\|_A=0$. Since $m_{r^{(0)}}$ divides $m$, $d_0\le d=\deg m$; for symmetric $A$ with $d$ distinct eigenvalues $m(\lambda)=\prod_{j=1}^{d}(\lambda-\lambda_j)$, so $d\le n$. The relative count is the sharp one: three distinct eigenvalues with $r^{(0)}$ orthogonal to one eigenvector give $d_0=2$, so CG is exact after two steps, not three. \\
\emph{Remark:} All the above is true in exact arithmetic. In floating-point arithmetic, loss of orthogonality may delay (or prevent) exact termination; the classical analysis of this effect is due to Paige \cite{paigeAccuracyEffectivenessLanczos1980}.
\end{proof}

\begin{tcolorbox}[colback=yellow!5!white,colframe=yellow!50!black]
\begin{proposition}[Convergence of Conjugate Gradient]
Let $x^\ast = A^{-1}b$ denote the exact solution and $\kappa(A) =
\lambda_{\max}(A)/\lambda_{\min}(A)$ the condition number of $A$.  
Then the CG iterates satisfy the error bound
\[
\frac{\|x^{(k)}-x^\ast\|_A}{\|x^{(0)}-x^\ast\|_A}
\le 2\left(\frac{\sqrt{\kappa(A)}-1}{\sqrt{\kappa(A)}+1}\right)^k,
\]
where $\|v\|_A^2 = v^\top A v$. 
In exact arithmetic the method terminates in at most $n$ steps.
\end{proposition}
\end{tcolorbox}

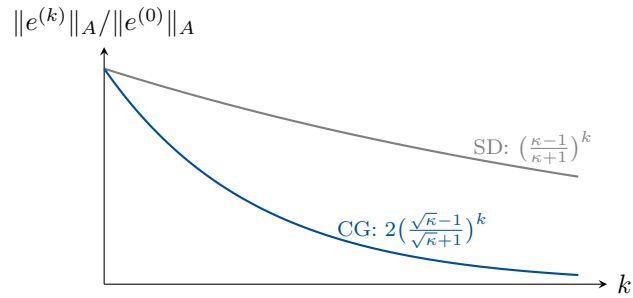
\begin{figure}[htbp]
\centering
\begin{tikzpicture}[>=stealth,scale=0.95]
  \draw[->] (0,0) -- (7.0,0) node[right,font=\small]{$k$};
  \draw[->] (0,0) -- (0,3.3) node[above,font=\small]{$\|e^{(k)}\|_A/\|e^{(0)}\|_A$};
  \draw[gray,smooth,domain=0:6.6,samples=80,thick] plot(\x,{3*0.90^(\x)});
  \draw[TUblue,smooth,domain=0:6.6,samples=80,thick] plot(\x,{3*0.62^(\x)});
  \node[gray,font=\scriptsize,anchor=west] at (5.0,1.9)
     {SD: $\bigl(\tfrac{\kappa-1}{\kappa+1}\bigr)^k$};
  \node[TUblue,font=\scriptsize,anchor=west] at (3.1,0.75)
     {CG: $2\bigl(\tfrac{\sqrt{\kappa}-1}{\sqrt{\kappa}+1}\bigr)^k$};
\end{tikzpicture}
\caption{Error decay of steepest descent and CG for the same SPD matrix. Both are linear (geometric)
rates, but CG's contraction factor is governed by $\sqrt{\kappa(A)}$ rather than $\kappa(A)$. For the 1D
Poisson matrix, where $\kappa(A)=\mathcal{O}(n^2)$, this is the difference between $\mathcal{O}(n^2)$ and
$\mathcal{O}(n)$ iterations. The bound is only an upper estimate: when the eigenvalues of $A$ are
clustered, CG typically converges considerably faster than the curve shown.}
\label{fig:sd-cg-rates}
\end{figure}

\paragraph{Example.} For the 1D Poisson matrix of size $n$, steepest descent needs on the order of $n^2$
iterations, while CG requires only on the order of $n$ (Problem~4). This dramatic acceleration is what makes CG the
standard choice for large SPD PDE systems. A full treatment of CG and its convergence theory can be found in \cite{greenbaumIterativeMethods1997,saadIterativeMethodsSparse2003}.

\begin{takeaway}

\begin{itemize}
  \item \textbf{Comparison with classical iterations.}
    \begin{itemize}
      \item Jacobi, Gauss--Seidel, and SOR are simple stationary schemes with iteration counts scaling like $\mathcal{O}(n^2)$ for PDEs.
      \item Steepest descent improves the picture but still requires
            $\mathcal{O}(\kappa(A))$ iterations, which is prohibitive when
            $\kappa(A)$ scales like $n^2$.
      \item Conjugate Gradient reduces the required iterations to
            $\mathcal{O}(\sqrt{\kappa(A)})$, leading to dramatic savings in
            practice.
    \end{itemize}

  \item \textbf{Finite termination property.}  
  In exact arithmetic CG produces the exact solution in at most $n$ steps,
  unlike stationary methods which only converge asymptotically.

  \item \textbf{Energy minimization viewpoint.}  
  Each CG iterate minimizes the quadratic energy functional over a growing
  Krylov subspace, making the method ``optimal'' among all polynomial
  approximations of degree $k$.

\end{itemize}
    
\end{takeaway}

\begin{table}[h]
\centering
\renewcommand{\arraystretch}{1.3}
\begin{tabular}{|l|c|c|l|}
\hline
\textbf{Method} & \textbf{Convergence} & \textbf{Iterations (PDE)} & \textbf{Remarks} \\
\hline
Jacobi & $\big(1 - c n^{-2}\big)^k$ & $\mathcal{O}(n^2)$ & Very slow; mainly used as smoother \\
Gauss--Seidel & $\big(1 - c n^{-2}\big)^k$  & $\mathcal{O}(n^2)$ & Roughly twice as fast as Jacobi \\
SOR (optimal $\omega$) & $\big(1 - c n^{-1}\big)^k$ & $\mathcal{O}(n)$ & Improved, but still asymptotically slow \\
\hline
Steepest descent & $\big(\frac{\kappa-1}{\kappa+1}\big)^k$ & $\mathcal{O}(\kappa)$ & Gradient method; zig-zag convergence \\
Conjugate Gradient & $2\Big(\tfrac{\sqrt{\kappa}-1}{\sqrt{\kappa}+1}\Big)^k$ & $\mathcal{O}(\sqrt{\kappa})$ & Optimal Krylov solver \\
\hline
\end{tabular}
\caption{Comparison of classical stationary methods, steepest descent, and Conjugate Gradient for SPD PDE systems.}
\end{table}

\section{The Lanczos Method and the Symmetric Eigenproblem}\label{sec:lanczos}

So far, the Krylov subspace $\mathcal{K}_m(A,r^{(0)})$ has been a useful tool for \emph{solving} $Ax=b$. The
same subspace answers a different question as well: what does the \emph{spectrum} of $A$ look like? There are two different aspects to this question:

\begin{itemize}
\item \textbf{We need $\kappa(A)$.} Every bound in Section~\ref{sec:cg} is stated in terms of
$\lambda_{\min}(A)$ and $\lambda_{\max}(A)$, yet for a large sparse $A$ these are exactly the quantities we
cannot compute directly. Estimating them is not a side task: it is what tells us whether CG will take ten
iterations or ten thousand, and, from Chapter~6, whether a preconditioner is doing its job.
\item \textbf{The spectrum \emph{is} needed \emph{but also} the eigenvectors.} In Chapter~3 we saw that for a graph Laplacian $L=D-A$ the
eigenvector belonging to the smallest nonzero eigenvalue, i.e. the \emph{Fiedler vector} \cite{fiedlerAlgebraicConnectivity1973}, encodes how the
graph splits into communities. 
\end{itemize}

In both cases $n$ is far too large for a full eigendecomposition, which costs $\mathcal{O}(n^3)$ and
destroys sparsity. What we can afford, as always, is the product $Av$.
Instead, we rely on iterative Krylov subspace methods such as the \emph{Lanczos algorithm} \cite{lanczosIterationMethod1950}. The Lanczos algorithm is an efficient iterative method for approximating eigenvalues and eigenvectors of a large symmetric matrix $A$. The key idea is:
\begin{itemize}
  \item Start with a normalized vector $q_1$.
  \item Build the Krylov subspace
  \[
  \mathcal{K}_m(A,q_1) = \text{span}\{q_1, Aq_1, A^2 q_1, \dots, A^{m-1}q_1\}.
  \]
  \item Iteratively construct an orthonormal basis $Q_m = [q_1, \dots, q_m]$ of $\mathcal{K}_m(A,q_1)$ and project:
  \[ 
  T_m = Q_m^\top A Q_m.
  \]
  \item For a general matrix $A$, the projected matrix $T_m$ is \emph{upper Hessenberg}, i.e.\ upper
  triangular with one extra subdiagonal. This is the \emph{Arnoldi} case and is the subject of Chapter~5.
  \item When $A$ is symmetric, however,
  \[
  T_m^{\top} = (Q_m^\top A Q_m)^\top = Q_m^\top A^{\top} Q_m = Q_m^\top A Q_m = T_m,
  \]
  so $T_m$ is symmetric as well. Being simultaneously Hessenberg and symmetric, it must be
  \emph{tridiagonal}:
\[  T_m =
\begin{bmatrix}
\alpha_1 & \beta_1 &        &   \\
\beta_1  & \alpha_2 & \ddots &   \\
         & \ddots  & \ddots & \beta_{m-1} \\
         &         & \beta_{m-1} & \alpha_m
\end{bmatrix}.
\]
  \item Only the scalars $\alpha_j$ and $\beta_j$ therefore need to be computed, and each new basis vector
  is orthogonal to all but two of its predecessors automatically. This is the \emph{Lanczos} algorithm,
  displayed in Algorithm~\ref{alg:lanczos}.
  \item Eigenvalues of $T_m$ (Ritz values) approximate eigenvalues of $A$, and 
  $Q_m y$ approximates an eigenvector of $A$ if $y$ is an eigenvector of $T_m$.
\end{itemize}

\begin{algorithm}[h!] 
\caption{Lanczos Algorithm} \label{alg:lanczos}
\begin{algorithmic}[1]
\REQUIRE Symmetric matrix $A \in \mathbb{R}^{n \times n}$, starting vector $q_1$ with $\|q_1\|_2 = 1$, number of steps $m$.
\STATE Set $\beta_0 = 0$, $q_0 = 0$.
\FOR{$j = 1, \dots, m$}
  \STATE $w \gets A q_j - \beta_{j-1} q_{j-1}$
  \STATE $\alpha_j \gets q_j^\top w$
  \STATE $w \gets w - \alpha_j q_j$
  \STATE $\beta_j \gets \|w\|_2$
  \IF{$\beta_j = 0$}
    \STATE \textbf{stop}
  \ELSE
    \STATE $q_{j+1} \gets w / \beta_j$
  \ENDIF
\ENDFOR
\STATE Construct the tridiagonal matrix $T_m$.
\end{algorithmic}
\end{algorithm}

\begin{tcolorbox}[colback=blue!3!white,colframe=blue!40!black]
\textbf{Notation: two different $\alpha$'s and $\beta$'s.}
The symbols $\alpha$ and $\beta$ are unfortunately standard in \emph{both} algorithms, with unrelated
meanings. In CG (Definition~\ref{def:cg}) $\alpha_k$ is a \emph{step length} and $\beta_k$ a
\emph{direction-update} coefficient; in Lanczos (Algorithm~\ref{alg:lanczos}) $\alpha_j$ and $\beta_j$
are the diagonal and off-diagonal \emph{entries of $T_m$}. Both conventions are too well established to
be worth changing, so we keep them. In Section~\ref{sec:cglanczos}, where the two algorithms are compared
side by side, we write $\alpha_k^{\mathrm{cg}},\beta_k^{\mathrm{cg}}$ for the CG scalars and reserve the $\alpha_j,\beta_j$ for the entries of $T_m$.
\end{tcolorbox}

\begin{figure}[htbp]
\centering
\begin{tikzpicture}[>=stealth,node distance=2.4cm]
\tikzstyle{bx}=[draw,TUblue!60!black,rounded corners,inner sep=6pt,fill=softblue,font=\small]
\node[bx] (A) {$A\in\mathbb{R}^{n\times n}$};
\node[bx,right=2.6cm of A] (Q) {$Q_m\in\mathbb{R}^{n\times m}$};
\node[bx,right=2.6cm of Q] (T) {$T_m\in\mathbb{R}^{m\times m}$};
\draw[->,thick] (A) -- node[above,font=\scriptsize]{Krylov basis}
                        node[below,font=\scriptsize]{$m$ matvecs} (Q);
\draw[->,thick] (Q) -- node[above,font=\scriptsize]{$Q_m^\top A Q_m$}
                        node[below,font=\scriptsize]{tridiagonal} (T);
\node[font=\scriptsize,anchor=north] at ($(A)!0.5!(T)+(0,-1.15)$)
  {$n$ huge and sparse \quad$\longrightarrow$\quad $m\ll n$, dense but tiny: eigenvalues by any direct method};
\end{tikzpicture}
\caption{The Lanczos idea in one line. The large sparse symmetric matrix $A$ is never factorised; it is
only applied to vectors. After $m$ steps the problem has been compressed into an $m\times m$ symmetric
tridiagonal matrix $T_m$ whose eigenvalues (the \emph{Ritz values}) approximate the extreme eigenvalues of
$A$, and whose eigenvectors lift back to approximate eigenvectors of $A$ through $Q_m$.}
\label{fig:lanczos-pipeline}
\end{figure}

\paragraph{Interpretation.}
The matrix $T_m \in \mathbb{R}^{m \times m}$ represents the orthogonal projection of $A$ onto the Krylov
subspace $\mathcal{K}_m(A,q_1)$, expressed in the orthonormal basis $q_1,\dots,q_m$. Its eigenvalues,
called \emph{Ritz values}, provide approximations to the eigenvalues of $A$ and, as the next
proposition makes precise, to the \emph{extremal} ones first.
Since $T_m$ is symmetric and tridiagonal, its eigenvalues can be computed efficiently and stably using standard algorithms \cite{GolubvanLoan}.
Eigenvectors of $T_m$, lifted back via $Q_m = [q_1,\dots,q_m]$, yield approximations of the eigenvectors of $A$.

\begin{tcolorbox}[colback=yellow!5!white,colframe=yellow!50!black]
\begin{proposition}[Ritz values converge at the ends of the spectrum]
Let $A=A^\top$ have eigenvalues $\lambda_1\le\cdots\le\lambda_n$ and let
$\theta_1\le\cdots\le\theta_m$ be the eigenvalues of $T_m=Q_m^\top AQ_m$. Then
\[
\lambda_1\;\le\;\theta_1,\qquad \theta_m\;\le\;\lambda_n,
\]
and more generally the $\theta_i$ interlace the $\lambda_i$ (Problem~9). Moreover $\theta_1$ and
$\theta_m$ are the smallest and largest values of the Rayleigh quotient $R_A(x)=x^\top Ax/x^\top x$ over
$\mathcal{K}_m(A,q_1)$, so they improve monotonically with $m$.
\end{proposition}
\end{tcolorbox}

Two consequences matter in practice. First, the Ritz values approach $\lambda_1$ and $\lambda_n$ from the
inside, so a handful of Lanczos steps already brackets the spectrum which is precisely the estimate of
$\kappa(A)$ that the CG bound of Section~\ref{sec:cg} needs. Second, \emph{interior} eigenvalues are
reached much more slowly. If the eigenvalue of interest is not extremal, one must deflate the
eigenvectors that are in the way. The next example uses the first of these. Sharp convergence
estimates for the Ritz values, which show that the rate again improves by a square root over the power
method, are given in \cite{saadRatesConvergenceLanczos1980}.

\subsection{Lanczos and CG are the same recurrence}\label{sec:cglanczos}

It is no accident that Lanczos and CG appear in the same chapter. Both build an orthonormal basis of the
same Krylov subspace by a three-term recurrence, and in exact arithmetic they build the \emph{same} one.
The following proposition makes this precise and, more usefully, shows that the entries of $T_m$ can be
read off from a CG run without any extra work.

\begin{tcolorbox}[colback=yellow!5!white,colframe=yellow!50!black]
\begin{proposition}[CG and Lanczos generate the same basis,
{\cite[Section~3.1, Theorem~3]{bouyouliNewResultsConvergence2009}}]\label{prop:cglanczos}
Let $A$ be SPD and $b\neq0$. Run CG (Definition~\ref{def:cg}) from $x^{(0)}=0$, so that $r^{(0)}=b$, and
assume $r^{(j)}\neq0$ for $j=0,\dots,m-1$; denote its scalars by
$\alpha^{\mathrm{cg}}_k,\beta^{\mathrm{cg}}_k$. Run Lanczos (Algorithm~\ref{alg:lanczos}) on $A$ with
$q_1=b/\|b\|_2$, producing $q_1,\dots,q_m$ and the entries $\alpha_j,\beta_j$ of $T_m$. Then
\begin{enumerate}
\item[(i)] the two methods produce the same orthonormal basis of $\mathcal{K}_m(A,b)$, namely
\[
q_{j+1}=(-1)^{j}\,\frac{r^{(j)}}{\|r^{(j)}\|_2},\qquad j=0,\dots,m-1;
\]
\item[(ii)] the entries of $T_m$ are determined by the CG scalars,
\[
\alpha_1=\frac{1}{\alpha^{\mathrm{cg}}_{0}},\qquad
\alpha_j=\frac{1}{\alpha^{\mathrm{cg}}_{j-1}}+\frac{\beta^{\mathrm{cg}}_{j-2}}{\alpha^{\mathrm{cg}}_{j-2}}
\ \ (j\ge2),\qquad
\beta_j=\frac{\sqrt{\beta^{\mathrm{cg}}_{j-1}}}{\alpha^{\mathrm{cg}}_{j-1}}\ \ (j\ge1).
\]
\end{enumerate}
\end{proposition}
\end{tcolorbox}

\begin{proof}
(i) By Proposition~\ref{prop:cg_properties} the residuals $r^{(0)},\dots,r^{(m-1)}$ are mutually
orthogonal, and $r^{(j)}\in\mathcal{K}_{j+1}(A,b)\setminus\mathcal{K}_{j}(A,b)$. Once normalised they are
therefore an orthonormal basis of $\mathcal{K}_m(A,b)$ adapted to the nested sequence
$\mathcal{K}_1\subset\mathcal{K}_2\subset\cdots$. The Lanczos vectors are, by construction, the
Gram--Schmidt basis of that same nested sequence, and such a basis is unique up to the sign of each
vector. The signs are pinned down by the convention $\beta_j=\|w\|_2>0$ in
Algorithm~\ref{alg:lanczos}, which produces the alternating factor $(-1)^j$.

(ii) Eliminate the search directions. From $r^{(k+1)}=r^{(k)}-\alpha^{\mathrm{cg}}_kAp^{(k)}$ we get
$Ap^{(k)}=\bigl(r^{(k)}-r^{(k+1)}\bigr)/\alpha^{\mathrm{cg}}_k$, and substituting this into
$Ar^{(k)}=Ap^{(k)}-\beta^{\mathrm{cg}}_{k-1}Ap^{(k-1)}$, which follows from
$p^{(k)}=r^{(k)}+\beta^{\mathrm{cg}}_{k-1}p^{(k-1)}$, gives the three-term recurrence
\[
Ar^{(k)}=-\frac{r^{(k+1)}}{\alpha^{\mathrm{cg}}_k}
+\Bigl(\frac{1}{\alpha^{\mathrm{cg}}_k}+\frac{\beta^{\mathrm{cg}}_{k-1}}{\alpha^{\mathrm{cg}}_{k-1}}\Bigr)r^{(k)}
-\frac{\beta^{\mathrm{cg}}_{k-1}}{\alpha^{\mathrm{cg}}_{k-1}}\,r^{(k-1)} .
\]
Now write $\rho_k=\|r^{(k)}\|_2$, so that $r^{(k)}=(-1)^k\rho_k\,q_{k+1}$ by (i). The definition of
$\beta^{\mathrm{cg}}_k$ gives $\rho_{k+1}/\rho_k=\sqrt{\beta^{\mathrm{cg}}_k}$. Dividing the recurrence
by $(-1)^k\rho_k$ turns every sign positive and yields
\[
Aq_{k+1}=\frac{\sqrt{\beta^{\mathrm{cg}}_{k}}}{\alpha^{\mathrm{cg}}_{k}}\,q_{k+2}
+\Bigl(\frac{1}{\alpha^{\mathrm{cg}}_k}+\frac{\beta^{\mathrm{cg}}_{k-1}}{\alpha^{\mathrm{cg}}_{k-1}}\Bigr)q_{k+1}
+\frac{\sqrt{\beta^{\mathrm{cg}}_{k-1}}}{\alpha^{\mathrm{cg}}_{k-1}}\,q_{k} .
\]
Comparing with the Lanczos recurrence $Aq_j=\beta_jq_{j+1}+\alpha_jq_j+\beta_{j-1}q_{j-1}$ and setting
$j=k+1$ gives the stated formulas.
\end{proof}

\begin{remark}
The alternating sign in (i) is a normalisation convention: replacing
$q_{j+1}$ by $-q_{j+1}$ conjugates $T_m$ by a diagonal matrix of signs, which changes the signs of the
off-diagonal entries but leaves the Ritz values, and hence everything we use them for, unchanged.
\end{remark}

The practical payoff is that spectral information comes \emph{for free} from a run one was performing
anyway: a CG solve can report an estimate of $\lambda_{\min}$, $\lambda_{\max}$ and hence $\kappa(A)$ at
the cost of a few scalar operations and no extra matrix--vector products. Proposition~\ref{prop:cglanczos}
also explains the finite-termination property of Section~\ref{sec:cg} from the other side: CG stops after
$d$ steps because Lanczos exhausts the Krylov space after $d$ steps, $d$ being the number of distinct
eigenvalues. A systematic account of this correspondence, together with the consequences it has for
error estimation inside a running CG solve, is given in
\cite{bouyouliNewResultsConvergence2009}.

\subsection{Application: computing the Fiedler vector}\label{sec:fiedlerlanczos}

We return to the graph of Chapter~3: two triangles joined by a single edge,
\[
A=\begin{bmatrix}
0&1&1&0&0&0\\ 1&0&1&0&0&0\\ 1&1&0&1&0&0\\
0&0&1&0&1&1\\ 0&0&0&1&0&1\\ 0&0&0&1&1&0
\end{bmatrix},
\qquad L=D-A,\quad D=\mathrm{diag}(2,2,3,3,2,2).
\]
There we ran Jacobi on this Laplacian and observed that the slowest-decaying error mode was the Fiedler
vector. We can perform a similar computation with Lanczos.

\paragraph{Deflating the known null vector.} $L$ is symmetric positive \emph{semi}definite and always has
the eigenpair $(0,\mathbf{1})$, since every row sums to zero; we assume the graph connected, so that
$0$ is simple and $\lambda_2>0$. The Fiedler value $\lambda_2$ is therefore
not extremal, but it sits just above a known eigenvalue. By the proposition above, Lanczos would
converge to $0$ and to $\lambda_{\max}$ long before it resolved $\lambda_2$. The remedy is immediate:
start from a vector orthogonal to $\mathbf{1}$ and keep the iteration in that subspace,
\[
q_1\perp\mathbf{1},\qquad w\;\leftarrow\;w-(\mathbf{1}^\top w)\,\mathbf{1}/n \ \text{ at every step}.
\]
On the deflated subspace $\lambda_2$ \emph{is} the smallest eigenvalue, and the proposition applies again.

\begin{example}[Lanczos on the two-triangle Laplacian]
The exact spectrum is 
$$\{0,\;0.4384,\;3,\;3,\;3,\;4.5616\}.$$ 
Starting from a normalised random
$q_1\perp\mathbf{1}$ and deflating as above, the Ritz values are
\begin{center}
\begin{tabular}{c|l}
\hline
$m$ & Ritz values $\theta_i$ \\
\hline
$1$ & $3.1323$ \\
$2$ & $2.4887,\;3.7217$ \\
$3$ & $\mathbf{0.4384},\;3.0000,\;4.5616$ \\
\hline
\end{tabular}
\end{center}
After three steps the Ritz values are the exact nonzero eigenvalues, and this is not an accident. On
$\mathbf 1^{\perp}$ the eigenvalues of $L$ are $\{0.4384,\,3,\,3,\,3,\,4.5616\}$, that is, only
\emph{three} distinct values, so the minimal polynomial of $L$ relative to any starting vector
(Definition~\ref{def:minpoly-rel}) has degree at most $3$ and Lanczos must terminate at $m=3$
whatever $q_1$ is. The Ritz vector belonging to
$\theta_1=0.4384$ is
\[
v \;=\; Q_3y_1 \;=\; (0.4647,\;0.4647,\;0.2610,\;-0.2610,\;-0.4647,\;-0.4647)^\top,
\]
the Fiedler vector to machine precision, and thresholding it is the simplest form of spectral clustering. Its sign pattern splits the vertices exactly as the two
triangles, $\{1,2,3\}$ against $\{4,5,6\}$, and the small entries $\pm0.2610$ identify nodes $3$ and $4$
as the bottleneck across the cut.
\end{example}

\begin{takeaway}
\begin{itemize}
  \item Lanczos is a Krylov projection: $m$ matrix-vector products compress $A$ into
  an $m\times m$ tridiagonal $T_m$, whose eigenvalues (Ritz values) approximate those of $A$.
  \item Convergence happens \emph{at the ends of the spectrum} first. Extremal eigenvalues come cheaply; a few Lanczos steps therefore bracket $\lambda_{\min},\lambda_{\max}$ hence $\kappa(A)$, which the CG bound of Section~\ref{sec:cg} needs.   
  \item On graph Laplacians, deflating the constant vector makes the Fiedler value extremal, and
  Lanczos delivers the community structure in an optimal number of steps.
\end{itemize}
\end{takeaway}

\section{Krylov Methods and Machine Learning (optional)}
\label{sec:krylovml}

Fitting a linear model to data by least squares is the simplest learning task and one of the
most widely used. The same problem appears well beyond machine learning: reconstructing a sharp image
from a blurred one, recovering a physical parameter from indirect measurements, and combining a forecast
with observations in the data assimilation systems. In each case the matrix
is ill-conditioned and the data are noisy, so the point of view differs from that of the earlier
sections: there the system had a unique solution and we wanted it accurately, whereas here solving to
machine precision reproduces the noise rather than the underlying model.  \footnote{This section is not examinable but
the ideas in it are likely to be useful for the project.}

\subsection{Least squares and the SVD}\label{sec:lsq}

We are given $N$ observations, stored as the rows of a \emph{design matrix} $X\in\mathbb{R}^{N\times d}$,
together with the corresponding targets $y\in\mathbb{R}^N$, and we look for coefficients
$x\in\mathbb{R}^d$ for which $Xx$ is as close to $y$ as possible. This is the linear least squares
problem
\begin{equation}\label{eq:lsq}
\min_{x\in\mathbb{R}^d}\ \|Xx-y\|_2^2
\qquad\Longleftrightarrow\qquad
X^\top X\,x=X^\top y ,
\end{equation}
the equivalence being obtained by setting the gradient of the objective to zero. The matrix $X^\top X$ is
symmetric positive semidefinite, and positive definite whenever $X$ has full column rank, so we are back
in the setting of this chapter.

\begin{tcolorbox}[colback=yellow!5!white,colframe=yellow!50!black]
\begin{proposition}[Least-squares solution via the SVD]\label{prop:lsq}
Let $X\in\mathbb{R}^{N\times d}$ have full column rank, with thin SVD $X=U\Sigma V^\top$ and
$\Sigma=\mathrm{diag}(\sigma_1,\dots,\sigma_d)$, $\sigma_1\ge\cdots\ge\sigma_d>0$. Then the
solution of~\eqref{eq:lsq} is unique and given by
\begin{equation}\label{eq:lssol}
x^{\mathrm{ls}}=\sum_{i=1}^{d}\frac{u_i^\top y}{\sigma_i}\,v_i .
\end{equation}
\end{proposition}
\end{tcolorbox}

\begin{proof}
Since $U$ has orthonormal columns, $X^\top X=V\Sigma U^\top U\Sigma V^\top=V\Sigma^2V^\top$, which is
invertible because every $\sigma_i>0$; the normal equations in~\eqref{eq:lsq} therefore have the unique
solution $x^{\mathrm{ls}}=(X^\top X)^{-1}X^\top y=V\Sigma^{-2}V^\top V\Sigma U^\top y=V\Sigma^{-1}U^\top y$,
which written columnwise is~\eqref{eq:lssol}. (If $\operatorname{rank}X=r<d$ the solution is no longer
unique, but replacing $d$ by $r$ in~\eqref{eq:lssol} gives the one of smallest norm.)
\end{proof}

Formula~\eqref{eq:lssol} also shows where the difficulty lies. Suppose for example that the target $y$ carry noise, $y$ being replaced by $y+e$. By~\eqref{eq:lssol} the $i$-th coefficient 
 changes by $u_i^\top e/\sigma_i$. If
$\sigma_i=10^{-3}$, a noise component of size $10^{-3}$ shifts that coefficient by $1$. The coefficients
the data determine the least well are thus the ones the noise affects the most, which is why $x^{\mathrm{ls}}$
is of little use when $X$ is ill-conditioned.

\subsection{Spectral filters and the truncated SVD}\label{sec:filters}

The remedy is modify the form of~\eqref{eq:lssol} while damping the terms responsible for 
amplification of errors. We therefore replace $x^{\mathrm{ls}}$ by
\begin{equation}\label{eq:filter}
\widehat{x}=\sum_{i=1}^{d}f_i\,\frac{u_i^\top y}{\sigma_i}\,v_i ,\qquad f_i\in[0,1],
\end{equation}
where the coefficients $f_i$ are called \emph{filter factors}: $f_i=1$ retains the $i$-th direction
exactly as least squares had it, $f_i=0$ discards it, and intermediate values damp it. Choosing the
$f_i$, keeping the directions the data determine well and suppressing the others, is what
\emph{regularisation} means \cite{hansenRankDeficientDiscrete1998}.

The simplest choice is a sharp cutoff: for a fixed $k\le d$, set $f_i=1$ for $i\le k$ and $f_i=0$
otherwise, keeping the $k$ best-determined directions. This is the \emph{truncated SVD}, the most
transparent of the filters but also the most expensive: it requires the singular value decomposition of
$X$, out of reach for large $d$, together with a cutoff $k$ chosen by hand. The next two sections present
filters that require neither.

\subsection{Ridge regression: a smooth filter}\label{sec:ridge}

Rather than truncating, one may penalise large coefficients. \emph{Ridge regression} solves
\[
\min_x\ \tfrac{1}{2}\|Xx-y\|_2^2+\tfrac{\lambda}{2}\|x\|_2^2
\qquad\Longleftrightarrow\qquad
(X^\top X+\lambda I)\,x = X^\top y ,
\]
where $\lambda>0$ is a regularisation parameter. The matrix $A_\lambda=X^\top X+\lambda I$ is SPD, so CG
applies directly and never requires $A_\lambda$ itself: $A_\lambda v=X^\top(Xv)+\lambda v$ costs two
passes over the data, and no $d\times d$ matrix is ever formed. This is what makes Krylov methods natural
here. It remains to determine the effect of $\lambda$.

\begin{tcolorbox}[colback=yellow!5!white,colframe=yellow!50!black]
\begin{proposition}[Ridge regression is a spectral filter]\label{prop:ridgefilter}
Under the assumptions of Proposition~\ref{prop:lsq} and for $\lambda>0$, the ridge solution
$x_\lambda=(X^\top X+\lambda I)^{-1}X^\top y$ has the form~\eqref{eq:filter} with
$f_i=\sigma_i^2/(\sigma_i^2+\lambda)\in(0,1)$.
\end{proposition}
\end{tcolorbox}

\begin{proof}
As above, $X^\top X+\lambda I=V(\Sigma^2+\lambda I)V^\top$ and $X^\top y=V\Sigma U^\top y$, so
$$x_\lambda=V(\Sigma^2+\lambda I)^{-1}\Sigma U^\top y=\sum_{i}\frac{\sigma_i}{\sigma_i^2+\lambda}(u_i^\top y)v_i = \frac{\sigma_i^2}{\sigma_i^2+\lambda}\cdot\frac{1}{\sigma_i} (u_i^\top y)v_i$$
which represents the solution in the form~\eqref{eq:filter}.
\end{proof}

Hence $f_i\approx1$ when $\sigma_i^2\gg\lambda$ and $f_i\approx0$ when $\sigma_i^2\ll\lambda$: ridge
regression is a smoothed truncated SVD, in which $\lambda$ rather than $k$ determines where the
transition occurs (Figure~\ref{fig:filters}, left). No singular value decomposition is needed, but a
value of $\lambda$ still has to be supplied.

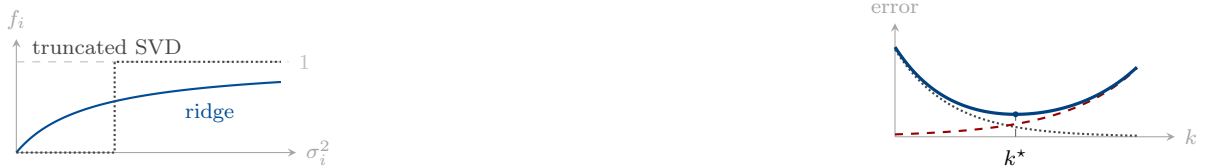
\begin{figure}[htbp]
\centering
\begin{tikzpicture}[>=stealth,scale=0.8]
  \draw[->,gray!70] (0,0) -- (7.4,0) node[right,font=\scriptsize]{$\sigma_i^2$};
  \draw[->,gray!70] (0,0) -- (0,3.0) node[above,font=\scriptsize]{$f_i$};
  \draw[gray!50,dashed] (0,2.4) -- (7.2,2.4) node[right,font=\scriptsize]{$1$};
  \draw[TUblue,thick,smooth,domain=0:7,samples=100] plot(\x,{2.4*\x/(\x+2)});
  \draw[gray!55!black,thick,densely dotted] (0,0) -- (2.6,0) -- (2.6,2.4) -- (7,2.4);
  \node[TUblue,font=\scriptsize,anchor=west] at (4.2,1.05) {ridge};
  \node[gray!55!black,font=\scriptsize,anchor=west] at (0.15,2.8) {truncated SVD};
\end{tikzpicture}
\hfill
\begin{tikzpicture}[>=stealth,scale=0.8]
  \draw[->,gray!70] (0,0) -- (7.4,0) node[right,font=\scriptsize]{$k$};
  \draw[->,gray!70] (0,0) -- (0,3.0) node[above,font=\scriptsize]{error};
  \draw[gray!55!black,thick,densely dotted,domain=0:8,samples=90,variable=\x]
    plot ({0.8*\x},{1.05*2.2*exp(-0.55*\x)});
  \draw[red!60!black,thick,dashed,domain=0:8,samples=90,variable=\x]
    plot ({0.8*\x},{1.05*0.06*exp(0.42*\x)});
  \draw[TUblue!85!black,very thick,domain=0:8,samples=140,variable=\x]
    plot ({0.8*\x},{1.05*(2.2*exp(-0.55*\x)+0.06*exp(0.42*\x))});
  \draw[gray!75!black,dashed] (3.19,0) -- (3.19,0.594);
  \fill[TUblue!85!black] (3.19,0.594) circle (2.2pt);
  \node[font=\scriptsize,anchor=north] at (3.19,-0.08) {$k^\star$};
\end{tikzpicture}
\caption{\emph{Left:} the truncated SVD cuts small-$\sigma_i$ directions off abruptly, ridge damps them
smoothly, and the CG filter lies between the two. \emph{Right:} semi-convergence --- the unresolved part
(dotted) falls while the admitted noise (dashed) grows, so the total (solid) is smallest at $k^\star$.}
\label{fig:filters}
\end{figure}

\subsection{Early stopping: the filter generated by CG}\label{sec:earlystop}

We now consider a third possibility, in which no filter is designed and no parameter selected: solve the
\emph{unregularised} normal equations in~\eqref{eq:lsq} by CG from $x^{(0)}=0$, and stop after $k$ steps.
The iteration then produces a filter of its own.

\begin{tcolorbox}[colback=yellow!5!white,colframe=yellow!50!black]
\begin{proposition}[CG is a spectral filter]\label{prop:cgfilter}
Apply CG to $X^\top X\,x=X^\top y$ from $x^{(0)}=0$, and let $q_k$ be the CG residual polynomial of
Proposition~\ref{prop:cg_poly}, so $\deg q_k\le k$ and $q_k(0)=1$. Then $x^{(k)}$ has the
form~\eqref{eq:filter} with $f_i=1-q_k(\sigma_i^2)$, and among all filters built this way from a
polynomial, that of CG minimises the $A$-norm of the error at every step.
\end{proposition}
\end{tcolorbox}

\begin{proof}
The matrix $A=X^\top X$ is SPD with eigenpairs $(\sigma_i^2,v_i)$ and $Ax=X^\top y$ has solution
$x^{\mathrm{ls}}$. By Proposition~\ref{prop:cg_poly}, $e^{(k)}=q_k(A)e^{(0)}$; since $x^{(0)}=0$ we have
$e^{(0)}=-x^{\mathrm{ls}}$, so $x^{(k)}=x^{\mathrm{ls}}+e^{(k)}=(I-q_k(A))x^{\mathrm{ls}}$. Expanding
$x^{\mathrm{ls}}$ in the basis $\{v_i\}$ and using $q_k(A)v_i=q_k(\sigma_i^2)v_i$ gives the filter; the
optimality is that of Proposition~\ref{prop:cg_poly}, read through this correspondence.
\end{proof}

Stopping early is therefore a form of regularisation. The iteration count $k$ plays the role that $k$
plays for the truncated SVD and $\lambda$ for ridge regression, but at no additional cost, since the
iteration has to be carried out in any case.

To be noted that the first two filters are fixed before the data are seen, whereas
$q_k$ is built from $\mathcal{K}_k(A,X^\top y)$ and hence depends on $y$. CG selects a filter adapted to
the problem at hand, which is why it outperforms a fixed filter at the same $k$.

\begin{remark}[When to stop]
Filtering explains why early stopping helps, but not when to stop. Two effects compete, as shown on the
right of Figure~\ref{fig:filters}:
\begin{itemize}[topsep=2pt,itemsep=0pt,parsep=0pt,partopsep=0pt]
  \item the part of the solution not yet resolved \emph{decreases} with $k$;
  \item the noise admitted through the filter \emph{grows} with $k$, since the filter factors approach
  $1$ on the small $\sigma_i$ for which $1/\sigma_i$ is large.
\end{itemize}
Their sum therefore attains a minimum at some $k^\star$ and increases again beyond it, a phenomenon
known as \emph{semi-convergence}. A stopping rule is thus indispensable.

If the noise level $\delta$ is known, one stops once $\|y-Xx^{(k)}\|_2\le\tau\delta$ with $\tau>1$, since
fitting the data more closely than the noise amounts to fitting the noise; this makes CG a genuine
regularisation method \cite{hankeConjugateGradientTypeMethods1995}. In machine learning $\delta$ is
seldom known, and one stops when the error on a held-out validation set begins to rise.
\end{remark}

\newpage
\section{Exercises}
\paragraph{Problem 1 (A link between fixed point and Krylov iterative methods, exam 2023/2024).}
Let $A,M \in \mathbb{R}^{n\times n}$ be non-singular, $I:=I_n$, and $b \in \mathbb{R}^n$.  
Let $x^0 \in \mathbb{R}^n$ be arbitrary and define $r^k = b - A x^k$ for $k=0,1,2,\dots$.  
Abbreviate $r:=b-Ax^0$. Assume one splits \(A = M - (M-A).\)
\begin{enumerate}
\item[(a)] Show that $Ax=b$ if and only if
\begin{equation}\label{eq:fixedpoint}
x = x + M^{-1}(b - A x).
\end{equation}

\item[(b)] Relation \eqref{eq:fixedpoint} leads to the fixed-point method
\[
x^{k+1} = x^k + M^{-1} r^k.\]
Show that
\[
r^{k+1} = (I - A M^{-1}) r^k.
\]

\item[(c)] Show that
\[
M^{-1} r^{k+1} = (I - M^{-1} A)\, M^{-1} r^k.
\]

\item[(d)] Define the difference $u^k = x^{k+1} - x^k$. Show that
\[
u^k = M^{-1} r^k.
\]

\item[(e)] Define \(U^k = [\,u^0, u^1, \ldots, u^{k-1}\,] \in \mathbb{R}^{n\times k}.\)
Show that
\[
U^k = \big[M^{-1} r,\; (I - M^{-1}A) M^{-1} r,\; \ldots,\; (I - M^{-1}A)^{k-1} M^{-1} r\big].
\]
\end{enumerate}

\paragraph{Problem 2 (Krylov spaces, homework 2023/2024).}
Let $n=5$ and
\[
A =
\begin{bmatrix}
2 & -1 & 0 & 0 & 0 \\
-1 & 2 & -1 & 0 & 0 \\
0 & -1 & 2 & -1 & 0 \\
0 & 0 & -1 & 2 & -1 \\
0 & 0 & 0 & -1 & 2
\end{bmatrix}.
\]
One finds that its Schur decomposition is $A = S D S^{-1}$, with $S$ orthogonal (so $S^{-1}=S^\top$) and
\[
D = \mathrm{diag}\!\bigl(2-\sqrt{3},\,1,\,2,\,3,\,2+\sqrt{3}\bigr),\qquad
S =
\begin{bmatrix}
\frac{1}{\sqrt{12}} & -\frac{1}{2} & \frac{1}{\sqrt{3}} & \frac{1}{2} & -\frac{1}{\sqrt{12}}\\
\frac{1}{2} & -\frac{1}{2} & 0 & -\frac{1}{2} & \frac{1}{2}\\
\frac{1}{\sqrt{3}} & 0 & -\frac{1}{\sqrt{3}} & 0 & -\frac{1}{\sqrt{3}}\\
\frac{1}{2} & \frac{1}{2} & 0 & \frac{1}{2} & \frac{1}{2}\\
\frac{1}{\sqrt{12}} & \frac{1}{2} & \frac{1}{\sqrt{3}} & -\frac{1}{2} & -\frac{1}{\sqrt{12}}
\end{bmatrix}.
\]
Write $S=[s_1,\dots,s_5]$ as a matrix of five columns $s_1,\dots,s_5$.

\begin{enumerate}

\item[(a)] {Determine the minimal polynomial of $A$ and the polynomial $q$ for which $A^{-1} = q(A)$.}

\item[(b)] {Determine $\dim \mathcal{K}_k(A, s_2)$ and $\dim \mathcal{K}_k(A,e_1)$ for $k=1,2,3,\dots$.}

\item[(c)] Let $b = s_2 + s_4$ and $x_0=0$. Does there exist (in exact arithmetic) a $k\in\mathbb{N}$ for which the CG iterate $x_k$ equals the solution $\hat x$ of $A\hat x=b$? If yes, what is that value of $k$?

\end{enumerate}

\paragraph{Problem 3 (Conjugate Gradient and Lanczos by hand).}
Let
\[
A=\begin{bmatrix}2&-1&0\\-1&2&-1\\0&-1&2\end{bmatrix},\qquad b=e_1=\begin{bmatrix}1\\0\\0\end{bmatrix},\qquad x^{(0)}=0 .
\]
\begin{enumerate}
\item[(a)] Perform two steps of CG (Definition~\ref{def:cg}): give $\alpha_0$, $x^{(1)}$, $r^{(1)}$,
$\beta_0$, $p^{(1)}$, $\alpha_1$, $x^{(2)}$ and $r^{(2)}$. Check that $r^{(0)},r^{(1)},r^{(2)}$ are
mutually orthogonal and that $p^{(0)\top}Ap^{(1)}=0$.
\item[(b)] Perform the third step and show that $x^{(3)}=A^{-1}b$. Use
Proposition~\ref{prop:cg_poly} to explain why three steps suffice, and why for this $b$ CG cannot
stop earlier. (The eigenvalues of $A$ are $2-\sqrt2$, $2$, $2+\sqrt2$.)
\item[(c)] Compute $\|e^{(k)}\|_A$ for $k=0,1,2$ and compare the ratios
$\|e^{(k)}\|_A/\|e^{(0)}\|_A$ with the bound $2\bigl(\frac{\sqrt\kappa-1}{\sqrt\kappa+1}\bigr)^k$.
\item[(d)] Run two steps of the Lanczos algorithm (Algorithm~\ref{alg:lanczos}) from $q_1=e_1$ and
verify the formulas of Proposition~\ref{prop:cglanczos}(ii) for $\alpha_1$, $\alpha_2$ and $\beta_1$.
\end{enumerate}

\paragraph{Problem 4 (Predicting iteration counts for the 1D Poisson matrix).}
Let $A=\mathrm{tridiag}(-1,2,-1)\in\mathbb{R}^{n\times n}$, with eigenvalues
$\lambda_k=2-2\cos\bigl(\frac{k\pi}{n+1}\bigr)$, $k=1,\dots,n$ (Chapter~2).
\begin{enumerate}
\item[(a)] Show that $\kappa(A)=\cot^2\!\bigl(\frac{\pi}{2(n+1)}\bigr)$ and that
$\kappa(A)\approx\frac{4(n+1)^2}{\pi^2}$ for large $n$.
\item[(b)] Using the convergence bounds of steepest descent and of CG, estimate the number of
iterations each method needs to reduce $\|e^{(k)}\|_A$ by a factor $10^{6}$, for $n=100$ and
$n=1000$. You may use $\ln\frac{c+1}{c-1}\approx\frac{2}{c}$ for large $c$.
\item[(c)] For $n=100$, compare the CG estimate of (b) with the finite-termination property of
Proposition~\ref{prop:cg_poly}. Which of the two is the more useful prediction, and why is there no
contradiction?
\item[(d)] Run the companion notebook with $n=100$, monitor $\|e^{(k)}\|_A$, and compare the measured
iteration counts with (b) and (c).
\end{enumerate}

\notebookpointer{\nbpath{week5/script\_week5\_1.ipynb}}
\paragraph{Problem 5 (Kantorovich inequality).}
Let $A$ be SPD with condition number $\kappa(A)$. 
Show that for any $v \neq 0$,
\[
1 \le \frac{(v^\top A v)(v^\top A^{-1} v)}{(v^\top v)^2}
\;\leq\; \frac{1}{4}\left(\sqrt{\kappa(A)}+\frac{1}{\sqrt{\kappa(A)}}\right)^2.
\]

\paragraph{Problem 6 (Orthogonality, $A$-conjugacy, best approximation).}
Let $A$ be an SPD matrix. Prove the following fundamental CG properties:
\begin{enumerate}
\item[(a)] (\emph{Residual orthogonality}) $r^{(i)\top}r^{(j)}=0$ for $i\neq j$.
\item[(b)] (\emph{$A$-orthogonality of search directions}) $p^{(i)\top}A p^{(j)}=0$ for $i\neq j$.
\item[(c)] (\emph{Best $A$-norm approximation}) $x^{(k)}$ minimizes $\|x-x^\ast\|_A$ over $x^{(0)}+K_k(A,r^{(0)})$.
\end{enumerate}

\paragraph{Problem 7 (Chebyshev bound and CG convergence rate).}
Let $0<m=\lambda_{\min}(A)\le \lambda\le M=\lambda_{\max}(A)$ and $\kappa(A)=M/m$.
\begin{enumerate}
\item[(a)] Show that
\[
\frac{\|e^{(k)}\|_A}{\|e^{(0)}\|_A}\;\le\;
\min_{\substack{p\in\Pi_k\\ p(0)=1}}\;\max_{\lambda\in[m,M]} |p(\lambda)|,
\]
where $\Pi_k$ denotes polynomials of degree $\le k$.

\item[(b)] We will assume as known the Chebyshev optimality property
$$
\max_{\lambda\in[m,M]}|p_k(\lambda)|= \min_{\substack{p\in\Pi_k\\ p(0)=1}}\;\max_{\lambda\in[m,M]} |p(\lambda)| = \frac{1}{T_k\left(\frac{M+m}{M-m}\right)}
$$
where $T_k$ is the Chebyshev polynomial having the complex representation
$$
T_k(z) = \frac{(z+\sqrt{z^2-1})^k + (z-\sqrt{z^2-1})^k}{2}.
$$
Prove the bound
\[
\max_{\lambda\in[m,M]}|p_k(\lambda)|
\;\le\;2\left(\frac{\sqrt{\kappa(A)}-1}{\sqrt{\kappa(A)}+1}\right)^k,
\]
and conclude the classical CG estimate
\[
\|e^{(k)}\|_A \;\le\; 2\left(\frac{\sqrt{\kappa(A)}-1}{\sqrt{\kappa(A)}+1}\right)^k \|e^{(0)}\|_A.
\]
\end{enumerate}

\notebookpointer{\nbpath{week5/script\_week5\_1.ipynb}}
\paragraph{Problem 8 (Orthogonality of Lanczos vectors).}
Let $A=A^\top$ and generate $\{q_j\}$ by the Lanczos steps
\[
\begin{aligned}
 &w \leftarrow A q_j - \beta_{j-1} q_{j-1},\qquad
 \alpha_j \leftarrow q_j^\top w,\\
 &w \leftarrow w - \alpha_j q_j,\qquad
 \beta_j \leftarrow \|w\|_2,\qquad
 q_{j+1} \leftarrow w/\beta_j\quad(\beta_j\neq0),
\end{aligned}
\]
with $\|q_1\|_2=1$, $q_0=0$, $\beta_0=0$. Show that, in exact arithmetic, the
vectors $q_1,\dots,q_m$ are orthonormal.

\paragraph{Problem 9 (Networks: Lanczos Ritz values and interlacing).}
Let $A \in \mathbb{R}^{n\times n}$ be symmetric, and let $T_m = Q_m^\top A Q_m$ 
be the Lanczos tridiagonal matrix. Show that the eigenvalues of $T_m$ interlace those of $A$, i.e.
\[
\lambda_i(A) \leq \theta_i \leq \lambda_{i+n-m}(A),
\]
where $\theta_i$ are the eigenvalues of $T_m$.
\emph{Hint:} use the Courant--Fischer characterisation
$\lambda_i(A)=\min_{\dim U=i}\,\max_{x\in U\setminus\{0\}}R_A(x)
=\max_{\dim V=n-i+1}\,\min_{x\in V\setminus\{0\}}R_A(x)$, with $R_A(x)=x^\top Ax/x^\top x$, and apply
it to $T_m$ with $x=Q_my$.

\notebookpointer{\nbpath{week5/script\_week5\_2.ipynb}}

\chapter{Krylov Methods for Non-Symmetric Systems}


\section*{Overview}

In Chapter~4 we exploited symmetry to the full but unfortunately none of these results will hold when $A$ is non-symmetric. Convection--diffusion discretisations with upwinding, linearised Navier--Stokes Jacobians, transport and reaction networks, and the Google matrix are all non-symmetric and often \emph{non-normal}, meaning they have no orthogonal eigenbasis at all. This chapter develops the tools for that general setting. Replacing Lanczos by the \emph{Arnoldi} process yields an upper Hessenberg projection $\underline{H}_k$ and a full (long) recurrence; minimising $\|b-Ax\|_2$ over $x^{(0)}+\mathcal{K}_k(A,r^{(0)})$ then gives \emph{GMRES} \cite{saadGMRESGeneralizedMinimal1986}, at a storage and orthogonalisation cost that grows with $k$ which is why restarting is needed in practice. The convergence theory changes character too: for non-normal $A$ the eigenvalues alone no longer predict behaviour, and one turns to the field of values instead. This chapter applies the methodology to a wider class of problem: (1) Arnoldi and GMRES for $Ax=b$; (2) directed networks, where the same Arnoldi decomposition is used for eigenvalue computations; (3) rectangular systems from machine learning, where least squares poses the \emph{same} minimisation but where the Krylov space is built from $A$ and $A^\top$.

\begin{tcolorbox}[colback=softblue,colframe=TUblue!40!black,title=\textbf{Learning objectives},fonttitle=\bfseries]
By the end of this chapter, you should be able to:
\begin{itemize}
    \item Derive the Arnoldi decomposition $AV_k=V_{k+1}\underline{H}_k$, explain why the projection is upper Hessenberg, and recover Lanczos as the symmetric special case.
    \item Formulate GMRES as the least-squares problem $\min_y\|\beta e_1-\underline{H}_ky\|_2$, and state its residual-minimisation, monotonicity, and finite-termination properties.
    \item Express the GMRES residual as $q_k(A)r^{(0)}$ with $q_k(0)=1$, and use eigenvalue and field-of-values bounds to explain why eigenvalues alone mispredict convergence when $A$ is non-normal.
    \item Formulate PageRank as an eigenproblem for the Google matrix, apply $G$ without forming it, explain why $|\lambda_2(G)|\le\alpha$ bounds the power-iteration rate.
    \item Explain why rectangular $A$ forces a Krylov space built from $A$ and $A^\top$, and reduce LSQR to $\min_y\|\gamma e_1-B_ky\|_2$, mirroring the GMRES methodology.
\end{itemize}
\end{tcolorbox}

\clearpage
\section{GMRES and Arnoldi for non-normal matrices}\label{sec:gmres}

We now wish to solve $Ax=b$ when $A$ is no longer symmetric. Two routes are possible in this case. CG requires $A$ to be SPD, so it does not apply. Forming the normal
equations $A^\top Ax=A^\top b$ restores symmetry, but squares the condition number and destroys the
sparsity of $A$; Section~\ref{sec:ml} will explain why this is a bad idea. What remains is the natural
alternative: work directly with $A$, build a Krylov basis of $\mathcal{K}_k(A,r^{(0)})$, and choose from
that subspace the iterate whose residual is smallest in the $2$-norm. 


\begin{tcolorbox}[colback=green!5!white,colframe=green!50!black]
\begin{definition}[Krylov subspace and Arnoldi decomposition]
Given $A\in\mathbb{R}^{n\times n}$ (not necessarily symmetric) and $r^{(0)}=b-Ax^{(0)}$,
\[
\mathcal{K}_k(A,r^{(0)})=\mathrm{span}\{r^{(0)},Ar^{(0)},\ldots,A^{k-1}r^{(0)}\}.
\]
The Arnoldi process \cite{arnoldiPrincipleMinimizedIterations1951} (with $v_1=r^{(0)}/\|r^{(0)}\|_2$) builds an orthonormal basis
$V_{k+1}=[v_1,\ldots,v_{k+1}]$ and an \emph{upper Hessenberg} matrix
$\underline{H}_k\in\mathbb{R}^{(k+1)\times k}$ such that
\[
A V_k = V_{k+1}\,\underline{H}_k,\qquad
\underline{H}_k=\begin{bmatrix} H_k \\ h_{k+1,k}e_k^\top\end{bmatrix},
\]
where $H_k\in\mathbb{R}^{k\times k}$ is upper Hessenberg with entries $h_{ij}=(Av_j,v_i)$ for $1\le i\le j+1\le k$.
\end{definition}
\end{tcolorbox}
\vspace{-0.3cm}

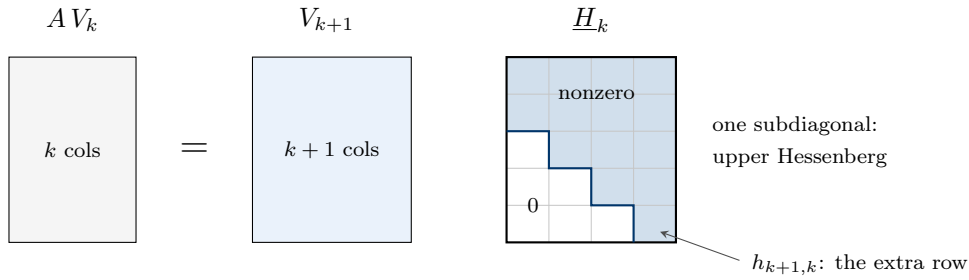
\begin{figure}[htbp]
\centering
\begin{tikzpicture}[>=stealth,scale=1.4]
\node[font=\small] at (0.6,2.55) {$A\,V_k$};
\draw[fill=gray!8] (0,0.45) rectangle (1.2,2.2);
\node[font=\scriptsize] at (0.6,1.325) {$k$ cols};
\node[font=\Large] at (1.75,1.325) {$=$};
\node[font=\small] at (3.0,2.55) {$V_{k+1}$};
\draw[fill=softblue] (2.3,0.45) rectangle (3.8,2.2);
\node[font=\scriptsize] at (3.05,1.325) {$k+1$ cols};
\node[font=\small] at (5.5,2.55) {$\underline{H}_k$};
\foreach \j in {1,...,4}{
  \foreach \i in {1,...,5}{
    \pgfmathtruncatemacro{\nz}{ifthenelse(\i<\j+2,1,0)}
    \ifnum\nz=1
      \fill[TUblue!18] ({4.7+0.4*(\j-1)},{2.2-0.35*\i}) rectangle ({4.7+0.4*\j},{2.2-0.35*(\i-1)});
    \fi
  }
}
\draw[gray!50,very thin,step=0pt] (4.7,0.45) rectangle (6.3,2.2);
\foreach \j in {1,2,3}{\draw[gray!45,very thin] ({4.7+0.4*\j},0.45) -- ({4.7+0.4*\j},2.2);}
\foreach \i in {1,...,4}{\draw[gray!45,very thin] (4.7,{2.2-0.35*\i}) -- (6.3,{2.2-0.35*\i});}
\draw[black,thick] (4.7,0.45) rectangle (6.3,2.2);
\draw[TUblue!70!black,thick]
  (4.7,1.50) -- (5.10,1.50) -- (5.10,1.15) -- (5.50,1.15) -- (5.50,0.80)
  -- (5.90,0.80) -- (5.90,0.45);
\node[font=\scriptsize] at (5.55,1.85) {nonzero};
\node[font=\scriptsize] at (4.95,0.80) {$0$};
\draw[->,gray!60!black] (6.9,0.28) -- (6.15,0.55);
\node[font=\scriptsize,anchor=west] at (6.92,0.24) {$h_{k+1,k}$: the extra row};
\node[font=\scriptsize,anchor=west] at (6.55,1.55) {one subdiagonal:};
\node[font=\scriptsize,anchor=west] at (6.55,1.25) {upper Hessenberg};
\end{tikzpicture}
\caption{The Arnoldi decomposition $AV_k=V_{k+1}\underline{H}_k$, drawn for $k=4$. Applying $A$ to the
orthonormal basis $V_k$ of $\mathcal{K}_k$ never leaves $\mathcal{K}_{k+1}$, so the result factors
through exactly one extra basis vector. Column $j$ of $\underline{H}_k$ is nonzero only in its first
$j+1$ entries, which is what \emph{upper Hessenberg} means; in particular the first column has just two
nonzeros. The matrix is $(k+1)\times k$, the last row containing the single entry $h_{k+1,k}$ that
carries $v_{k+1}$. When $A$ is symmetric, $\underline{H}_k$ is additionally symmetric and therefore
tridiagonal, and Arnoldi reduces to the Lanczos algorithm of Chapter~4.}
\label{fig:arnoldi}
\end{figure}

\begin{tcolorbox}[colback=blue!3!white,colframe=blue!40!black,title=What is an upper Hessenberg matrix?]
A matrix $H=[h_{ij}]$ is \emph{upper Hessenberg} if $h_{ij}=0$ whenever $i>j+1$, i.e.\ all entries strictly below the first subdiagonal are zero. Example ($k=6$):
\[
\footnotesize
H=\begin{bmatrix}
\star & \star & \star & \star & \star \\
\star & \star & \star & \star & \star \\
0      & \star & \star & \star & \star \\
0      & 0      & \star & \star & \star \\
0      & 0      & 0      & \star & \star \\
0      & 0      & 0      & 0      & \star 
\end{bmatrix}.
\]
Arnoldi produces such an $H_k$ because $Av_j$ only has components along $v_1,\ldots,v_{j+1}$.
\end{tcolorbox}

\begin{anecdote}
Large non-symmetric systems were becoming routine in CFD and transport. Memory budgets were counted in \emph{megabytes}, so one could not afford dense factorizations.
Saad and Schultz's 1986 idea \cite{saadGMRESGeneralizedMinimal1986} was elegantly pragmatic: let \emph{Arnoldi} produce an orthonormal snapshot of the action of $A$, and within that small window solve the \emph{best possible} least-squares. 
\end{anecdote}

\subsection{GMRES: minimising the residual over the Krylov space}\label{sec:gmresdef}

Arnoldi supplies a basis; it does not yet say which vector in the subspace to pick. GMRES makes the
choice that CG made in Chapter~4, but in the $2$-norm of the residual rather than the $A$-norm of the
error, the latter no longer being available when $A$ is not SPD.

\begin{tcolorbox}[colback=green!5!white,colframe=green!50!black]
\begin{definition}[GMRES — the principle (method definition)]
Given $A\in\mathbb{R}^{n\times n}$, $b$, and an initial guess $x^{(0)}$ with residual $r^{(0)}=b-Ax^{(0)}$, the \emph{Generalized Minimal Residual method (GMRES)} defines the $k$-th iterate as
\[
x^{(k)} \;=\; \arg\min_{x\in x^{(0)}+\mathcal{K}_k(A,r^{(0)})}\;\|b-Ax\|_2,
\quad
\mathcal{K}_k(A,r^{(0)})=\mathrm{span}\{r^{(0)},Ar^{(0)},\ldots,A^{k-1}r^{(0)}\}.
\]
\end{definition}
\end{tcolorbox}

\begin{tcolorbox}[colback=yellow!5!white,colframe=yellow!50!black]
\begin{proposition}[Equivalence of the GMRES definition and the Arnoldi least-squares]\label{prop:gmres-ls}
Let $r^{(0)}=b-Ax^{(0)}$ and $V_{k+1}=[v_1,\dots,v_{k+1}]$ be an \emph{orthonormal} Arnoldi basis of $\mathcal{K}_k(A,r^{(0)})$ with
\(
v_1=\frac{r^{(0)}}{\|r^{(0)}\|_2}
\)
and set $\beta=\|r^{(0)}\|_2$. Then the GMRES iterate \(x^{(k)}\)
is obtained by choosing $x^{(k)}=x^{(0)}+V_k y^{(k)}$, where $y^{(k)}$ solves the small least-squares
\[
y^{(k)}=\arg\min_{y\in\mathbb{R}^k}\;\Big\|\beta e_1-\underline H_k y\Big\|_2,
\]
and the residual norm satisfies $\|r^{(k)}\|_2=\|\beta e_1-\underline H_k y^{(k)}\|_2$.
\end{proposition}
\end{tcolorbox}
\begin{proof}
Every $x\in x^{(0)}+\mathcal{K}_k(A,r^{(0)})$ can be written \emph{uniquely} as
$x=x^{(0)}+V_k y$ for some $y\in\mathbb{R}^k$, because $V_k$ has orthonormal columns spanning $\mathcal{K}_k$.
For such $x$,
\[
r(y):=b-Ax=b-A(x^{(0)}+V_k y)=r^{(0)}-AV_k y.
\]
By Arnoldi, $AV_k=V_{k+1}\underline H_k$; also $r^{(0)}=\beta v_1=V_{k+1}(\beta e_1)$.
Hence
\[
r(y)=V_{k+1}\bigl(\beta e_1-\underline H_k y\bigr).
\]
Because $V_{k+1}$ has orthonormal columns, it is an isometry on $\mathbb{R}^{k+1}$:
$\|V_{k+1} z\|_2=\|z\|_2$ for all $z\in\mathbb{R}^{k+1}$. Therefore
\[
\|r(y)\|_2=\big\|\beta e_1-\underline H_k y\big\|_2.
\]
Minimizing $\|r(y)\|_2$ over $x\in x^{(0)}+\mathcal{K}_k$ is thus \emph{exactly}
the $(k\!+\!1)\times k$ least-squares
\[
y^{(k)}=\arg\min_{y}\big\|\beta e_1-\underline H_k y\big\|_2,
\]
and the GMRES iterate is $x^{(k)}=x^{(0)}+V_k y^{(k)}$ with residual norm
$\|r^{(k)}\|_2=\|\beta e_1-\underline H_k y^{(k)}\|_2$.
If $h_{k+1,k}=0$, then $\underline H_k$ has a zero last row, so the system is square (and upper
Hessenberg) and the least-squares residual can be made zero; consequently $r^{(k)}=0$ and GMRES terminates in $k$ steps.
\end{proof}

\begin{remark}[Why orthonormality is the key feature]
The proof used $V_{k+1}$ only once, but decisively: because its columns are orthonormal it preserves
norms, so the $n$-dimensional residual and the $(k+1)$-dimensional vector $\beta e_1-\underline H_ky$
have exactly the same length. Had the basis merely been linearly independent, the reduced problem would
have acquired a weighting by $V_{k+1}^\top V_{k+1}$, and minimising it would no longer minimise the true
residual. This is the same reason Lanczos orthonormalises in Chapter~4.
\end{remark}


Proposition~\ref{prop:gmres-ls} turns each GMRES step into a small least-squares problem, but solving it
from scratch at every step would be wasteful: $\underline H_k$ differs from $\underline H_{k-1}$ by a
single new column. Maintaining a QR factorisation of $\underline H_k$ incrementally costs one rotation
per step, and returns the residual norm as a by-product.

\begin{tcolorbox}[colback=green!5!white,colframe=green!50!black]
\begin{definition}[Givens rotations (notation used by GMRES)]
For scalars $a,b$, a \emph{Givens rotation} $G(c,s)$ is the $2\times2$ orthogonal matrix
\[
G(c,s)=\begin{bmatrix} c & s \\ -s & c \end{bmatrix},\qquad
c=\frac{a}{\sqrt{a^2+b^2}},\ \ s=\frac{b}{\sqrt{a^2+b^2}},
\]
which maps $\begin{bmatrix}a\\ b\end{bmatrix}\mapsto\begin{bmatrix}\sqrt{a^2+b^2}\\ 0\end{bmatrix}$.
In GMRES we apply $G_k$ to \emph{rows} $k$ and $k\!+\!1$ of the tall matrix $\underline H_k$ to zero the subdiagonal entry $h_{k+1,k}$, and we apply the same rotation to the right-hand side $g=\beta e_1$ to keep $g=Q^\top(\beta e_1)$ up to step $k$.
\end{definition}
\end{tcolorbox}

\begin{tcolorbox}[colback=green!5!white,colframe=green!50!black]
\begin{definition}[Full GMRES (Givens form) — construction]
Given $x^{(0)}$, set $r^{(0)}=b-Ax^{(0)}$, $\beta=\|r^{(0)}\|_2$, $v_1=r^{(0)}/\beta$.
For $k=1,2,\dots$ until convergence:

\begin{enumerate}
\item \textbf{Arnoldi expansion.} Compute $w=Av_k$. For $i=1,\dots,k$:
      \[
      h_{i,k} = v_i^\top w,\quad w \gets w - h_{i,k} v_i.
      \]
      Set $h_{k+1,k}=\|w\|_2$. If $h_{k+1,k}=0$ (lucky breakdown), stop; else $v_{k+1}=w/h_{k+1,k}$.
\item \textbf{Update small LS (Givens).} Apply the stored rotations $G_1,\dots,G_{k-1}$ to the new column of $\underline H_k$, then form $G_k$ to annihilate $h_{k+1,k}$ and apply $G_k$ to $g=\beta e_1$; this maintains the QR of $\underline H_k$ and yields $\|r^{(k)}\|_2=|g_{k+1}|$.

\item \textbf{Residual and test.} The current residual norm is
      \[
      \|r^{(k)}\|_2 = |g_{k+1}|,
      \]
      available \emph{without} forming $x^{(k)}$. If $\|r^{(k)}\|_2/\|b\|_2 \le \text{tol}$, recover $y^{(k)}$ by back–substitution in the $k\times k$ upper triangular $R$ and set $x^{(k)}=x^{(0)}+V_k y^{(k)}$.
\end{enumerate}
\end{definition}
\end{tcolorbox}

\begin{tcolorbox}[colback=yellow!5!white,colframe=yellow!50!black]
\begin{proposition}[GMRES basics in one shot]
Let $A\in\mathbb{R}^{n\times n}$, $x^{(0)}$ be an initial guess, $r^{(0)}=b-Ax^{(0)}$, and let
$V_{k+1}$ and $\underline H_k$ come from $k$ steps of Arnoldi with
$v_1=r^{(0)}/\|r^{(0)}\|_2$ and $AV_k=V_{k+1}\underline H_k$; set $\beta=\|r^{(0)}\|_2$.
The GMRES iterate $x^{(k)}=x^{(0)}+V_k y^{(k)}$, where
$y^{(k)}=\arg\min_{y}\|\beta e_1-\underline H_k y\|_2$. This satisfies:

\begin{enumerate}
\item[\emph{(i)}] \textbf{Orthogonality.} $r^{(k)}\perp A\mathcal{K}_k(A,r^{(0)})$.
\item[\emph{(ii)}] \textbf{Monotonicity (full GMRES).} $\|r^{(k+1)}\|_2\le\|r^{(k)}\|_2$.
\item[\emph{(iii)}] \textbf{Residual polynomial and finite termination.}
There exists $q_k\in\Pi_k$ with $q_k(0)=1$ such that
\(
r^{(k)}=q_k(A)\,r^{(0)},
\)
and $q_k$ minimizes $\|q(A)r^{(0)}\|_2$ over all such polynomials.
If $m$ is the minimal polynomial of $A$ relative to $r^{(0)}$ (Definition~\ref{def:minpoly-rel})
with $\deg m=d$, then GMRES terminates in at most $d$ steps (i.e., $r^{(d)}=0$).
\end{enumerate}
\end{proposition}
\end{tcolorbox}
\begin{proof}
    See Problem 4.
\end{proof}
\subsection{Convergence: why eigenvalues are not enough}\label{sec:fov}

\begin{figure}[htbp]
\centering
\begin{tikzpicture}[>=stealth,scale=1.0]
\begin{scope}
  \node[font=\small] at (0,2.15) {normal $A$};
  \draw[->,gray!70] (-1.9,0)--(1.9,0) node[right,font=\scriptsize]{$\Re$};
  \draw[->,gray!70] (0,-1.6)--(0,1.7) node[above,font=\scriptsize]{$\Im$};
  \foreach \x/\y in {0.9/0.35, 1.1/-0.25, 1.25/0.6, 0.75/-0.55, 1.4/0.1, 1.0/0.0}{
    \fill[TUblue] (\x,\y) circle (1.3pt);}
  \draw[TUblue,thick] (1.07,0.04) circle (0.62);
  \node[font=\scriptsize,align=center] at (0,-2.05) {eigenvalues clustered away from $0$\\$\Rightarrow$ GMRES converges fast};
\end{scope}
\begin{scope}[xshift=6.4cm]
  \node[font=\small] at (0,2.15) {non-normal $A$};
  \draw[->,gray!70] (-1.9,0)--(1.9,0) node[right,font=\scriptsize]{$\Re$};
  \draw[->,gray!70] (0,-1.6)--(0,1.7) node[above,font=\scriptsize]{$\Im$};
  \fill[TUblue!12,draw=TUblue,thick] (0.35,0) circle (1.25);
  \foreach \x/\y in {0.9/0.35, 1.1/-0.25, 1.25/0.6, 0.75/-0.55, 1.4/0.1, 1.0/0.0}{
    \fill[TUblue] (\x,\y) circle (1.3pt);}
  \fill (0,0) circle (1.4pt) node[below left=-2pt,font=\scriptsize]{$0$};
  \node[font=\scriptsize,anchor=west] at (0.55,1.2) {$F(A)$};
  \node[font=\scriptsize,align=center] at (0,-2.05) {same eigenvalues, but $0\in F(A)$\\$\Rightarrow$ GMRES may stagnate};
\end{scope}
\end{tikzpicture}
\caption{Why eigenvalues are not enough. For a normal matrix the eigenvalues alone control
$\|q(A)\|_2=\max_j|q(\lambda_j)|$, so a tight cluster away from the origin guarantees rapid GMRES
convergence. For a non-normal matrix the eigenvector basis is ill-conditioned and the relevant object is
the field of values $F(A)=\{x^\ast Ax:\|x\|_2=1\}$, a convex set containing $\sigma(A)$. If $F(A)$ stays
in a disk that excludes the origin, the Crouzeix--Palencia inequality \cite{crouzeixNumericalRangeSpectral2017} gives a genuine convergence rate;
if the origin lies inside $F(A)$, no such bound is available, and GMRES can stagnate for $n-1$ steps even
with all eigenvalues equal to $1$ (see the example at the end of this section).}
\label{fig:fov}
\end{figure}
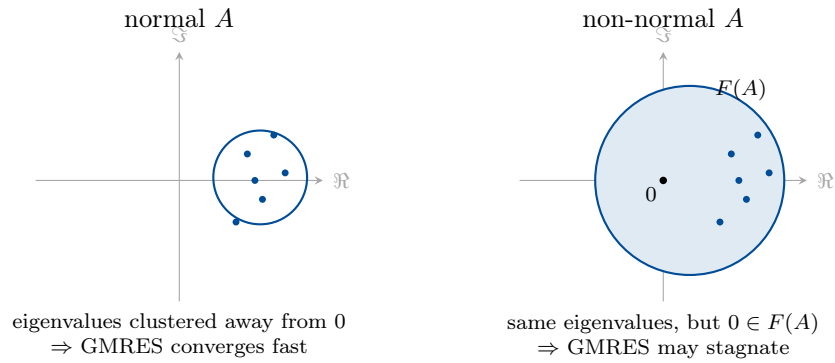

For CG, Chapter~4 gave a single convergence bound in terms of $\kappa(A)$, and that was the end of the
story. For GMRES there is no such clean result. The natural first attempt is to bound the residual using
the eigenvalues of $A$, and it works whenever $A$ is diagonalisable with a well-conditioned eigenvector
matrix, but it becomes vacuous exactly when $A$ is strongly non-normal, which is the case we introduced
the method for.

\begin{tcolorbox}[colback=yellow!5!white,colframe=yellow!50!black]
\begin{proposition}[Diagonalizable-$A$ bound]
If $A=V\Lambda V^{-1}$, then
\[
\frac{\|r^{(k)}\|_2}{\|r^{(0)}\|_2}\ \le\
\kappa_2(V)\,\min_{\substack{q\in\Pi_k\\ q(0)=1}}\max_i |q(\lambda_i)|.
\]
\end{proposition}
\end{tcolorbox}

\begin{proof}
$r^{(k)}=q_k(A)r^{(0)}=V q_k(\Lambda)V^{-1}r^{(0)}$, so
$\|r^{(k)}\|\le\|V\|\,\|V^{-1}\|\,\max_i|q_k(\lambda_i)|\,\|r^{(0)}\|$.
\end{proof}

\begin{tcolorbox}[colback=green!5!white,colframe=green!50!black]
\begin{definition}[Field of values (numerical range)]
\[
F(A)=\left\{\frac{x^\ast A x}{x^\ast x}\ :\ x\in\mathbb{C}^n\setminus\{0\}\right\}.
\]
\end{definition}
\end{tcolorbox}
In words, $F(A)$ is the set of all Rayleigh quotients of $A$. It is compact and convex, and it contains
every eigenvalue of $A$, as one sees by taking $x$ to be an eigenvector. The point of replacing the
spectrum by $F(A)$ is that the latter is insensitive to the conditioning of the eigenvectors: its shape
and position control $\|q(A)\|_2$ directly, and hence the GMRES residual through
$r^{(k)}=q_k(A)r^{(0)}$.
\begin{tcolorbox}[colback=yellow!5!white,colframe=yellow!50!black]
\begin{proposition}[Disk-in-$F(A)$ bound]
If $F(A)\subset\{\,|z-c|\le s\,\}$ and $0\notin\{\,|z-c|\le s\,\}$, then
\[
\|r^{(k)}\|_2 \ \le\ (1+\sqrt{2})\Big(\frac{s}{|c|}\Big)^{\!k}\,\|r^{(0)}\|_2.
\]
\end{proposition}
\end{tcolorbox}

\begin{proof}
Crouzeix-Palencia: $\|q(A)\|\le (1+\sqrt2)\max_{z\in F(A)}|q(z)|$.
Test the degree-$k$ polynomial $q_k(z)=(1-z/c)^k$; then
$\max_{F(A)}|q_k|\le(s/|c|)^k$.
\end{proof}
\paragraph{When eigenvalues mislead: a companion matrix.}
The following example shows how far the spectrum alone can be from the truth. For a monic polynomial
$p(z)=z^n+c_{n-1}z^{n-1}+\cdots+c_0$,
the companion matrix
\[
C(p)=\begin{bmatrix}
0 & 1 & 0 & \cdots & 0\\
0 & 0 & 1 & \cdots & 0\\
\vdots & \vdots & \ddots & \ddots & \vdots\\
0 & 0 & \cdots & 0 & 1\\
- c_0 & - c_1 & \cdots & - c_{n-2} & - c_{n-1}
\end{bmatrix}
\]
has eigenvalues equal to the roots of $p$. These matrices are typically \emph{highly nonnormal}.
Take for example $p(z)=(z-1)^3=z^3-3z^2+3z-1$. Then
\[
A=C(p)=\begin{bmatrix} 0&1&0\\[2pt] 0&0&1\\[2pt] 1&-3&3\end{bmatrix},
\quad \Lambda(A)=\{1,1,1\}.
\]
Run GMRES on $Ax=f$ with $x^{(0)}=0$ and $f=e_1$.
One observes
\[
\|r^{(0)}\|_2=\|r^{(1)}\|_2=\|r^{(2)}\|_2=1,\qquad \|r^{(3)}\|_2=0,
\]
i.e., \emph{no decrease} for two steps, then exact termination at $k=3$. Moreover,
$e_1^\top A e_1=0$, so $0\in F(A)$ and the disk bound above \emph{does not apply} here. This illustrates that tightly clustered eigenvalues (even all at $1$) do not guarantee fast GMRES if the matrix is nonnormal; the field of values and nonnormality matter. In fact any nonincreasing residual curve whatsoever can be realised by some matrix with any prescribed spectrum \cite{Greenbaum1996}. A systematic account of non-normal behaviour is given in \cite{trefethenSpectraPseudospectra2005}.

\section{Networks: PageRank as an Eigenproblem}\label{sec:pagerank}

\begin{anecdote}
The original PageRank algorithm, developed by Larry Page and Sergey Brin in the late 1990s,  was essentially an eigenvector computation on the web graph. At the time, the web already had tens of millions of pages,  making full eigen-decompositions infeasible.  Instead, iterative methods such as power iteration were the only scalable option. This numerical perspective i.e. treating web search as a giant linear algebra problem, was one of the key insights that allowed Google to leap ahead of earlier search engines. It remains a striking example of how classical numerical methods became  transformative in an entirely new application domain.
\end{anecdote}

Chapter~4 used the Krylov subspace twice: to solve $Ax=b$ and, through Lanczos, to compute eigenvalues of a
\emph{symmetric} matrix. Section~\ref{sec:gmres} has just repeated the first of these for non-symmetric
$A$; this section repeats the second. The Arnoldi decomposition $AV_m=V_{m+1}\underline H_m$ is again all
we need for the eigenvalues of the small matrix $H_m$ rather than for a
linear solve. The application we envision here is ranking the nodes of a directed graph, where the matrix is
non-symmetric by construction and far too large to factorise.

\begin{tcolorbox}[colback=green!5!white,colframe=green!50!black]
\begin{definition}[Eigenvector centrality (directed graphs)]
Let $A\in\mathbb{R}^{n\times n}$ be the (possibly non-symmetric) adjacency matrix.
\emph{Right eigenvector centrality} assigns scores $x\ge0$ to nodes by solving
\[
A x=\lambda_1 x, \qquad \|x\|_1=1,
\]
while \emph{left eigenvector centrality} assigns scores $y\ge0$ from \(A^\top y=\lambda_1 y, \qquad \|y\|_1=1.\)
\end{definition}
\end{tcolorbox}

\paragraph{The Google matrix}\label{sec:googlematrix}

Let $P \in \mathbb{R}^{n \times n}$ be the column-stochastic matrix where
\[
P_{ij} =
\begin{cases}
\dfrac{1}{d_j}, & \text{if page $j$ links to page $i$}, \\
0, & \text{otherwise},
\end{cases}
\]
with $d_j$ the out-degree of page $j$.
The Google matrix is defined as
\[
G = \alpha P + (1-\alpha)\frac{1}{n}\mathbf{1}\mathbf{1}^\top, \qquad 0 < \alpha < 1.
\]
By construction, $G$ is stochastic\footnote{Every entry is a value from 0 to 1, and the numbers in each column (or row) add up to exactly 1} and irreducible\footnote{An irreducible matrix is a square matrix that cannot be rearranged into a block upper-triangular form using row and column permutations. It means you cannot split the matrix into smaller independent parts.}; the Perron-Frobenius theorem guarantees a unique dominant eigenvector $x$ with positive entries such that $\|x\|_1 = 1$, called the \emph{PageRank vector} \cite{langvilleGooglePageRankBeyond2006,pagePageRankCitationRanking1999}.

\paragraph{Efficient application of the Google matrix} We should notice that $G$ is \emph{dense}: every entry of $\mathbf{1}\mathbf{1}^\top$ is nonzero, so storing $G$
for a web-scale graph is out of the question, but it doesn't have to be stored. For any $x$ with
$\mathbf{1}^\top x=1$ the rank-one perturbation term acts as a constant vector,
\[
Gx=\alpha Px+\frac{1-\alpha}{n}\mathbf{1}\,(\mathbf{1}^\top x)=\alpha Px+\frac{1-\alpha}{n}\mathbf{1},
\]
so one application of $G$ costs one sparse product with $P$ plus a vector addition. The power iteration
of Chapter~3 (Algorithm~\ref{alg:power}) applied to the dominant eigenvector is therefore simply
\begin{equation}\label{eq:pagerank-power}
x^{(k+1)}=\alpha P x^{(k)}+\frac{1-\alpha}{n}\mathbf{1},
\end{equation}
one sparse matrix--vector product per step, two vectors of storage, and no factorisation anywhere. This
is the entire algorithm that ranked the early web.

\begin{example}[Google Page Rank]
Consider a graph of three webpages with links and find which is the most central node
\begin{center}
\begin{tikzpicture}[->, >=stealth, node distance=2.5cm, thick]
  \tikzstyle{node}=[circle, draw, minimum size=1cm, font=\bfseries]

  \node[node] (1) {1};
  \node[node, right of=1] (2) {2};
  \node[node, below of=2] (3) {3};

  \draw (1) -- (2);
  \draw (2) -- (3);
  \draw (3) -- (1);
  \draw (3) -- (2);
\end{tikzpicture}
\end{center}
The adjacency matrix is
\[
A = \begin{bmatrix}
0 & 0 & 1 \\
1 & 0 & 1 \\
0 & 1 & 0
\end{bmatrix}, \quad
P = \begin{bmatrix}
0 & 0 & 1/2 \\
1 & 0 & 1/2 \\
0 & 1 & 0
\end{bmatrix}.
\]
With damping factor $\alpha = 0.85$, 
\[
G \;=\; \alpha P + \frac{1-\alpha}{n}\mathbf{1}\mathbf{1}^\top
      \;=\; 0.85\,P + 0.05\,\mathbf{1}\mathbf{1}^\top
      \;=\;
\begin{bmatrix}
0.05 & 0.05 & 0.475\\
0.90 & 0.05 & 0.475\\
0.05 & 0.90 & 0.05
\end{bmatrix},
\quad\text{(column-stochastic)}.
\]
Start from the uniform vector $x^{(0)}=\tfrac{1}{3}(1,1,1)^\top$. Iterating with the power method gives
\[
x^\star \approx \begin{bmatrix} 0.215\\ 0.397\\ 0.388 \end{bmatrix}.
\]
hence node $2$ is the most central.
\end{example}

\subsection{The role of the damping factor}\label{sec:damping}

First, what $\alpha$ \emph{means}. The Google matrix models a random surfer who, at each step, follows
one of the links on the current page with probability $\alpha$, and with probability $1-\alpha$ abandons
the page and jumps to a page chosen randomly. So $\alpha$ is a modelling choice.

It is also the only thing that determines the cost of computing the ranking.

\begin{tcolorbox}[colback=yellow!5!white,colframe=yellow!50!black]
\begin{proposition}[Role of the damping factor \cite{haveliwalaSecondEigenvalueGoogle2003}]\label{prop:damping}
Let $P$ be column-stochastic, $0<\alpha<1$, and $G=\alpha P+\frac{1-\alpha}{n}\mathbf{1}\mathbf{1}^\top$.
Then every eigenvalue of $G$ other than $\lambda_1=1$ satisfies $|\lambda|\le\alpha$. Consequently the
power iteration \eqref{eq:pagerank-power} converges linearly with asymptotic factor at most $\alpha$,
independently of the size of the graph.
\end{proposition}
\end{tcolorbox}

\begin{proof}
Because $G$ is column-stochastic, $\mathbf{1}^\top G=\mathbf{1}^\top$: the vector $\mathbf{1}$ is a
\emph{left} eigenvector for the eigenvalue $1$. Let $Gx=\lambda x$ with $\lambda\neq1$. Left and right
eigenvectors belonging to different eigenvalues are orthogonal, so $\mathbf{1}^\top x=0$. Then we have,
\[
Gx=\alpha Px+\frac{1-\alpha}{n}\mathbf{1}\,(\mathbf{1}^\top x)=\alpha Px ,
\]
so $x$ is also an eigenvector of $P$, with eigenvalue $\lambda/\alpha$. Since $P$ is stochastic all its
eigenvalues satisfy $|\mu|\le1$, hence $|\lambda|\le\alpha$.
\end{proof}

Two consequences are resulting from this but they pull in opposite directions.

\begin{itemize}[topsep=2pt,itemsep=1pt,parsep=0pt,partopsep=0pt]
  \item \textbf{The method is usable at scale.} The bound involves $\alpha$ and nothing else, in
  particular not $n$. Doubling the size of the web does not cost a single extra iteration, which is why
  \eqref{eq:pagerank-power} scaled from millions of pages to billions unchanged.
  \item \textbf{The model and the cost are linked.}  Taking $\alpha$ close to $1$ makes the ranking
  faithful to the link structure and simultaneously makes the iteration slow: to reach a tolerance
  $\varepsilon$ needs about $\log\varepsilon/\log\alpha$ steps, which behaves like
  $\log(1/\varepsilon)/(1-\alpha)$ as $\alpha\to1$ and therefore blows up.
\end{itemize}

The three-node example makes this concrete. Its link matrix $P$ has characteristic polynomial
$\lambda^3-\tfrac12\lambda-\tfrac12=(\lambda-1)(\lambda^2+\lambda+\tfrac12)$, so besides $\lambda=1$ its
eigenvalues are $(-1\pm i)/2$, of modulus $1/\sqrt2$. By the proof above the corresponding eigenvalues of
$G$ are exactly $\alpha$ times these, giving $|\lambda_2(G)|=\alpha/\sqrt2\approx0.707\,\alpha$, which is
what Table~\ref{tab:damping} reports. The bound $|\lambda_2|\le\alpha$ is thus attained here up to the
factor $|\lambda_2(P)|$; on a real web graph $|\lambda_2(P)|$ is very close to $1$ and the bound is
essentially sharp.

\begin{table}[htbp]
\centering
\renewcommand{\arraystretch}{1.2}
\begin{tabular}{|c|c|c|}
\hline
$\alpha$ & $|\lambda_2(G)|$ & power iterations to $10^{-8}$\\
\hline
$0.50$ & $0.354$ & $17$\\
$0.85$ & $0.601$ & $36$\\
$0.95$ & $0.672$ & $45$\\
$0.99$ & $0.700$ & $52$\\
\hline
\end{tabular}
\caption{The damping factor on the three-node example of this section, where
$|\lambda_2(G)|=\alpha/\sqrt2$ exactly. The last column is the number of iterations
observed to gain eight digits, in agreement with the asymptotic estimate
$8\log 10/\log(1/|\lambda_2(G)|)$. The value $\alpha=0.85$ used by Google is a compromise
between fidelity to the link structure and cost, not a mathematical constant.}
\label{tab:damping}
\end{table}

\subsection{Beyond the power iteration: the Arnoldi process}\label{sec:notenough}

Power iteration is attractive because \eqref{eq:pagerank-power} is one line of code, but it delivers a
single eigenvector and its rate is fixed at $\alpha$: it cannot be accelerated, and it has no way of
certifying its answer.

\begin{tcolorbox}[colback=blue!3!white,colframe=blue!40!black]
\begin{remark}[When is power iteration \emph{not} enough?]
Four situations call for Arnoldi instead.
\begin{itemize}
\item \textbf{Small spectral gap:} if $|\lambda_1|\approx|\lambda_2|$, power converges very slowly.
\item \textbf{Multiple/clustered targets or several eigenpairs:} power returns at most one direction; Arnoldi delivers many Ritz pairs at once.
\item \textbf{Nonnormal matrices:} eigenvectors may be poorly conditioned and power can behave erratically; Arnoldi tracks subspaces and provides residuals $\|Au-\theta u\|$.
\item \textbf{Stopping guarantees:} Arnoldi gives rigorous a posteriori residuals $|h_{m+1,m}||e_m^\top y|$.
\end{itemize}
\end{remark}
\end{tcolorbox}

The remedy is the Arnoldi process, which we run on $A$ (here $A=G$, or a shifted
problem), and from the decomposition $AV_m=V_{m+1}\underline H_m$ take an eigenpair $(\theta,y)$ of the
square part $H_m$. The Ritz pair is then $(\theta,u)$ with $u=V_my$, and the associated residual is
available for free,
\begin{equation}\label{eq:ritz-residual}
\|Au-\theta u\|_2=|h_{m+1,m}|\,|e_m^\top y| ,
\end{equation}
because 
$$
Au-\theta u=(AV_m-V_mH_m)y=h_{m+1,m}v_{m+1}(e_m^\top y)$$ 
and $\|v_{m+1}\|_2=1$. 

\begin{takeaway}
\begin{itemize}
\item PageRank is the dominant eigenvector of the Google matrix $G$; Perron--Frobenius guarantees it is
unique and positive.
\item $G$ is dense but never formed: one application costs one sparse product with $P$, so the power
iteration \eqref{eq:pagerank-power} is one line of code.
\item The damping factor is the single difficult parameter; $|\lambda_2(G)|\le\alpha$ bounds the
convergence rate independently of the size of the graph.
\item Power iteration gives one eigenpair at a fixed rate and no error guarantees; in turn, Arnoldi i.e. the same
decomposition built in Section~\ref{sec:gmres} gives several Ritz pairs at once, plus the residual
$|h_{m+1,m}||e_m^\top y|$ for free.
\end{itemize}
\end{takeaway}

\section{Krylov Methods for Large Least Squares (optional)}\label{sec:ml}

Section~\ref{sec:gmres} solved $Ax=b$ for square non-symmetric $A$ by minimising $\|b-Ax\|_2$ over a
Krylov subspace. Least squares asks for the minimiser of the \emph{same} quantity when $A$ is
rectangular, and this section is about what changes. The answer is a single structural point: with $A$
rectangular the sequence $r^{(0)},Ar^{(0)},A^2r^{(0)},\dots$ does not even make sense, because $Ar^{(0)}$
lives in the wrong space. The Krylov space must be built by alternating $A$ and $A^\top$, and everything
else follows from this single change. \footnote{This section is not examinable but the ideas in it are likely to be useful for the project.}

\paragraph{A motivating example: sparse regression on text}\label{sec:textreg}

Consider a linear model fitted to a corpus of documents: row $i$ of the design matrix $X$ holds the word
counts of document $i$ over a vocabulary of $n$ terms, and $b_i$ is a response, as for example a rating in recommender systems, a
click-through rate, or a price. Fitting by least squares means solving
\[
\min_x\ \|Xx-b\|_2,\qquad X\in\mathbb{R}^{m\times n},
\]
with $m$ in the millions and $n$ in the tens of thousands. Three features of this problem decide the
choice of algorithm, and all of them are frequently encountered (recommender systems, genomics, and inverse problems).

\begin{itemize}
\item \textbf{$X$ is rectangular}, so ``solve the linear system'' is not even the right statement; there
is no square matrix in sight and generally no exact solution.
\item \textbf{$X$ is extremely sparse.} A document contains a few hundred of the $n$ vocabulary terms, so
$X$ has a fraction of a percent of its entries nonzero.
\item \textbf{$X$ is ill-conditioned.} Word frequencies are heavy-tailed and features are correlated
(``machine'' and ``learning'' co-occur), so the singular values of $X$ decay over many orders of
magnitude.
\end{itemize}

In what follows will briefly explain why the obvious route fails. 
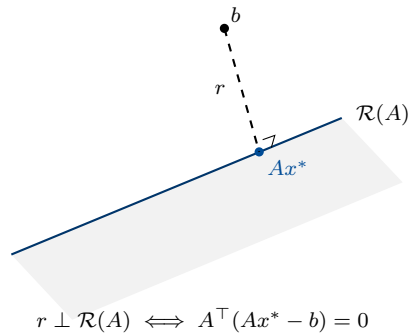
\begin{figure}[htbp]
\centering
\begin{tikzpicture}[>=stealth,scale=1.15]
  \fill[gray!10] (-1.9,-0.95) -- (1.9,0.62) -- (2.6,-0.12) -- (-1.2,-1.7) -- cycle;
  \draw[TUblue!70!black,thick] (-1.9,-0.95) -- (1.9,0.62);
  \node[font=\scriptsize,anchor=west] at (1.95,0.66) {$\mathcal{R}(A)$};
  \coordinate (b) at (0.55,1.65);
  \coordinate (Ax) at (0.95,0.23);
  \fill (b) circle (1.5pt) node[above right=-2pt,font=\scriptsize]{$b$};
  \fill[TUblue] (Ax) circle (1.5pt) node[below right=-1pt,font=\scriptsize]{$Ax^\ast$};
  \draw[dashed,thick] (b) -- (Ax);
  \node[font=\scriptsize,anchor=east] at (0.68,0.95) {$r$};
  \draw (0.99,0.37) -- (1.13,0.42) -- (1.09,0.28);
  \node[font=\scriptsize,align=center,anchor=north] at (0.3,-1.45)
     {$r\perp\mathcal{R}(A)\iff A^\top(Ax^\ast-b)=0$};
\end{tikzpicture}
\caption{Least squares as orthogonal projection. The minimiser of $\|Ax-b\|_2$ is the point of the column
space $\mathcal{R}(A)$ closest to $b$, so the residual is orthogonal to $\mathcal{R}(A)$, which leads
exactly to the normal equations $A^\top Ax=A^\top b$. LSQR realises the same projection without forming $A^\top A$.}
\label{fig:lsq-projection}
\end{figure}

As explained in Chapter 4, the first ideas would be to solve $\min_x\|Xx-b\|_2$ by differentiating: the minimiser satisfies the \emph{normal equations}
$X^\top Xx=X^\top b$, a square SPD system of size $n$, to which the CG of Chapter~4 applies immediately.
Indeed Section~\ref{sec:krylovml} did precisely this: it read least squares as an SPD problem and ran CG
on $X^\top X$ in order to exhibit the filter factors $1-q_k(\sigma_i^2)$. That was the right way to
\emph{understand} regularisation, because the filtering is visible only in the eigenbasis of
$X^\top X$. It is not the right way to \emph{compute}. At any serious scale one never forms
$X^\top X$, and never runs CG on it, for two independent reasons.

\paragraph{Sparsity is destroyed.} Two terms contribute a nonzero to $(X^\top X)_{jk}$ as soon as they
co-occur in a \emph{single} document, so the Gram matrix is far denser than $X$. On a representative
sparse design matrix, for example $20\,000\times5000$ at $0.2\%$ density,  forming $X^\top X$ produces a matrix
that is $7.7\%$ dense, nearly ten times as many nonzeros as $X$ itself, and this ratio grows with the
corpus.

\paragraph{Conditioning is squared.} Since the singular values of $X^\top X$ are $\sigma_i^2$,
\[
\kappa_2(X^\top X)=\kappa_2(X)^2 .
\]
Running CG on $X^\top X$ (the method known as CGNR) therefore converges at a rate governed by
$\kappa_2(X)$ rather than $\sqrt{\kappa_2(X)}$, and in floating point the damage can be worse than slow:
information present in $X$ can be destroyed by the time $X^\top X$ has been formed.

\begin{tcolorbox}[colback=softred,colframe=red!50!black]
\textbf{A classical warning.} Let $\varepsilon>0$ and
\[
X=\begin{bmatrix}1&1\\ \varepsilon&0\\ 0&\varepsilon\end{bmatrix},
\qquad
X^\top X=\begin{bmatrix}1+\varepsilon^2&1\\ 1&1+\varepsilon^2\end{bmatrix}.
\]
$X$ has rank $2$ for every $\varepsilon\neq0$. But if $\varepsilon<\sqrt{u}$, where $u\approx2.2\cdot10^{-16}$
is the unit roundoff, then $1+\varepsilon^2$ rounds to $1$ and the Gram matrix is
$\left[\begin{smallmatrix}1&1\\1&1\end{smallmatrix}\right]$, so exactly singular. In double precision
this already happens at $\varepsilon=10^{-9}$: the rank information is annihilated by
forming $X^\top X$.
\end{tcolorbox}

\noindent What we want, then, is a method with the \emph{convergence} of CG on $X^\top X$ but the
\emph{conditioning} of $X$, that means one that uses $X$ and $X^\top$ only as operators and never multiplies them
together. That is exactly what Golub--Kahan bidiagonalisation provides
\cite{golubCalculatingSingularValues1965}. From here on we revert to the neutral notation
$A\in\mathbb{R}^{m\times n}$ for the rectangular matrix, to keep the parallel with
Section~\ref{sec:gmres} visible.

\subsection{Golub--Kahan bidiagonalisation and LSQR}\label{sec:gk}

\begin{tcolorbox}[colback=green!5!white,colframe=green!50!black]
\begin{definition}[Golub--Kahan bidiagonalization (core recurrence)]
Choose $u_1=\frac{b}{\|b\|_2}$ and set $\beta_0=0$, $v_0=0$. For $k=1,2,\dots$:
\[
\begin{aligned}
&\text{(i) } r = A^\top u_k - \beta_{k-1} v_{k-1},\quad \alpha_k=\|r\|_2,\quad v_k = r/\alpha_k,\\
&\text{(ii) } p = A v_k - \alpha_k u_k,\quad \beta_k=\|p\|_2,\quad u_{k+1} = p/\beta_k.
\end{aligned}
\]
Since $Av_k=\alpha_ku_k+\beta_ku_{k+1}$, we obtain $A V_k = U_{k+1} B_k$ with
$B_k\in\mathbb{R}^{(k+1)\times k}$ \emph{lower} bidiagonal, carrying $\alpha_1,\dots,\alpha_k$ on the
diagonal and $\beta_1,\dots,\beta_k$ on the subdiagonal.
\end{definition}
\end{tcolorbox}

\begin{figure}[htbp]
\centering
\begin{tikzpicture}[>=stealth,scale=1.4]
  \tikzstyle{vv}=[circle,draw=TUblue!70!black,fill=softblue,minimum size=8.5mm,font=\small]
  \tikzstyle{uu}=[circle,draw=gray!70!black,fill=gray!12,minimum size=8.5mm,font=\small]
  \node[vv] (v1) at (0,0) {$v_1$};
  \node[uu] (u1) at (1.7,-1.0) {$u_1$};
  \node[vv] (v2) at (3.4,0) {$v_2$};
  \node[uu] (u2) at (5.1,-1.0) {$u_2$};
  \node[vv] (v3) at (6.8,0) {$v_3$};
  \node[font=\small] at (8.2,-0.5) {$\cdots$};
  \draw[->,thick] (v1) -- node[above right=-3pt,font=\scriptsize]{$A$} (u1);
  \draw[->,thick] (u1) -- node[above left=-3pt,font=\scriptsize]{$A^\top$} (v2);
  \draw[->,thick] (v2) -- node[above right=-3pt,font=\scriptsize]{$A$} (u2);
  \draw[->,thick] (u2) -- node[above left=-3pt,font=\scriptsize]{$A^\top$} (v3);
  \node[font=\scriptsize,anchor=west] at (-1.0,1.05) {$\{v_j\}$ orthonormal in $\mathbb{R}^n$ (row space)};
  \node[font=\scriptsize,anchor=west] at (-1.0,-2.0) {$\{u_j\}$ orthonormal in $\mathbb{R}^m$ (column space)};
\end{tikzpicture}
\caption{Golub--Kahan bidiagonalisation. Two orthonormal sequences are grown in alternation, one in the
row space and one in the column space, each step costing one product with $A$ and one with $A^\top$. The
coefficients assemble into a lower bidiagonal $B_k$ with $AV_k=U_{k+1}B_k$.
This is Lanczos applied implicitly to $A^\top A$ but the squaring never happens numerically, which is
precisely why LSQR is stable where CG on the normal equations is not.}
\label{fig:golub-kahan}
\end{figure}
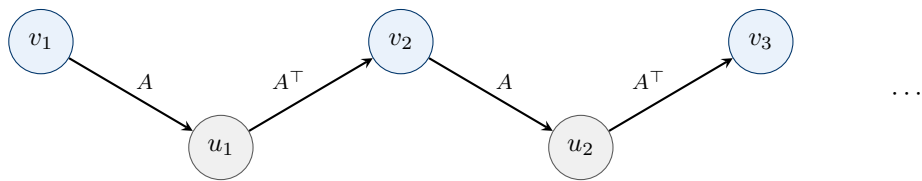

The reason this construction is worth the trouble is that it reduces the least-squares problem to a tiny
one of exactly the shape met in Section~\ref{sec:gmres}.

\begin{tcolorbox}[colback=yellow!5!white,colframe=yellow!50!black]
\begin{proposition}[LSQR as a small least-squares problem]\label{prop:lsqr}
Write $\gamma=\|b\|_2$, so that $u_1=b/\gamma$, and let $k$ steps of Golub--Kahan produce $V_k$,
$U_{k+1}$ and the bidiagonal $B_k\in\mathbb{R}^{(k+1)\times k}$ with $AV_k=U_{k+1}B_k$. Then for any
$x=V_ky$,
\[
\|b-Ax\|_2=\bigl\|\gamma e_1-B_ky\bigr\|_2 ,
\]
and consequently the minimiser of $\|b-Ax\|_2$ over $x\in\mathcal{R}(V_k)$ is $x^{(k)}=V_ky^{(k)}$ with
\[
y^{(k)}=\arg\min_{y\in\mathbb{R}^k}\bigl\|\gamma e_1-B_ky\bigr\|_2 .
\]
\end{proposition}
\end{tcolorbox}

\begin{proof}
Since $b=\gamma u_1=U_{k+1}(\gamma e_1)$ and $AV_ky=U_{k+1}B_ky$,
\[
b-AV_ky=U_{k+1}\bigl(\gamma e_1-B_ky\bigr).
\]
The columns of $U_{k+1}$ are orthonormal, so $U_{k+1}$ is an isometry on $\mathbb{R}^{k+1}$ and the norm
is unchanged. Minimising over $y$ gives the second claim.
\end{proof}

Compare this with Proposition~\ref{prop:gmres-ls}: GMRES minimises $\|\beta e_1-\underline H_ky\|_2$ against an upper
Hessenberg $\underline H_k$, LSQR minimises $\|\gamma e_1-B_ky\|_2$ against a bidiagonal $B_k$. The
algorithms differ only in how the small matrix is generated and, consequently, in how cheaply it is
factorised: a bidiagonal $B_k$ is reduced to triangular form by a \emph{single} Givens rotation per step.
This is what allows the LSQR iterate to be updated by short recurrences in $\mathcal{O}(m+n)$ work
and storage, with no need to keep $V_k$. Mathematically it produces the same iterates as CG
on $A^\top A$, but the squaring never happens numerically, which is why it is stable where CGNR is not.

\subsection{Early stopping and the truncated SVD}\label{sec:lsqrfilter}

\begin{figure}[htbp]
\centering
\begin{tikzpicture}[>=stealth,scale=1.2]
  \draw[->,gray!70] (0,0)--(6.4,0) node[right,font=\scriptsize]{singular value index $i$};
  \draw[->,gray!70] (0,0)--(0,2.9) node[above,font=\scriptsize]{$\sigma_i$};
  \foreach \i/\h in {0.5/2.55, 1.1/1.95, 1.7/1.42, 2.3/1.0, 2.9/0.68, 3.5/0.44, 4.1/0.27, 4.7/0.15, 5.3/0.08, 5.9/0.04}{
    \draw[TUblue,line width=3pt] (\i,0) -- (\i,\h);}
  \draw[gray!80,thick,rounded corners,dashed] (0.2,-0.12) rectangle (2.6,2.75);
  \node[font=\scriptsize,align=center,anchor=north] at (1.4,-0.25) {captured in the\\first LSQR steps};
  \node[font=\scriptsize,align=center,anchor=north] at (4.6,-0.25) {reached late; these are the\\directions that amplify noise};
\end{tikzpicture}
\caption{LSQR as an iterative truncated SVD. Because the Krylov space is built from $A^\top b$ and grows
by powers of $A^\top A$, components associated with large singular values are resolved first. Stopping
early therefore acts as a spectral filter, suppressing the small-$\sigma_i$ directions that magnify noise
in an ill-posed problem. The iteration count plays the role of a regularisation parameter, the same
mechanism as truncating the SVD, but obtained without ever computing one.}
\label{fig:lsqr-filter}
\end{figure}
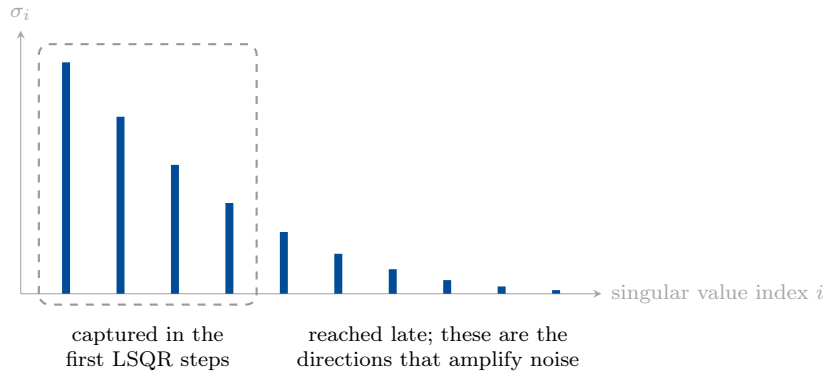

Running $\ell=k+p$ steps of Golub--Kahan on a sparse term--document matrix $M$ and computing the (tiny)
SVD $B_\ell=\hat U\Sigma\hat V^\top$ gives approximate singular triplets of $M$ through
$U\approx U_{\ell+1}\hat U_{:,1:k}$ and $V\approx V_\ell\hat V_{:,1:k}$, with a small oversampling $p$ to
improve accuracy of the last few. This is \emph{latent semantic analysis}: the leading singular directions of $M$ define a
low-dimensional semantic space in which related documents cluster, and it is obtained from a few dozen
matrix--vector products rather than a dense $\mathcal{O}(mn\min(m,n))$ factorisation.

\begin{figure}[h]
\centering
\begin{tikzpicture}[scale=1.2,>=stealth,thick]
\draw[->] (-0.2,0)--(4.2,0); \draw[->] (0,-0.2)--(0,3.2);
\fill[blue] (0.8,2.4) circle (1.5pt) node[above] {\scriptsize sports};
\fill[blue] (1.0,2.1) circle (1.5pt);
\fill[blue] (0.9,1.9) circle (1.5pt);
\fill[red] (3.2,0.7) circle (1.5pt) node[right] {\scriptsize finance};
\fill[red] (2.9,0.9) circle (1.5pt);
\fill[red] (3.4,0.6) circle (1.5pt);
\node at (3.5,3.0) {\footnotesize $V_k$ doc-embedding (2D sketch)};
\end{tikzpicture}
\caption{After truncated SVD, documents cluster in a low-dimensional semantic space.}
\end{figure}
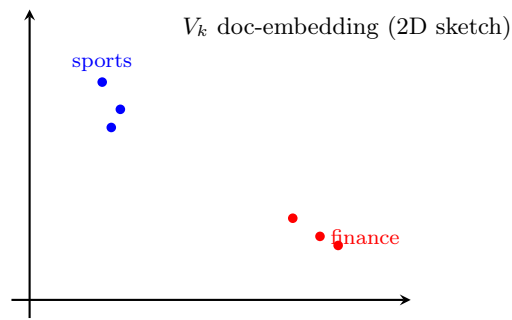

\begin{takeaway}
\begin{itemize}
\item Rectangular $A$ forces the Krylov space to be built from $A$ \emph{and} $A^\top$; Golub--Kahan
bidiagonalisation is the resulting recurrence.
\item LSQR reduces the least-squares problem to $\min_y\|\gamma e_1-B_ky\|_2$, structurally identical
to the GMRES problem of Section~\ref{sec:gmres}, but with a bidiagonal in place of a Hessenberg matrix.
\item Never form $A^\top A$: the geometry of the normal equations is right, the arithmetic is not.
\item The same bidiagonalisation, stopped early, acts as a spectral filter (regularisation); kept and
factorised, it yields a truncated SVD.
\end{itemize}
\end{takeaway}

\newpage
\section{Exercises}

\paragraph{Problem 1 (Krylov spaces, resit exam 2024/2025).}
Consider the matrix
\[ A = \begin{bmatrix} 2 & -1 & 0 \\ -1 & 2 & -1 \\ 0 & -1 & 2 \end{bmatrix}, \quad b = \begin{bmatrix} 1 \\ 0 \\ 1 \end{bmatrix}. \]
\begin{enumerate}
    \item[(a)] Define the Krylov space $K_k(A, b)$ for $k = 1, 2, \dots$.

    \item[(b)] Compute the Krylov spaces $K_1(A, b)$, $K_2(A, b)$, and $K_3(A, b)$. Determine their dimensions.

    \item[(c)] Explain the general principle of the GMRES method for solving $Ax = b$. How many iterations are needed for convergence?
\end{enumerate}

\paragraph{Problem 2 (Companion matrices and GMRES convergence, exam October 2024).}
Let
\[
A=\begin{bmatrix}
0 & 1 & & & 0\\
 & 0 & 1 & & \\
 & & \ddots & \ddots & \\
 & & & 0 & 1\\
 c_0 & c_1 & \cdots & c_{n-2} & c_{n-1}
\end{bmatrix}.
\]
and the polynomial $p(z) = z^n -\sum_{j=0}^{n-1} c_j z^j$.
Show that:
\begin{itemize}
\item[(a)] $p(z)$ is its characteristic and minimal polynomial.
\item[(b)] Let
\[
 A=\begin{bmatrix}0&1&0\\[2pt]0&0&1\\[2pt]1&-3&3\end{bmatrix},
 \qquad p(z)=(z-1)^3=z^3-3z^2+3z-1,\qquad b=e_1=\begin{bmatrix}1\\0\\0\end{bmatrix}.
\]
Compute the Krylov spaces $K_k(A,b), k=1,2,3$.
\item[(c)] Compute the GMRES iterates $u_k$ for $Au=b$ with $u_0=0$ and conclude on how many iterations are needed for convergence.
\item[(d)] Can this result be generalised to an $n\times n$ companion matrix? Explain your answer.
\end{itemize}

\notebookpointer{\nbpath{week6/script\_week6\_3.ipynb}}
\paragraph{Problem 3 (Arnoldi and GMRES by hand).}
Let
\[
A=\begin{bmatrix}1&2&0\\0&1&2\\0&0&1\end{bmatrix},\qquad b=e_3=\begin{bmatrix}0\\0\\1\end{bmatrix},\qquad x^{(0)}=0 .
\]
\begin{enumerate}
\item[(a)] Run two steps of the Arnoldi process from $v_1=b$. Give $v_2$, $v_3$ and the
$3\times2$ matrix $\underline H_2$, and check that $AV_2=V_3\underline H_2$.
\item[(b)] Using Proposition~\ref{prop:gmres-ls}, compute the GMRES iterates $x^{(1)},x^{(2)}$ and
the residual norms $\|r^{(1)}\|_2,\|r^{(2)}\|_2$. For $k=1$, check that the Givens rotation which
annihilates $h_{21}$ returns the same residual norm as $|g_2|$.
\item[(c)] Show that $x^{(3)}$ is the exact solution and compute it.
\item[(d)] All eigenvalues of $A$ equal $1$. For which matrix with this spectrum would GMRES
converge in one step? Show, with $x=\tfrac12(1,-\sqrt2,1)^\top$, that $0\in F(A)$, and explain why
the disk-in-$F(A)$ bound gives no information here.
\end{enumerate}
\notebookpointer{\nbpath{week6/script\_week6\_3.ipynb}}

\paragraph{Problem 4 (GMRES basics in one shot).}
Let $A\in\mathbb{R}^{n\times n}$, $x^{(0)}$ be an initial guess, $r^{(0)}=b-Ax^{(0)}$, and let
$V_{k+1}$ and $\underline H_k$ come from $k$ steps of Arnoldi with
$v_1=r^{(0)}/\|r^{(0)}\|_2$ and $AV_k=V_{k+1}\underline H_k$; set $\beta=\|r^{(0)}\|_2$.
The GMRES iterate $x^{(k)}=x^{(0)}+V_k y^{(k)}$, where
$y^{(k)}=\arg\min_{y}\|\beta e_1-\underline H_k y\|_2$. Prove that it satisfies:

\begin{enumerate}
\item[\emph{(i)}] \textbf{Orthogonality.} $r^{(k)}\perp A\mathcal{K}_k(A,r^{(0)})$.
\item[\emph{(ii)}] \textbf{Monotonicity (full GMRES).} $\|r^{(k+1)}\|_2\le\|r^{(k)}\|_2$.
\item[\emph{(iii)}] \textbf{Residual polynomial and finite termination.}
There exists $q_k\in\Pi_k$ with $q_k(0)=1$ such that
\(
r^{(k)}=q_k(A)\,r^{(0)},
\)
and $q_k$ minimizes $\|q(A)r^{(0)}\|_2$ over all such polynomials.
If $m$ is the minimal polynomial of $A$ relative to $r^{(0)}$ (Definition~\ref{def:minpoly-rel})
with $\deg m=d$, then GMRES terminates in at most $d$ steps (i.e., $r^{(d)}=0$).
\end{enumerate}

\paragraph{Problem 5 (Field of values and GMRES convergence).}
Let $A\in\mathbb{C}^{n\times n}$ and suppose its field of values (numerical range) \(F(A) \) is contained in a disk $\{z:\,|z-c|\le s\}$ with $0\notin \{z:\,|z-c|\le s\}$.
Let $r^{(k)}$ denote the $k$-th GMRES residual for $Ax=b$ with some initial guess $x^{(0)}$ and $r^{(0)}=b-Ax^{(0)}$.

\begin{enumerate}[label=(\alph*)]
\item Using Problem~4(iii), show that $\|r^{(k)}\|_2\le\|p(A)\|_2\,\|r^{(0)}\|_2$ for every
$p\in\Pi_k$ with $p(0)=1$.
\item Prove the ``disk-in-$F(A)$'' bound
\[
\|r^{(k)}\|_2 \;\le\; C \left(\frac{s}{|c|}\right)^k \|r^{(0)}\|_2,
\]
for a universal constant $C$. (\emph{Hint:} Use Crouzeix--Palencia: $\|p(A)\|\le (1+\sqrt{2})\max_{z\in F(A)}|p(z)|$ and test $p(z)=(1-z/c)^k$.)
\item Explain briefly why this bound can be more descriptive than eigenvalue-only bounds for highly non-normal matrices. Give a concrete example of a matrix with tightly clustered eigenvalues but slow GMRES (Problems~2 and~3 contain two) and relate this to the geometry of $F(A)$.
\end{enumerate}

\paragraph{Problem 6 (Arnoldi, Ritz values, and a posteriori eigen-residuals).}
Let $A\in\mathbb{C}^{n\times n}$ and suppose $k$ steps of Arnoldi with $v_1=\frac{r^{(0)}}{\|r^{(0)}\|_2}$ produce
\[
AV_k = V_{k+1} \underline H_k,\qquad
V_{k+1}=\begin{bmatrix}v_1&\cdots&v_{k+1}\end{bmatrix},\quad
\underline H_k=\begin{bmatrix} H_k \\ h_{k+1,k}\, e_k^\top \end{bmatrix},
\]
where $V_{k+1}$ has orthonormal columns, $\underline H_k\in\mathbb{C}^{(k+1)\times k}$ is upper Hessenberg, and $H_k\in\mathbb{C}^{k\times k}$ is its leading square part.
\begin{enumerate}[label=(\alph*)]
\item Show that if $(\theta,y)$ is an eigenpair of $H_k$ with $\|y\|_2=1$, then $(\theta, v)$ with $v:=V_k y$ is a Ritz pair for $A$ (i.e.\ $\theta$ is a Ritz value and $v$ a Ritz vector).
\item Prove the a posteriori residual formula
\[
\|A v - \theta v\|_2 \;=\; |h_{k+1,k}|\,|e_k^\top y|.
\]
\item Interpret this bound: under what condition is $\theta$ a good eigenvalue approximation to $A$? How does this relate to (i) small subdiagonal element $|h_{k+1,k}|$ (a ``lucky'' or near-breakdown), and (ii) the last component $|e_k^\top y|$ of the eigenvector of $H_k$?
\end{enumerate}

\notebookpointer{\nbpath{week6/script\_week6\_3.ipynb}}
\paragraph{Problem 7 (PageRank on a four-page web).}
Four pages link as follows: page $1\to2,3$; page $2\to3$; page $3\to1$; page $4\to1,3$.
\begin{enumerate}
\item[(a)] Write down the link matrix $P$ (with $P_{ij}=1/d_j$ when page $j$ links to page $i$) and check
that it is column-stochastic. Page $4$ receives no links: what does that say about row $4$ of $P$, and
what will make its PageRank nonzero nevertheless?
\item[(b)] Take $\alpha=0.85$. Without writing out the dense matrix $G$, give the formula that computes
$Gx$ from $Px$, and say how many nonzeros are touched per application.
\item[(c)] Starting from $x^{(0)}=\tfrac14(1,1,1,1)^\top$, perform two steps of the iteration
\eqref{eq:pagerank-power} and rank the four pages.
\item[(d)] Using Proposition~\ref{prop:damping}, bound the number of iterations needed to reduce the
error by $10^{-6}$, first for $\alpha=0.85$ and then for $\alpha=0.99$. Comment on what the ranking gains
and what the computation loses when $\alpha$ is increased.
\item[(e)] Suppose a fifth page is added that links to nothing. Show that $P$ is then no longer
column-stochastic, and describe a way of repairing it that keeps the cost of applying $G$ unchanged.
\end{enumerate}

\notebookpointer{\nbpath{week6/script\_week6\_2.ipynb}}

\chapter{Preconditioning and Accelerating Solvers}

\section*{Overview}

Chapters~4 and~5 showed that Krylov methods are optimal over the subspace they build, but that their speed is dictated by the spectrum of $A$: CG needs $\mathcal{O}(\sqrt{\kappa(A)})$ iterations, and for the PDE matrices of Chapter~2 this means $\mathcal{O}(h^{-1})$ iterations, growing without bound as the mesh is refined. Optimality over a Krylov subspace cannot fix a badly conditioned problem. This is what preconditioning does. Instead of solving $Ax=b$ we solve an equivalent system involving $M^{-1}A$, where $M$ approximates $A$ but is cheap to invert, and we choose $M$ so that the transformed spectrum is clustered and the condition number is bounded independently of the mesh size; broad surveys of the subject are given in \cite{benziPreconditioningTechniques2002,wathenPreconditioning2015}. This chapter covers: (1) the algebraic framework of left, right, and symmetric preconditioning, preconditioned CG, and simple choices such as Jacobi and incomplete factorisations; (2) Schwarz domain decomposition, where the preconditioner is built from independent subdomain solves and the overlap width controls the convergence rate; (3) two-level and multigrid ideas, where a coarse-grid correction removes the smooth error components that relaxation cannot touch, yielding preconditioners whose quality is \emph{independent} of $h$.

\begin{tcolorbox}[colback=softblue,colframe=TUblue!40!black,title=\textbf{Learning objectives},fonttitle=\bfseries]
By the end of this chapter, you should be able to:
\begin{itemize}
    \item Explain why Krylov convergence rates make preconditioning necessary for mesh-refined PDE problems, and state the left, right, and symmetric preconditioning formulations.
    \item Derive the preconditioned conjugate gradient algorithm and explain why it requires only solves with $M$, never $M^{-1/2}$, and why $M$ must be SPD.
    \item Construct and assess simple preconditioners (Jacobi, SSOR, incomplete Cholesky/LU) and estimate their effect on $\kappa(M^{-1}A)$ for the discrete Laplacian.
 
    \item Formulate the alternating and additive Schwarz methods, analyse the 1D convergence factor in terms of the overlap $\delta$, and explain the role of a coarse space in restoring scalability.
    \item Compute the smoothing factor of weighted Jacobi, explain why relaxation damps oscillatory but not smooth error, and describe how coarse-grid correction and the resulting two-grid/multigrid cycle give an $h$-independent preconditioner.
\end{itemize}
\end{tcolorbox}

\clearpage
\section{Motivation and Basic Preconditioners}\label{sec:basic}

Preconditioning accelerates Krylov iterations by modifying the spectrum of the coefficient matrix. We aim to accelerate the solution of
\(
A x = b, \, A \in \mathbb{R}^{n\times n},
\)
using a matrix \(M\) that approximates \(A\) but is cheaper to invert.

\begin{tcolorbox}[colback=green!5!white,colframe=green!50!black]
\begin{definition}[Left, right, and symmetric preconditioning]
We have three different options:
\begin{itemize}[leftmargin=*]
  \item \textbf{Left preconditioning:}
  \[
  M_L^{-1}Ax=M_L^{-1}b.
  \]
  This changes the residual norm seen by the Krylov method.

  \item \textbf{Right preconditioning:}
  \[
  A M_R^{-1} y=b,\qquad x=M_R^{-1}y.
  \]
  This preserves the original residual $b-Ax$ but modifies the search space.

  \item \textbf{Symmetric preconditioning:}
  for SPD problems, one uses an SPD matrix $M$ and rewrites the system as
  \[
  M^{-1/2}AM^{-1/2}\,\tilde x=M^{-1/2}b,
  \qquad \tilde x=M^{1/2}x.
  \]
  This preserves symmetry and positive definiteness.
\end{itemize}
\end{definition}
\end{tcolorbox}

\begin{keyidea}
For SPD problems, the convergence of CG depends on the condition number of the operator.
A good preconditioner replaces $A$ by an operator whose eigenvalues are more tightly clustered,
ideally near $1$:
\[
M^{-1}A \approx I,
\qquad
\kappa(M^{-1}A)\ll \kappa(A).
\]
The gain can already be quantified exactly on simple model matrices.
\end{keyidea}

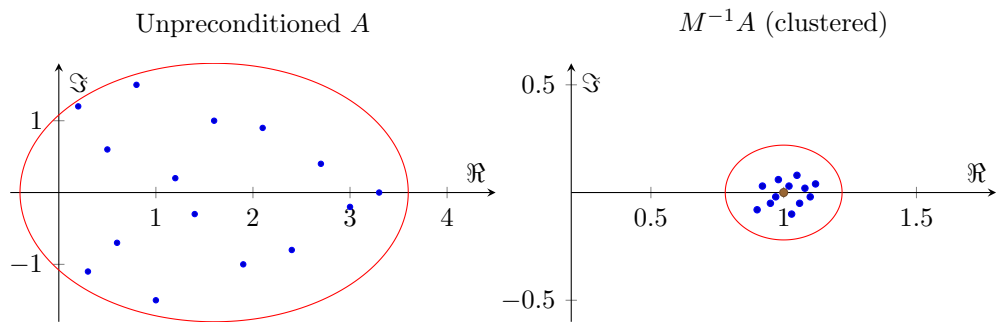
\begin{figure}[h!]
\centering
\begin{tikzpicture}
\begin{axis}[
name=left,
width=0.5\textwidth,
height=5cm,
axis lines=middle,
xlabel={$\Re$}, ylabel={$\Im$},
xmin=-0.5, xmax=4.5, ymin=-1.8, ymax=1.8,
title={Unpreconditioned $A$},
ticklabel style={font=\small},
label style={font=\small},
title style={font=\small}
]
\addplot+[only marks, mark=*, mark size=1pt]
coordinates{(0.2,1.2) (0.3,-1.1) (0.5,0.6) (0.6,-0.7) (0.8,1.5)
(1.0,-1.5) (1.2,0.2) (1.4,-0.3) (1.6,1.0) (1.9,-1.0)
(2.1,0.9) (2.4,-0.8) (2.7,0.4) (3.0,-0.2) (3.3,0.0)};
\addplot+[domain=0:360,samples=200, no marks]
({1.6+2.0*cos(x)},{0+1.8*sin(x)});
\end{axis}

\begin{axis}[
at={(left.east)}, anchor=west, xshift=1.0cm,
width=0.45\textwidth,
height=5cm,
axis lines=middle,
xlabel={$\Re$}, ylabel={$\Im$},
xmin=0.2, xmax=1.8, ymin=-0.6, ymax=0.6,
title={$M^{-1}A$ (clustered)},
ticklabel style={font=\small},
label style={font=\small},
title style={font=\small}
]
\addplot+[only marks, mark=*, mark size=1.2pt]
coordinates{(1.00,0.00) (1.05,0.08) (0.95,-0.05) (1.10,-0.02) (0.92,0.03)
(1.03,-0.10) (0.98,0.06) (1.08,0.02) (0.90,-0.08) (1.12,0.04)
(0.97,-0.02) (1.06,-0.05) (1.02,0.03)};
\addplot+[domain=0:360,samples=200, no marks]
({1+0.22*cos(x)},{0+0.22*sin(x)});
\addplot+[only marks, mark=*, mark size=1.5pt]
coordinates{(1,0)};
\end{axis}
\end{tikzpicture}
\caption{Effect of preconditioning on the spectrum: eigenvalues of $A$ (left) are spread, whereas those of $M^{-1}A$ (right) are tightly clustered around 1.}
\end{figure}

\subsection{Preconditioning SPD problems}
For SPD problems, preconditioning is most naturally understood through the symmetrically
preconditioned operator
\[
\widetilde A = M^{-1/2} A M^{-1/2}.
\]
CG is then applied to
\[
\widetilde A \widetilde x = \widetilde b,
\qquad
\widetilde x = M^{1/2}x,
\qquad
\widetilde b = M^{-1/2}b.
\]
Equivalently, in the original variables, the search space is the preconditioned Krylov space
\[
x^{(0)}+\mathcal K_k(M^{-1}A,M^{-1}r^{(0)}).
\]
\begin{tcolorbox}[colback=yellow!5!white,colframe=yellow!50!black]
\begin{proposition}[Preconditioned CG]
\label{prop:pcg}
Let $A$ and $M$ be SPD. Applying CG to
\[
M^{-1/2}AM^{-1/2}\,\widetilde x = M^{-1/2}b
\]
is equivalent to the following iteration in the original variables:
\[
\begin{aligned}
&r^{(0)}=b-Ax^{(0)},\qquad z^{(0)}=M^{-1}r^{(0)},\qquad p^{(0)}=z^{(0)},\\
&q^{(k)}=Ap^{(k)},\qquad
\alpha_k=\frac{\langle r^{(k)},z^{(k)}\rangle}{\langle p^{(k)},q^{(k)}\rangle},\\
&x^{(k+1)}=x^{(k)}+\alpha_k p^{(k)},\qquad
r^{(k+1)}=r^{(k)}-\alpha_k q^{(k)},\\
&z^{(k+1)}=M^{-1}r^{(k+1)},\qquad
\beta_{k+1}=\frac{\langle r^{(k+1)},z^{(k+1)}\rangle}{\langle r^{(k)},z^{(k)}\rangle},\\
&p^{(k+1)}=z^{(k+1)}+\beta_{k+1}p^{(k)}.
\end{aligned}
\]
Moreover,
\[
\operatorname{spec}(M^{-1/2}AM^{-1/2})=\operatorname{spec}(M^{-1}A),
\]
so the usual CG convergence estimate holds with $\kappa(M^{-1}A)$ replacing $\kappa(A)$.
\end{proposition}
\end{tcolorbox}

\begin{proof}[Proof sketch]
Set
\[
\widetilde x=M^{1/2}x,\qquad
\widetilde r=M^{-1/2}r,\qquad
\widetilde p=M^{1/2}p.
\]
Then standard CG on $\widetilde A=M^{-1/2}AM^{-1/2}$ gives
\[
\alpha_k
=
\frac{\langle \widetilde r^{(k)},\widetilde r^{(k)}\rangle}
{\langle \widetilde p^{(k)},\widetilde A\widetilde p^{(k)}\rangle}
=
\frac{\langle r^{(k)},M^{-1}r^{(k)}\rangle}
{\langle p^{(k)},Ap^{(k)}\rangle},
\]
which yields the stated recursion with $z^{(k)}=M^{-1}r^{(k)}$.
Similarly,
\[
\beta_{k+1}
=
\frac{\langle \widetilde r^{(k+1)},\widetilde r^{(k+1)}\rangle}
{\langle \widetilde r^{(k)},\widetilde r^{(k)}\rangle}
=
\frac{\langle r^{(k+1)},z^{(k+1)}\rangle}
{\langle r^{(k)},z^{(k)}\rangle}.
\]
Finally,
\[
M^{-1}A = M^{-1/2}(M^{-1/2}AM^{1/2})M^{-1/2}
\]
is similar to $M^{-1/2}AM^{-1/2}$, so they have the same eigenvalues.
\end{proof}

\begin{remark}
Preconditioning does not alter the exact solution $x^*$ but changes the trajectory of residual minimization. The effectiveness depends on how well $M$ approximates $A$ in the subspace explored by the Krylov method.
\end{remark}

\subsection{A general framework for preconditioners}

\paragraph{Jacobi preconditioning.}
Let
\(
A=D+L+U,
\)
where $D$ is the diagonal of $A$.
The \emph{Jacobi preconditioner} is
\[
M=D.
\]
It is extremely cheap to apply, fully parallel, and often removes poor scaling between coordinates.
Its main limitation is that it only uses diagonal information. 
\paragraph{Example: Jacobi for the 1D discrete Laplacian.}
Consider now the tridiagonal matrix arising from the standard finite-difference
discretisation of
\[
-u''(x)=f(x),\qquad x\in(0,1),\qquad u(0)=u(1)=0.
\]
With $m$ interior grid points and mesh size $h=1/(m+1)$, the linear system is
\[
A_h u = f,
\qquad
A_h=\frac1{h^2}\operatorname{tridiag}(-1,2,-1)\in\mathbb{R}^{m\times m}.
\]
The diagonal of $A_h$ is constant:
\(
D=\operatorname{diag}(A_h)=\frac{2}{h^2}I.
\)
Therefore the Jacobi-preconditioned operator is simply
\[
D^{-1}A_h=\frac{h^2}{2}A_h
=\frac12\,\operatorname{tridiag}(-1,2,-1).
\]
Jacobi therefore rescales every eigenvalue by the same constant factor, and cannot improve the
condition number $\kappa(A_h)\sim h^{-2}$, which grows quadratically as the mesh is refined.

\begin{takeaway}
Jacobi preconditioning is effective when the main difficulty is simple coordinate scaling,
but it cannot improve problems such as the 1D discrete Laplacian where the diagonal is already uniform.
In PDEs, the hard part is often not local scaling but global error propagation, which motivates Schwarz methods, coarse spaces, and multigrid. A more powerful algebraic alternative is the incomplete factorisation $A\approx\tilde L\tilde U$, in which fill-in arising outside a prescribed sparsity pattern is simply discarded \cite{meijerinkIterativeSolutionMethod1977}.
\end{takeaway}

\subsection{One framework for the whole chapter}\label{sec:subspacecorrection}

Before turning to better preconditioners it is worth noticing that many of them (and all those discussed in this lecture) are instances of a single construction. Let
$R_i\in\mathbb{R}^{n_i\times n}$ be a \emph{restriction} onto a subspace of the unknowns, let
$A_i=R_iAR_i^\top$ be the corresponding local matrix, and set
\begin{equation}\label{eq:subspace-correction}
M^{-1}=\sum_{i} R_i^\top A_i^{-1}R_i .
\end{equation}
Each term solves the problem \emph{exactly} on one subspace and ignores the rest; the preconditioner is
their sum. Applying $M^{-1}$ costs one small solve per subspace, and the solves are independent of one
another which is what makes the construction parallel.

Everything in this chapter is a choice of the subspaces in \eqref{eq:subspace-correction}.

\begin{center}
\renewcommand{\arraystretch}{1.25}
\begin{tabular}{|l|l|l|}
\hline
\textbf{Subspaces $R_i$} & \textbf{Preconditioner} & \textbf{What it cannot reach}\\
\hline
single unknowns & Jacobi (\S\ref{sec:basic}) & anything non-local\\
overlapping subdomains & additive Schwarz (\S\ref{sec:schwarz}) & global, slowly varying error\\
subdomains $+$ a coarse space & two-level, multigrid (\S\ref{sec:multigrid}) & ---\\
\hline
\end{tabular}
\end{center}

\noindent The table above is the plan of the chapter, and each row fixes the disadvantages of
the one above it. For example, a pure \emph{local} preconditioner handles local error well and global error not at
all (for example Jacobi in its simple and block variants or Schwarz preconditioners). To fix this, one should add a \emph{global} subspace to catch what the local solves miss.

\section{Schwarz Domain Decomposition}\label{sec:schwarz}

\noindent The algorithms of this section, their convergence theory, and their parallel implementation are treated at length in \cite{doleanIntroductionDomainDecomposition2015,toselliDomainDecompositionMethods2005}.

In the language of \eqref{eq:subspace-correction}, this section makes the first non-trivial choice of
subspaces: instead of correcting each unknown independently as Jacobi does, we take the $R_i$ to restrict
to \emph{overlapping subdomains} of the mesh, so that each local solve of $A_i^{-1}$ inverts the operator
exactly on a whole piece of the domain. The resulting preconditioner is \emph{additive Schwarz}, and the
question we must answer about it is the one posed in Section~\ref{sec:basic}: how does its convergence behave?

It is easiest to answer that question by first studying the corresponding \emph{stationary} iteration,
where the local solves are applied one after another. That iteration is of independent historical
interest and its convergence factor exposes the two
parameters that will govern the preconditioner as well: the overlap width, and the number of subdomains.

\paragraph{Model problem: 1D Laplace equation.}
We consider
\[
-u''(x)=f(x),\qquad x\in(0,1),\qquad u(0)=u(1)=0.
\]
We split the interval into two overlapping subdomains
\(\Omega_1=(0,\beta),\, \Omega_2=(\alpha,1),
\, 0<\alpha<\beta<1.
\)
The overlap is the interval $(\alpha,\beta)$.
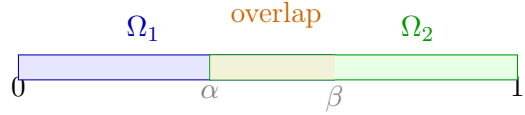
\begin{figure}[h!]
\centering
\begin{tikzpicture}[scale=1.1]
  \draw[thick] (0,0)--(6,0);
  \node[below] at (0,0){$0$};
  \node[below] at (6,0){$1$};
  \draw[dashed, gray] (2.3,0.1)--(2.3,-0.1) node[below]{$\alpha$};
  \draw[dashed, gray] (3.8,0.1)--(3.8,-0.1) node[below]{$\beta$};
  \filldraw[fill=blue!10,draw=blue!70!black] (0,-0.15) rectangle (3.8,0.15);
  \filldraw[fill=green!10,draw=green!60!black] (2.3,-0.15) rectangle (6,0.15);
  \fill[orange!20,opacity=0.5] (2.3,-0.15) rectangle (3.8,0.15);
  \node[above,blue!70!black] at (1.5,0.2){$\Omega_1$};
  \node[above,green!60!black] at (4.8,0.2){$\Omega_2$};
  \node[above,orange!80!black] at (3.1,0.35){overlap};
\end{tikzpicture}
\caption{1D domain split into two overlapping subdomains.}
\end{figure}
The alternating Schwarz iteration is:
\[
\begin{array}{cc}
\left\{\begin{array}{rcl}
-(u_1^{(n)})''&=&f \text{ in }(0,\beta),\\
u_1^{(n)}(0)&=&0,\\
u_1^{(n)}(\beta)&=&u_2^{(n-1)}(\beta),
\end{array}\right.
&
\left\{\begin{array}{rcl}
-(u_2^{(n)})''&=&f \text{ in }(\alpha,1),\\
u_2^{(n)}(1)&=&0,\\
u_2^{(n)}(\alpha)&=&u_1^{(n)}(\alpha).
\end{array}\right.
\end{array}
\]
Each subproblem is solved exactly, using boundary data from the latest iterate on
the neighbouring subdomain.

\begin{tcolorbox}[colback=yellow!5!white,colframe=yellow!50!black]
\begin{proposition}[Contraction of alternating Schwarz in 1D]
\label{prop:schwarz-1d}
Let $u$ be the exact solution of the Poisson BVP and let $u_1^{(n)},u_2^{(n)}$ be the alternating Schwarz iterates on
\(
\Omega_1=(0,\beta)\) and \(\Omega_2=(\alpha,1)\). Define the errors
\(
e_1^{(n)}=u|_{\Omega_1}-u_1^{(n)},
\,
e_2^{(n)}=u|_{\Omega_2}-u_2^{(n)}.
\)
Then
\[
-(e_1^{(n)})''=0 \quad \text{in }(0,\beta),
\qquad
-(e_2^{(n)})''=0 \quad \text{in }(\alpha,1),
\]
and the interface errors satisfy the \emph{exact} relation
\[
e_2^{(n)}(\beta)
=
\rho \, e_2^{(n-1)}(\beta),
\qquad
\rho=\frac{\alpha(1-\beta)}{\beta(1-\alpha)}\in(0,1).
\]
Hence $e_2^{(n)}(\beta)=\rho^{\,n}\,e_2^{(0)}(\beta)$: the alternating Schwarz method converges
geometrically, and the convergence improves as the overlap $\beta-\alpha$ increases.
\end{proposition}
\end{tcolorbox}

\begin{proof}
The iterates and the exact solution satisfy the same equation with different boundary data, so each error
is harmonic on its subdomain: $-(e_i^{(n)})''=0$, with $e_1^{(n)}(0)=0$, $e_1^{(n)}(\beta)=e_2^{(n-1)}(\beta)$
and $e_2^{(n)}(1)=0$, $e_2^{(n)}(\alpha)=e_1^{(n)}(\alpha)$. In one dimension, a harmonic function is
affine, so each error is determined by its two boundary values:
\[
e_1^{(n)}(x)=\frac{x}{\beta}\,e_2^{(n-1)}(\beta),
\qquad
e_2^{(n)}(x)=\frac{1-x}{1-\alpha}\,e_1^{(n)}(\alpha).
\]
Evaluating the first at $x=\alpha$ and substituting into the second at $x=\beta$ gives
$e_2^{(n)}(\beta)=\rho\,e_2^{(n-1)}(\beta)$ with $\rho=\frac{\alpha(1-\beta)}{\beta(1-\alpha)}$. No
absolute values and no inequality are needed: in one dimension the error is affine, hence determined
exactly by its two interface values. Finally $0<\alpha<\beta<1$ gives $\rho>0$, and
$\alpha(1-\beta)<\beta(1-\alpha)$ reduces to $\alpha<\beta$, so $\rho<1$.
\end{proof}

\begin{remark}
The formula makes the role of the overlap explicit: as $\beta-\alpha\to0$ we get $\rho\to1$ and
convergence stalls, while a wider overlap decreases $\rho$. Overlap is the
mechanism by which information is shared by the neighbouring subdomains.
\end{remark}

\begin{figure}[h!]
\centering
\begin{tikzpicture}
\begin{semilogyaxis}[
width=0.75\textwidth,
height=6cm,
xlabel={Schwarz iteration $n$},
ylabel={interface error magnitude},
xmin=0, xmax=10,
ymin=1e-4, ymax=1,
grid=both,
legend style={at={(0.5,1.02)},anchor=south,draw=none,legend columns=3,font=\small},
ticklabel style={font=\small},
label style={font=\small},
]
\addplot+[mark=o] coordinates {
(0,1) (1,0.70) (2,0.49) (3,0.343) (4,0.2401) (5,0.1681) (6,0.1176)
};
\addlegendentry{small overlap ($\rho\approx 0.7$)}

\addplot+[mark=square*] coordinates {
(0,1) (1,0.40) (2,0.16) (3,0.064) (4,0.0256) (5,0.01024) (6,0.0041)
};
\addlegendentry{medium overlap ($\rho\approx 0.4$)}

\addplot+[mark=triangle*] coordinates {
(0,1) (1,0.20) (2,0.04) (3,0.008) (4,0.0016) (5,0.00032)
};
\addlegendentry{large overlap ($\rho\approx 0.2$)}
\end{semilogyaxis}
\end{tikzpicture}
\caption{Alternating Schwarz converges geometrically in 1D. The contraction factor improves when the overlap increases.}
\end{figure}
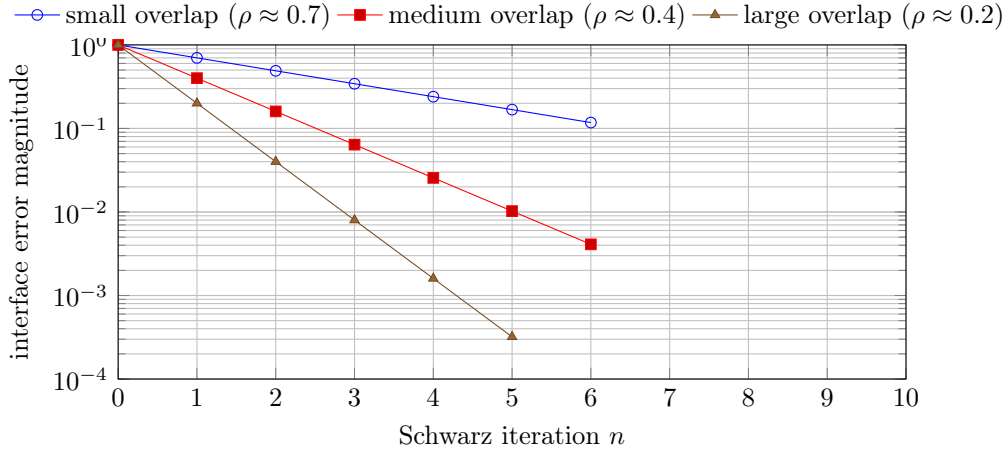

\paragraph{Discrete formulation.}
Discretising with $m=9$ interior points and taking $\beta=7h$, $\alpha=4h$ gives local matrices
$A_1\in\mathbb{R}^{6\times6}$ and $A_2\in\mathbb{R}^{5\times5}$, each the 1D Laplacian on its own
subdomain. The iteration becomes
\[
A_1u_1^{(n)}=f_1+\tilde B_{12}u_2^{(n-1)},\qquad
A_2u_2^{(n)}=f_2+\tilde B_{21}u_1^{(n)},
\]
where $\tilde B_{12},\tilde B_{21}$ inject the interface values from one subdomain into the other
(Dirichlet transmission). Note that only $A_1$ and $A_2$ are factorised: the global matrix is never
formed. This makes the method efficient at a larger scale.

\noindent The discrete contraction factor inherits the continuous one: the iteration converges
geometrically at a rate set by the overlap and, crucially, \emph{independently of $h$} such that refining the
mesh does not slow it down, because the overlap is a fixed fraction of the domain.

\begin{takeaway}
The alternating Schwarz iteration quickly damps high-frequency errors and converges independently of the mesh size.  
It provides a prototype for overlapping domain decomposition preconditioners used in large-scale PDE solvers.
\end{takeaway}

\paragraph{From subdomain solves to a preconditioner.}
The discrete Schwarz viewpoint becomes clearer if we start from the global residual
\(
r^{(n)} = f - A u^{(n)}.
\)  
Instead of trying to correct all unknowns at once, we restrict this residual to each subdomain and solve a \emph{local correction problem} there. Let $R_i$ be the restriction from the global vector to the unknowns belonging to
subdomain $\Omega_i$, and let
\[
A_i = R_i A R_i^T
\]
be the local stiffness matrix on that subdomain.
The local correction $\delta u_i$ is defined by
\[
A_i \,\delta u_i = R_i r^{(n)},
\qquad\text{so that}\qquad
\delta u_i = A_i^{-1} R_i r^{(n)}.
\]
This means that on each subdomain we solve exactly the local error equation driven by the
restricted residual. The correction is then prolongated back to the global space by $R_i^T$.
If the subdomains overlap, several local corrections contribute to the same global node,
so we blend them with diagonal weights $D_i$ satisfying the partition-of-unity condition
\[
\sum_i R_i^T D_i R_i = I.
\]
This leads to the global correction
\[
u^{(n+1)}
=
u^{(n)} + \sum_i R_i^T D_i \delta u_i
=
u^{(n)} + \sum_i R_i^T D_i A_i^{-1} R_i\, r^{(n)}.
\]

\paragraph{Additive Schwarz preconditioner.}
The operator
\[
M_{\mathrm{AS}}^{-1}
=
\sum_i R_i^T D_i A_i^{-1} R_i
\]
is called the \emph{additive Schwarz preconditioner}.
With this notation, one step of preconditioned Richardson iteration reads
\[
u^{(n+1)} = u^{(n)} + M_{\mathrm{AS}}^{-1}(f-Au^{(n)}).
\]
So the preconditioner acts as an approximate inverse of $A$ obtained by:
\begin{enumerate}[leftmargin=*]
  \item restricting the residual to each subdomain,
  \item solving local problems there,
  \item prolongating the local corrections,
  \item blending the overlap contributions.
\end{enumerate}

\begin{remark}
This additive preconditioner is closely related to the alternating Schwarz iteration,
but it is not exactly the same algorithm.
The alternating method is \emph{multiplicative}: the second subdomain uses the updated
information coming from the first.
By contrast, additive Schwarz computes all local corrections from the same global residual
and sums them together.
The additive form is especially convenient as a preconditioner inside Krylov methods .
\end{remark}

\begin{keyidea}
Schwarz preconditioning improves on Jacobi by replacing \emph{pointwise} updates with
\emph{subdomain solves}.
Instead of using only diagonal information, it inverts the operator locally on each
subdomain.
This captures nearest-neighbour couplings much better and yields a significantly stronger
preconditioner for discretised PDEs.
\end{keyidea}

\paragraph{Example of a decomposition.}
To make the link with the preconditioned formulation explicit, consider again $m=9$ interior nodes
and the decomposition
\[
\Omega_1 = (0,7h), \qquad \Omega_2 = (4h,1),
\]
which gives an overlap of three grid points ($x_5,x_6,x_7$).  
The global unknown vector is
\(
u = (u_1,u_2,\ldots,u_9)^\top,
\)
and the local vectors on each subdomain are
\[
u_1 = (u_1,\ldots,u_6)^\top, \qquad
u_2 = (u_5,\ldots,u_9)^\top.
\]

The restriction matrices $R_1\in\mathbb{R}^{6\times9}$ and $R_2\in\mathbb{R}^{5\times9}$ are the
corresponding rows of the identity: $R_1=[\,I_6\ \ 0\,]$ picks out $u_1,\dots,u_6$, and
$R_2=[\,0\ \ I_5\,]$ picks out $u_5,\dots,u_9$. They are never stored as matrices since applying $R_i$ is
an indexing operation.

Each restriction operator extracts local unknowns,
and its transpose \(R_i^T\) prolongates the local corrections into the global vector.
In the overlap region (nodes $x_5,x_6,x_7$), contributions from both subdomains are combined
using a discrete \emph{partition of unity}:
\[
D_1 = \mathrm{diag}(1,1,1,1,\tfrac{1}{2},\tfrac{1}{2},\tfrac{1}{2},0,0), \qquad
D_2 = \mathrm{diag}(0,0,0,0,\tfrac{1}{2},\tfrac{1}{2},\tfrac{1}{2},1,1).
\]
These satisfy
\[
R_1^T D_1 R_1 + R_2^T D_2 R_2 = I,
\]
ensuring that the overlapping corrections are weighted consistently and the global operator
remains a proper approximation of $A^{-1}$. The additive Schwarz preconditioner then reads
\(
M^{-1} = R_1^T D_1 A_1^{-1} R_1 + R_2^T D_2 A_2^{-1} R_2,
\)
which applies local solves and blends them smoothly through the partition weights.

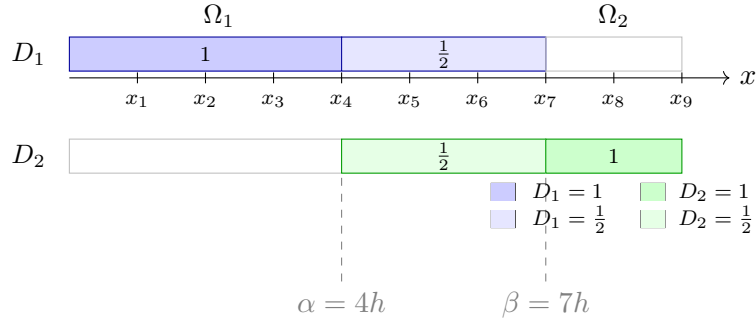
\begin{figure}[t]
\centering
\begin{tikzpicture}[x=0.9cm,y=0.9cm]

  \draw[->] (0,0) -- (9.7,0) node[right] {$x$};
  \foreach \k in {1,...,9}{
    \draw (\k,0.08) -- (\k,-0.08);
    \node[below] at (\k,-0.1) {\scriptsize $x_{\k}$};
  }

  \draw[dashed,gray] (4,-1.0) -- (4,-3) node[below,gray] {$\alpha=4h$};
  \draw[dashed,gray] (7,-1.0) -- (7,-3) node[below,gray] {$\beta=7h$};

  \node[above] at (2.2,0.6) {\small $\Omega_1$};
  \node[above] at (8.0,0.6) {\small $\Omega_2$};

  \node[left] at (-0.2,  0.35) {\small $D_1$};
  \node[left] at (-0.2, -1.15) {\small $D_2$};

  \fill[blue!20] (0,0.1) rectangle (4,0.6);
  \draw[blue!60!black] (0,0.1) rectangle (4,0.6);
  \node at (2,0.35) {\scriptsize 1};
  \fill[blue!10] (4,0.1) rectangle (7,0.6);
  \draw[blue!60!black] (4,0.1) rectangle (7,0.6);
  \node at (5.5,0.35) {\scriptsize $\tfrac{1}{2}$};
  \draw[gray!50] (7,0.1) rectangle (9,0.6);

  \draw[gray!50] (0,-1.4) rectangle (4,-0.9);
  \fill[green!10] (4,-1.4) rectangle (7,-0.9);
  \draw[green!60!black] (4,-1.4) rectangle (7,-0.9);
  \node at (5.5,-1.15) {\scriptsize $\tfrac{1}{2}$};
  \fill[green!20] (7,-1.4) rectangle (9,-0.9);
  \draw[green!60!black] (7,-1.4) rectangle (9,-0.9);
  \node at (8,-1.15) {\scriptsize 1};

  \begin{scope}[shift={(6.2,-1.8)}]
    \draw (0,0) rectangle +(0.35,0.25); \fill[blue!20] (0,0) rectangle +(0.35,0.25);
    \node[right] at (0.45,0.12) {\scriptsize $D_1=1$};
    \draw (0,-0.4) rectangle +(0.35,0.25); \fill[blue!10] (0,-0.4) rectangle +(0.35,0.25);
    \node[right] at (0.45,-0.28) {\scriptsize $D_1=\tfrac{1}{2}$};
    \draw (2.2,0) rectangle +(0.35,0.25); \fill[green!20] (2.2,0) rectangle +(0.35,0.25);
    \node[right] at (2.6,0.12) {\scriptsize $D_2=1$};
    \draw (2.2,-0.4) rectangle +(0.35,0.25); \fill[green!10] (2.2,-0.4) rectangle +(0.35,0.25);
    \node[right] at (2.6,-0.28) {\scriptsize $D_2=\tfrac{1}{2}$};
  \end{scope}

\end{tikzpicture}
\caption{Partition of unity in the overlap: $D_1$ (top) and $D_2$ (bottom). On the overlap, $D_1=D_2=\tfrac{1}{2}$ and outside the overlap they sum to one: $R_1^T D_1 R_1 + R_2^T D_2 R_2 = I$.}
\end{figure}

\section{Two-Level Preconditioning and Multigrid Ideas}\label{sec:multigrid}

In order to understand the mechanism of error reduction or smoothing and then the need of a coarse space correction let us consider the discretisation of the 1D Poisson problem. The eigenvectors of \(A_h\) are the discrete sine modes
\[
v_j(\ell)=\sin\!\Bigl(\frac{j\pi \ell}{m+1}\Bigr),
\qquad \ell=1,\dots,m,\quad j=1,\dots,m,
\]
with eigenvalues
\[
\lambda_j(A_h)=\frac{4}{h^2}\sin^2\!\Bigl(\frac{j\pi}{2(m+1)}\Bigr).
\]
Small values of \(j\) correspond to smooth, slowly varying modes, whereas large
values of \(j\) correspond to oscillatory modes.

\begin{figure}[h!]
\centering
\begin{tikzpicture}[scale=1.1]
\begin{scope}
\draw[gray!50] (0,0)--(5.2,0);
\foreach \i in {0,...,10}{\fill (0.5*\i,0) circle (1.1pt);}
\draw[very thick,TUblue]
plot[smooth] coordinates
{(0,0.05) (0.5,0.25) (1,0.45) (1.5,0.62) (2,0.74) (2.5,0.78)
 (3,0.74) (3.5,0.62) (4,0.45) (4.5,0.25) (5,0.05)};
\node at (2.5,-0.65) {\small smooth mode};
\node at (2.5,-1.05) {\small $\sin\!\bigl(\frac{\pi j}{m+1}\bigr)$, small \(j\)};
\end{scope}

\begin{scope}[xshift=7cm]
\draw[gray!50] (0,0)--(5.2,0);
\foreach \i in {0,...,10}{\fill (0.5*\i,0) circle (1.1pt);}
\draw[very thick,red!70!black]
plot[smooth] coordinates
{(0,0.05) (0.5,0.75) (1,-0.72) (1.5,0.78) (2,-0.76) (2.5,0.75)
 (3,-0.74) (3.5,0.78) (4,-0.72) (4.5,0.75) (5,0.05)};
\node at (2.5,-0.9) {\small oscillatory mode};
\node at (2.5,-1.3) {\small $\sin\!\bigl(\frac{\pi j}{m+1}\bigr)$, large \(j\)};
\end{scope}
\end{tikzpicture}
\caption{Low-frequency and high-frequency error modes on a 1D grid.}
\end{figure}
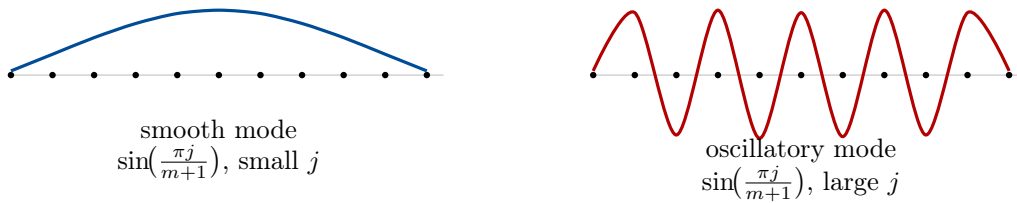

Consider the Jacobi iteration
\[
u^{(k+1)} = u^{(k)} + D^{-1}(f-A_hu^{(k)}),
\qquad D=\frac{2}{h^2}I.
\]
If \(e^{(k)}=u-u^{(k)}\) denotes the error, then
\[
e^{(k+1)} = T_J e^{(k)},
\qquad
T_J = I-D^{-1}A_h.
\]

\begin{tcolorbox}[colback=yellow!5!white,colframe=yellow!50!black]
\begin{proposition}[Mode-by-mode damping of Jacobi]
\label{prop:jacobi-mode-damping}
For the 1D Laplacian, each sine mode \(v_j\) is an eigenvector of the Jacobi
error propagator \(T_J\), with eigenvalue
\[
\mu_j
=
1-\frac{\lambda_j(A_h)}{2/h^2}
=
1-2\sin^2\!\Bigl(\frac{j\pi}{2(m+1)}\Bigr)
=
\cos\!\Bigl(\frac{j\pi}{m+1}\Bigr).
\]
Hence, if
\[
e^{(0)}=\sum_{j=1}^m c_j v_j,
\]
then after \(k\) Jacobi steps
\[
e^{(k)}=\sum_{j=1}^m c_j\,\mu_j^k\,v_j.
\]
\end{proposition}
\end{tcolorbox}

\begin{proof}
Since the \(v_j\) are eigenvectors of \(A_h\),
\[
A_h v_j=\lambda_j(A_h)v_j.
\]
Therefore
\[
T_J v_j
=
\Bigl(I-D^{-1}A_h\Bigr)v_j
=
\Bigl(1-\frac{\lambda_j(A_h)}{2/h^2}\Bigr)v_j.
\]
Using
\[
\lambda_j(A_h)=\frac{4}{h^2}\sin^2\!\Bigl(\frac{j\pi}{2(m+1)}\Bigr),
\]
we obtain
\[
\mu_j
=
1-2\sin^2\!\Bigl(\frac{j\pi}{2(m+1)}\Bigr).
\]
The trigonometric identity \(1-2\sin^2(\theta)=\cos(2\theta)\) gives
\[
\mu_j=\cos\!\Bigl(\frac{j\pi}{m+1}\Bigr).
\]
The expansion of \(e^{(k)}\) follows by diagonalising the iteration on the basis
\(\{v_j\}_{j=1}^m\).
\end{proof}

\begin{remark}
This formula shows explicitly why Jacobi behaves differently on different scales:
\[
j\ \text{small}
\quad\Longrightarrow\quad
|\mu_j|\approx 1
\qquad\text{(slow damping)},
\]
whereas
\[
j\ \text{large}
\quad\Longrightarrow\quad
|\mu_j|\ \text{is much smaller}
\qquad\text{(fast damping)}.
\]
So Jacobi is not a good solver by itself, but it is a good smoother.
\end{remark}

\subsection{Weighted Jacobi and the smoothing factor}\label{sec:smoothing}

A common variant is \emph{weighted Jacobi}:
\[
u^{(k+1)} = u^{(k)} + \omega D^{-1}(f-A_hu^{(k)}).
\]
Its error propagator is
\[
T_{J,\omega}=I-\omega D^{-1}A_h.
\]
For the sine mode \(v_j\),
\[
T_{J,\omega}v_j=\mu_j^{(\omega)}v_j,
\qquad
\mu_j^{(\omega)}
=
1-2\omega\sin^2\!\Bigl(\frac{j\pi}{2(m+1)}\Bigr).
\]
To study smoothing, we focus on the upper half of the spectrum,
\(
\frac{m+1}{2}\le j\le m,
\)
which corresponds to oscillatory modes. For these indices,
\[
\sin^2\!\Bigl(\frac{j\pi}{2(m+1)}\Bigr)\in\Bigl[\frac12,1\Bigr],
\]
hence
\[
\mu_j^{(\omega)}\in[\,1-2\omega,\;1-\omega\,].
\]
So the worst-case damping factor on the high-frequency range is
\[
\rho_{\mathrm{high}}(\omega)
=
\max\bigl(|1-2\omega|,\ |1-\omega|\bigr).
\]

\begin{tcolorbox}[colback=yellow!5!white,colframe=yellow!50!black]
\begin{proposition}[Optimal weighted Jacobi parameter for smoothing]
\label{prop:weighted-jacobi-smoothing}
For the 1D Laplacian, the value of \(\omega\) that minimizes the worst-case high-frequency damping factor is
\(
\omega=\frac23.
\)
For this choice,
\(
\rho_{\mathrm{high}}=\frac13.
\)
Thus every high-frequency mode is reduced by at least a factor \(1/3\) per step.
\end{proposition}
\end{tcolorbox}

\begin{proof}
We minimize
\[
\rho_{\mathrm{high}}(\omega)=\max(|1-2\omega|,|1-\omega|).
\]
The optimum is obtained by balancing the two endpoint magnitudes:
\[
|1-2\omega|=|1-\omega|.
\]
For \(\omega\in(1/2,1)\), this becomes
\(
2\omega-1=1-\omega,
\)
hence
\(
\omega=\frac23.
\)
Then
\(
\rho_{\mathrm{high}}=\frac13.
\)
\end{proof}

\begin{takeaway}
Weighted Jacobi with \(\omega=2/3\) is not a scalable solver, but it is an effective
smoother: it damps oscillatory components uniformly while leaving smooth components
for the coarse grid to remove.
\end{takeaway}
After several relaxation steps, the remaining error typically has the form
\[
e = e_{\mathrm{smooth}} + e_{\mathrm{osc}},
\qquad
\|e_{\mathrm{osc}}\|\ll \|e_{\mathrm{smooth}}\|.
\]
That is, the oscillatory part has already been reduced, while the smooth part remains.




The key multigrid observation is that a smooth error on the fine grid may look much
more oscillatory on a coarser grid.
This makes it easier to represent and eliminate there.

\begin{keyidea}
A good solver for an elliptic problem must act on two scales at once:
\[
\text{local smoothing}\qquad+\qquad\text{global coarse correction}.
\]
\end{keyidea}

\subsection{Coarse-grid correction}\label{sec:coarse}

Section~\ref{sec:smoothing} left us with an iterate whose error is \emph{smooth}: the smoother has
removed the oscillatory components and stalls on the rest. What remains satisfies the error equation
\[
Ae=r,\qquad r=f-Au,\qquad e=u^\ast-u ,
\]
and solving that on the fine grid is of course just the original problem again. The way out rests on a
single observation: \emph{a smooth vector is cheap to represent}. It is determined, to good accuracy, by
its values at half of the grid points, the rest being recovered by interpolation so the error
equation may be solved in a space of half the dimension.

\begin{figure}[htbp]
\centering
\begin{tikzpicture}[>=stealth,scale=0.98]
  \draw[gray!45] (0,-1.75) -- (9,-1.75);
  \draw[gray!35] (0,0) -- (9,0);
  \foreach \i in {0,...,32}{\fill[gray!55] (0.28125*\i,-1.75) circle (0.9pt);}
  \foreach \i in {0,...,16}{\fill[TUblue!70!black] (0.5625*\i,-1.75) circle (1.9pt);}
  \node[font=\scriptsize,anchor=north] at (1.3,-1.95) {fine nodes};
  \node[font=\scriptsize,anchor=north,TUblue!70!black] at (4.6,-1.95) {coarse nodes ($H=2h$)};
  \draw[gray!70,thin] plot coordinates {(0.000,0.000)(0.281,0.124)(0.562,-0.343)(0.844,-0.271)(1.125,-1.600)(1.406,1.179)(1.688,0.750)(1.969,-0.213)(2.250,0.507)(2.531,0.184)(2.812,-0.363)(3.094,0.641)(3.375,-0.204)(3.656,-0.215)(3.938,-0.519)(4.219,0.298)(4.500,-0.065)(4.781,0.357)(5.062,-0.398)(5.344,0.083)(5.625,-0.585)(5.906,0.551)(6.188,0.123)(6.469,0.217)(6.750,0.269)(7.031,-0.662)(7.312,0.513)(7.594,1.348)(7.875,-1.074)(8.156,-1.133)(8.438,-0.986)(8.719,0.551)(9.000,0.000)};
  \draw[TUblue,very thick] plot coordinates {(0.000,0.000)(0.281,-0.533)(0.562,-0.895)(0.844,-0.913)(1.125,-0.481)(1.406,0.195)(1.688,0.792)(1.969,1.032)(2.250,0.953)(2.531,0.699)(2.812,0.422)(3.094,0.127)(3.375,-0.143)(3.656,-0.334)(3.938,-0.342)(4.219,-0.233)(4.500,-0.107)(4.781,-0.108)(5.062,-0.162)(5.344,-0.168)(5.625,-0.004)(5.906,0.235)(6.188,0.441)(6.469,0.529)(6.750,0.563)(7.031,0.523)(7.312,0.281)(7.594,-0.303)(7.875,-1.058)(8.156,-1.600)(8.438,-1.567)(8.719,-0.955)(9.000,0.000)};
  \draw[orange!85!black,thick,densely dashed] plot coordinates {(0.000,0.000)(0.562,-0.895)(1.125,-0.481)(1.688,0.792)(2.250,0.953)(2.812,0.422)(3.375,-0.143)(3.938,-0.342)(4.500,-0.107)(5.062,-0.162)(5.625,-0.004)(6.188,0.441)(6.750,0.563)(7.312,0.281)(7.875,-1.058)(8.438,-1.567)(9.000,0.000)};
  \foreach \xx/\yy in {0.000/0.000,0.562/-0.895,1.125/-0.481,1.688/0.792,2.250/0.953,2.812/0.422,3.375/-0.143,3.938/-0.342,4.500/-0.107,5.062/-0.162,5.625/-0.004,6.188/0.441,6.750/0.563,7.312/0.281,7.875/-1.058,8.438/-1.567,9.000/0.000}{\fill[orange!85!black] (\xx,\yy) circle (1.6pt);}
  \draw[gray!70,thin] (0.15,2.35) -- (0.75,2.35);
  \node[gray!80,font=\scriptsize,anchor=west] at (0.85,2.35) {initial error};
  \draw[TUblue,very thick] (3.05,2.35) -- (3.65,2.35);
  \node[TUblue,font=\scriptsize,anchor=west] at (3.75,2.35) {after 5 smoothing sweeps};
  \draw[orange!85!black,thick,densely dashed] (0.15,1.95) -- (0.75,1.95);
  \node[orange!85!black,font=\scriptsize,anchor=west] at (0.85,1.95) {piecewise-linear interpolant of its coarse-node values};
\end{tikzpicture}
\caption{Why a coarse space can carry the correction. The thin grey curve is a random initial error on a
grid with $m=31$; the thick blue curve is the same error after five weighted-Jacobi sweeps
($\omega=\tfrac23$); the dashed orange curve is the piecewise-linear interpolant through its values at
the $15$ coarse nodes only.}
\label{fig:coarse-grid}
\end{figure}
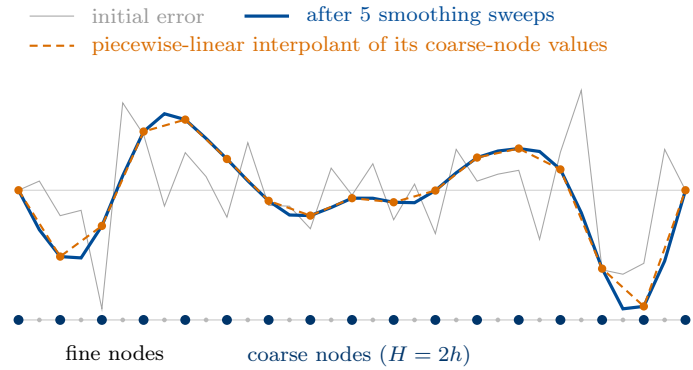

Formally, let $P:\mathbb{R}^{n_c}\to\mathbb{R}^{n}$ be the \emph{prolongation} (linear interpolation from
the coarse grid), let $R:\mathbb{R}^{n}\to\mathbb{R}^{n_c}$ be the \emph{restriction}, and define the
coarse operator by
\[
A_H=RAP .
\]
Taking $R\propto P^\top$ makes $A_H$ the Galerkin projection of $A$ onto the coarse space; for the 1D
Laplacian with full weighting $R=\tfrac12P^\top$ it coincides exactly with the direct discretisation on
the coarse grid (Problem~4). The correction then reads: restrict the residual, solve $A_He_H=Rr$, and
prolongate, $u\leftarrow u+Pe_H$.

\begin{tcolorbox}[colback=softred,colframe=red!50!black]
\textbf{Coarse-grid correction cannot be an iterative method on its own.} Its error propagation
operator is
\[
T=I-PA_H^{-1}RA ,
\]
and a short computation using $A_H=RAP$ shows that $T^2=T$ (Problem~5). A projection has eigenvalues $0$
and $1$ only, so $\rho(T)=1$: whatever lies in $\ker T$ is annihilated in one step, and whatever does not
is left \emph{completely untouched}, forever. Coarse-grid correction removes the part of the error that
the coarse space can see, and by construction can do nothing about the rest.
\end{tcolorbox}

\noindent This is the precise sense in which the two ingredients are complementary, and it is why neither
is optional. The smoother cannot touch smooth error; the coarse correction cannot touch anything outside
its range. Each ingredient is useless on its own; the next subsection shows that together they are optimal.

\subsection{From coarse correction to the two-grid method}

\begin{figure}[htbp]
\centering
\begin{tikzpicture}[>=stealth,scale=1.2]
\foreach \y/\lab in {2.4/{$\Omega_h$}, 1.2/{$\Omega_{2h}$}, 0/{$\Omega_{4h}$}}{
  \draw[gray!35,dashed] (-0.3,\y) -- (8.6,\y);
  \node[font=\scriptsize,anchor=east] at (-0.4,\y) {\lab};}
\begin{scope}
  \node[font=\small,anchor=south] at (1.55,3.05) {two-grid};
  \fill[TUblue] (0.5,2.4) circle (2pt); \fill[TUblue] (2.6,2.4) circle (2pt);
  \fill[TUblue] (1.55,1.2) circle (2pt);
  \draw[->,thick,TUblue] (0.5,2.4) -- (1.55,1.2);
  \draw[->,thick,TUblue] (1.55,1.2) -- (2.6,2.4);
  \node[font=\scriptsize,anchor=east] at (0.95,1.85) {$R$};
  \node[font=\scriptsize,anchor=west] at (2.15,1.85) {$P$};
  \node[font=\scriptsize,anchor=south] at (0.5,2.5) {$S^{\nu_1}$};
  \node[font=\scriptsize,anchor=south] at (2.6,2.5) {$S^{\nu_2}$};
  \node[font=\scriptsize,anchor=north] at (1.55,1.1) {exact solve};
\end{scope}
\begin{scope}[xshift=4.6cm]
  \node[font=\small,anchor=south] at (2.0,3.05) {V-cycle};
  \fill[TUblue] (0,2.4) circle (2pt); \fill[TUblue] (1,1.2) circle (2pt);
  \fill[TUblue] (2,0) circle (2pt);
  \fill[TUblue] (3,1.2) circle (2pt); \fill[TUblue] (4,2.4) circle (2pt);
  \draw[->,thick,TUblue] (0,2.4) -- (1,1.2);
  \draw[->,thick,TUblue] (1,1.2) -- (2,0);
  \draw[->,thick,TUblue] (2,0) -- (3,1.2);
  \draw[->,thick,TUblue] (3,1.2) -- (4,2.4);
  \node[font=\scriptsize,anchor=north] at (2,-0.1) {coarsest solve};
  \node[font=\scriptsize,anchor=south] at (0,2.5) {smooth};
  \node[font=\scriptsize,anchor=south] at (4,2.5) {smooth};
\end{scope}
\end{tikzpicture}
\caption{The two-grid method (left) and the V-cycle obtained by applying it recursively (right). Each
descending edge is a restriction of the residual, each ascending edge a prolongation of the correction,
and relaxation is applied before and after every transfer. Recursion is what makes the method optimal: the coarse solve is
itself replaced by a two-grid step until the problem is small enough to solve directly, so one cycle
costs $\mathcal{O}(n)$ work and reduces the error by a factor independent of $h$. Used as a
preconditioner for CG, this turns the $\mathcal{O}(h^{-1})$ iteration count of Chapter~4 into a constant.}
\label{fig:v-cycle}
\end{figure}
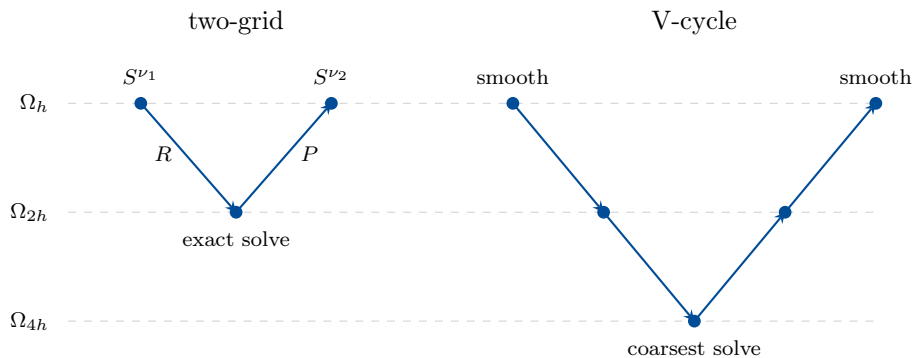

The two ingredients are now in place, and Section~\ref{sec:coarse} showed that each fails precisely where
the other succeeds. Combining them in the obvious order gives the \emph{two-grid method}: smooth to
remove what the coarse space cannot see, correct on the coarse grid to remove what the smoother cannot
touch, then smooth again to clean up the interpolation error introduced by $P$. Starting from an
approximation $u$:

\begin{enumerate}[leftmargin=*,itemsep=1pt]
\item \textbf{Pre-smoothing:} $u\leftarrow S^{\nu_1}(u)$, with $S$ weighted Jacobi or Gauss--Seidel.
\item \textbf{Residual restriction:} $r_H=R(f-Au)$.
\item \textbf{Coarse solve:} $A_He_H=r_H$.
\item \textbf{Coarse correction:} $u\leftarrow u+Pe_H$.
\item \textbf{Post-smoothing:} $u\leftarrow S^{\nu_2}(u)$.
\end{enumerate}

\paragraph{From two grids to multigrid.} Step~3 still requires an \emph{exact} solve with $n/2$
unknowns ($n/4$ in 2D), which may itself be far too large to solve directly. But the observation that
motivated
the coarse grid applies again one level down: once the coarse problem has itself been smoothed, its error
is smooth \emph{relative to $H$} and is representable on a grid of spacing $2H$. So replace the exact
coarse solve by one two-grid step on $\Omega_{2h}$, its coarse solve by a two-grid step on
$\Omega_{4h}$, and so on until the grid is small enough to solve directly.

That recursion is \emph{multigrid}, and the grids visited on the way down and back up form the V-cycle of
Figure~\ref{fig:v-cycle}. Two properties follow. The work is $\mathcal{O}(n)$, since the grids shrink
geometrically; and the error reduction per cycle is independent of $h$, being governed by the smoothing
factor of Section~\ref{sec:smoothing} e.g. $\tfrac13$ per step for weighted Jacobi in 1D rather than
by $\kappa(A)$.

\begin{takeaway}
Smoother and coarse correction are exactly complementary
\cite{briggsMultigridTutorial2000,hackbuschMultiGridMethods1985}: the first removes oscillatory error and
cannot touch smooth error, the second removes what lies in the coarse space and cannot touch anything
else. Recursion turns the pair into an $\mathcal{O}(n)$ solver whose contraction factor does not depend
on $h$.
\end{takeaway}

\paragraph{Multigrid as a preconditioner}

In practice, multigrid is often used not as a standalone solver, but as a preconditioner
inside a Krylov method. One multigrid cycle is then viewed as an approximate inverse:
\[
M_{\mathrm{MG}}^{-1}\approx A^{-1}.
\]
The linear system is solved in preconditioned form, for example by CG or GMRES:
\[
M_{\mathrm{MG}}^{-1}Ax = M_{\mathrm{MG}}^{-1}b.
\]

This fits naturally into the general framework of preconditioning:
multigrid constructs a strong approximate inverse by combining local relaxation
with a global coarse-space correction.

\begin{remark}
For elliptic problems, multigrid-preconditioned Krylov methods often have iteration
counts that remain essentially bounded under mesh refinement.
This is the sense in which multigrid is called \emph{mesh-independent} or even
\emph{optimal}: the number of iterations no longer grows significantly as the grid is refined.
\end{remark}

\newpage
\section{Exercises}
\paragraph{Problem 1 (Convergence of a parallel Schwarz method in 1D, past exam).}
Consider the one-dimensional Laplace problem
\[
-\,u''(x)=f(x),\qquad x\in(0,1),\qquad u(0)=u(1)=0,
\]
and, as in Proposition~\ref{prop:schwarz-1d}, the overlapping decomposition
$\Omega=\Omega_1\cup\Omega_2$ with
\[
\Omega_1=(0,\beta),\qquad \Omega_2=(\alpha,1),\qquad 0<\alpha<\beta<1,\quad \delta=\beta-\alpha>0 .
\]
Proposition~\ref{prop:schwarz-1d} treats the \emph{alternating} (multiplicative) iteration, in which
$\Omega_2$ uses the value just computed on $\Omega_1$. Here we study the \emph{parallel} (additive)
iteration, in which both subdomains use the previous iterate:
\[
\begin{cases}
-\,\big(u_1^{(n+1)}\big)''=f & \text{in }(0,\beta),\\
u_1^{(n+1)}(0)=0,\quad u_1^{(n+1)}(\beta)=u_2^{(n)}(\beta),
\end{cases}
\qquad
\begin{cases}
-\,\big(u_2^{(n+1)}\big)''=f & \text{in }(\alpha,1),\\
u_2^{(n+1)}(\alpha)=u_1^{(n)}(\alpha),\quad u_2^{(n+1)}(1)=0,
\end{cases}
\]
where $u_i^{(n)}$ is the $n$th iterate on $\Omega_i$.

\begin{enumerate}[label=(\alph*)]
\item Let $e_i^{(n)}:=u|_{\Omega_i}-u_i^{(n)}$ denote the errors on each subdomain (where $u$ is the
exact solution). Write the boundary value problems (BVPs) satisfied by $e_1^{(n+1)}$ and $e_2^{(n+1)}$.

\item Solve these BVPs explicitly (they are affine functions on each subdomain).

\item Let
\[
\eta_n:=e_2^{(n)}(\beta),\qquad \zeta_n:=e_1^{(n)}(\alpha).
\]
Show that the interface errors satisfy the two coupled recursions
\[
\zeta_{n+1}=\frac{\alpha}{\beta}\,\eta_n,\qquad
\eta_{n+1}=\frac{1-\beta}{\,1-\alpha\,}\,\zeta_n,
\]
and hence that every second iterate is contracted by the \emph{same} factor as in
Proposition~\ref{prop:schwarz-1d}:
\[
\frac{e_2^{(n)}(\beta)}{e_2^{(n-2)}(\beta)}
=\frac{e_1^{(n)}(\alpha)}{e_1^{(n-2)}(\alpha)}
=\rho,\qquad
\rho:=\frac{\alpha}{\beta}\cdot\frac{1-\beta}{\,1-\alpha\,}.
\]

\item Prove that $0<\rho<1$ whenever $\alpha<\beta$ (i.e.\ when the overlap $\delta$ is positive), and
conclude that the parallel Schwarz method converges. Comment on the price paid for parallelism: how
many iterations of the parallel method correspond to one iteration of the alternating method?
\end{enumerate}
\notebookpointer{\nbpath{week7/script\_week7\_2.ipynb}}
\paragraph{Problem 2 (Spectral equivalence $\Rightarrow$ PCG convergence).}
Let $A,B$ be SPD and suppose there exist $c_1,c_2>0$ such that
\[
c_1\,x^\top Bx \;\le\; x^\top A x \;\le\; c_2\,x^\top Bx \qquad \forall x\in\mathbb{R}^n.
\]
(a) Show that $\operatorname{spec}(B^{-1/2}AB^{-1/2})\subset[c_1,c_2]$, hence $\kappa(B^{-1}A)\le c_2/c_1$. \\
(b) Let $\widetilde A=B^{-1/2}AB^{-1/2}$ and $\widetilde e_k=B^{1/2}e_k$, where $e_k=u_k-u_\ast$ is the
error of PCG in the original variables. Show that $\|\widetilde e_k\|_{\widetilde A}=\|e_k\|_A$, and
conclude that PCG applied to $Au=b$ with preconditioner $B$ satisfies
\[
\frac{\|e_k\|_{A}}{\|e_0\|_{A}}
\;\le\;
2\!\left(\frac{\sqrt{\kappa(B^{-1}A)}-1}{\sqrt{\kappa(B^{-1}A)}+1}\right)^{\!k}
\;\le\;
2\!\left(\frac{\sqrt{c_2/c_1}-1}{\sqrt{c_2/c_1}+1}\right)^{\!k}.
\]
In words: spectral equivalence with constants independent of $h$ gives an iteration count independent
of $h$.

\paragraph{Problem 3 (Preconditioned conjugate gradient).}
Let $A,M$ be SPD. Apply CG to $\tilde A\tilde x=\tilde b$ with $\tilde A=M^{-1/2}AM^{-1/2}$ and $\tilde b=M^{-1/2}b$. 
Show that the mapped iterates $x^{(k)}=M^{-1/2}\tilde x^{(k)}$ satisfy the iteration in the proposition; in particular
$p^{(0)}=M^{-1}r^{(0)}$, and
\[
\alpha_k=\frac{\langle r^{(k)},z^{(k)}\rangle}{\langle p^{(k)},Ap^{(k)}\rangle},
\qquad
\beta_{k+1}=\frac{\langle r^{(k+1)},z^{(k+1)}\rangle}{\langle r^{(k)},z^{(k)}\rangle},
\quad z^{(k)}=M^{-1}r^{(k)}.
\]

\notebookpointer{\nbpath{week7/script\_week7\_1.ipynb}}
\paragraph{Problem 4 (Transfer operators and the Galerkin coarse matrix).}
Let 
$$A_h=\frac{1}{h^2}\operatorname{tridiag}(-1,2,-1)\in\mathbb{R}^{m\times m}$$ 
with $m=2m_H+1$,
$h=1/(m+1)$, and let $P\in\mathbb{R}^{m\times m_H}$ be linear interpolation from the coarse grid
$H=2h$: $(Pe_H)_{2j}=\tfrac12(e_H)_j$, $(Pe_H)_{2j+1}=(e_H)_j$, $(Pe_H)_{2j+2}=\tfrac12(e_H)_j$.
\begin{enumerate}
\item[(a)] Write $P$ explicitly for $m=7$, $m_H=3$, and show that the \emph{full weighting} restriction
$R=\tfrac12P^\top$ applies the stencil $\tfrac14(1,2,1)$.
\item[(b)] Compute $A_H=RA_hP$ for $m=7$ and verify that it coincides \emph{exactly} with the direct
discretisation $\frac{1}{H^2}\operatorname{tridiag}(-1,2,-1)$ on the coarse grid.
\item[(c)] What do you get from $P^\top A_hP$ instead, and why must the scalings of $R$ and $P$ be
consistent?
\end{enumerate}

\paragraph{Problem 5 (Coarse-grid correction alone cannot converge).}
With $A_h,P,R$ and $A_H=RA_hP$ as in Problem~4, the coarse-grid correction
$u\leftarrow u+PA_H^{-1}R(f-A_hu)$ has error propagation operator $T=I-PA_H^{-1}RA_h$.
\begin{enumerate}
\item[(a)] Show that $T^2=T$ (use $A_H=RA_hP$ to simplify $PA_H^{-1}RA_hP$), and deduce that every
eigenvalue of $T$ is $0$ or $1$, that $\rho(T)=1$ when $m_H<m$, and that $\operatorname{rank}T=m-m_H$.
\item[(b)] Conclude that coarse-grid correction alone is not convergent, identify which errors lie in
$\ker T$ (compare Figure~\ref{fig:coarse-grid}), and explain in one sentence why a smoother is therefore
mandatory rather than merely helpful.
\end{enumerate}

\paragraph{Problem 6 (The two-grid contraction factor).}
Let $S=I-\omega D^{-1}A_h$ be the weighted Jacobi smoother, $D=\operatorname{diag}(A_h)$, and $T$ as in
Problem~5; the two-grid operator with $\nu$ pre-smoothing steps is $M_{\rm TG}=TS^{\nu}$.
\begin{enumerate}
\item[(a)] From the eigenvalues $\mu_j(\omega)=1-2\omega\sin^2\!\bigl(\tfrac{j\pi}{2(m+1)}\bigr)$ of $S$,
explain why $\omega=1$ is a poor \emph{smoother} even though it minimises $\rho(S)$.
\item[(b)] Take $\omega=\tfrac23$ and $m=7$. Compute $\rho(M_{\rm TG})$ numerically for $\nu=1$ and
$\nu=2$, and compare with the smoothing factor $\tfrac13$ of
Proposition~\ref{prop:weighted-jacobi-smoothing}. State the relationship you observe between
$\rho(M_{\rm TG})$ and $\nu$.
\item[(c)] Repeat for $m=15$ and $m=31$: is the contraction factor independent of $h$? Contrast with the
$\mathcal{O}(h^{-2})$ iteration count of unpreconditioned Jacobi in Chapter~3.
\item[(d)] A V-cycle visits grids of size $m,m/2,m/4,\dots$; show its total cost is $\mathcal{O}(m)$, and
explain why this together with (c) makes multigrid an \emph{optimal} solver.
\end{enumerate}

\notebookpointer{\nbpath{week7/script\_week7\_3.ipynb}}

\backmatter

\chapter*{About these notes}
\addcontentsline{toc}{chapter}{About these notes}
\markboth{About these notes}{About these notes}

\paragraph{Acknowledgements.}
We thank Jan Willem van Beek, instructor for the course, for his contribution to the teaching and to
these notes. We are also grateful to the students of previous editions, whose questions have shaped
much of the exposition, and to colleagues who commented on earlier drafts.

\paragraph{Code.}
The numerical illustrations in these notes are reproducible. The accompanying Python code is
available at
\begin{center}
\url{https://github.com/vicdolean/scicomp_examples}
\end{center}
which is a fork of the repository maintained by Jan Willem van Beek,
\url{https://github.com/jwfvanbeek/scicomp_examples}. The two are kept synchronised, and any
divergence between them is stated in the repositories themselves. Readers are encouraged to run and
modify the code: changing a parameter, a matrix or a stopping criterion and watching what happens to
the convergence history is by far the fastest way to develop a feeling for when these methods work
and when they do not.

\cleardoublepage
\addcontentsline{toc}{chapter}{\bibname}
\bibliographystyle{plain}
\bibliography{scicomp}

\end{document}